\documentclass[10pt,reqno]{amsart}

\usepackage{mathtools}
\usepackage{amssymb,amsmath}
\usepackage{wrapfig}
\usepackage{amscd}
\usepackage{amsrefs}
\usepackage{graphicx}
\usepackage{hyperref}

\def\F{\mathbb{F}}
\def\A{{\mathcal A}}

\def\text#1{\hbox{#1}}
\def\Z{{\mathbb Z}}
\def\R{{\mathbb R}}
\def\N{{\mathbb N}}

\def\C{{\mathbb C}}
\def\Q{{\mathbb Q}}

\def\A{{\mathbb A}}
\def\O{{\mathcal O}}

\def\p{\partial}

\def\Re{{\mathrm {Re}}}
\def\Im{{\mathrm {Im}}}

\def\supp{{\rm supp}\,}

\def\Vol{{\rm E}}

\def\N{{\rm N}}

\def\logp{{\rm log}^\dagger}

    \DeclareFontFamily{U}{wncy}{}
    \DeclareFontShape{U}{wncy}{m}{n}{<->wncyr10}{}
    \DeclareSymbolFont{mcy}{U}{wncy}{m}{n}
    \DeclareMathSymbol{\Sh}{\mathord}{mcy}{"58}

\newtheorem{theorem}{Theorem}[section]
\newtheorem{definition}[theorem]{Definition}
\newtheorem{lemma}[theorem]{Lemma}
\newtheorem{corollary}[theorem]{Corollary}
\newtheorem{proposition}[theorem]{Proposition}

\newtheorem{remark}[theorem]{Remark}

\newtheorem{thmx}{Theorem}

\date{\today}
\theoremstyle{definition}
\numberwithin{equation}{section}
\usepackage{pgfplots}
\DeclareMathOperator*{\esssup}{ess\,sup}

\newcommand\blfootnote[1]{%
  \begingroup
  \renewcommand\thefootnote{}\footnote{#1}%
  \addtocounter{footnote}{-1}%
  \endgroup
}
\def \Leb {{\rm Leb}}
\def\N{{\mathbb N}}

\title{On the optimal rate of equidistribution in number fields} 
\date{}
\author{Miko\l aj Fraczyk}
\address{Alfr\'ed R\'enyi Institute of Mathematics,  Re\'altanoda utca 13-15, H-1053, Budapest, Hungary}
\address{Department of Mathematics, University of Chicago\\
5734 S University Ave, Chicago, IL 60637}
\email{mfraczyk@math.uchicago.edu}
 \author{Anna Szumowicz}
  \address{Institut de Mathématiques de Jussieu –- Paris Rive Gauche, 4, place Jussieu, 75252 Paris, France}
  \address{Caltech, The Division of Physics, Mathematics and Astronomy
1200 E California Blvd, Pasadena CA 91125}
\email{anna.szumowicz@caltech.edu}
\begin{document}
\maketitle
\blfootnote{\today}
\begin{abstract}
Let $k$ be a number field and let $\O_k$ be its ring of integers. We study how well can finite subsets of $\O_k$ equidistribute modulo powers of prime ideals, for all prime ideals at the same time. Our main result states that the optimal rate of equidistribution in $\O_k$ predicted by the local constraints cannot be achieved unless $k=\Q$. We deduce that $\Q$ is the only number field where the ring of integers $\mathcal O_k$ admits a simultaneous $\frak p$-ordering, answering a question of Bhargava. Along the way we establish a non-trivial upper bound on the number of solutions $x\in \O_k$ of the inequality $|N_{k/\Q}(x(a-x))|\leq X^2$ where $X$ is a positive real parameter and $a\in\mathcal O_k$ is of norm at least $e^{-B}X$ for a fixed real number $B$. The latter can be translated as an upper bound on the average number of solutions of certain unit equations in $\mathcal O_k$.
\end{abstract} 

\section{Introduction}
\subsection{Optimal rate of equidistribution in number fields.} Let $k$ be a number field with the ring of integers $\O_k$. We say that a sequence of finite subsets $E_n\subset \O_k$ \textbf{equidistributes locally} if for every prime ideal $\frak p$  
the sequence of probability measures $$\mu_n:=\frac{1}{|E_n|}\sum_{x\in E_n}\delta_x$$ seen as measures on the $\frak p$-adic completion $k_\frak p$ converge weakly-* to the normalized Haar measure on the ring of integers $\O_{k_\frak p}$.
One can measure the rate of equidistribution in $\O_{k_\frak p}$ by looking at the additive $\frak p$-adic valuation of the product of differences $\prod_{s\neq s'\in E_n}(s-s')$. 
Using the pigeonhole principle, one can show that $$v_\mathfrak{p} \left(\prod_{s\neq s' \in E_n}(s-s') \right) \geq \sum_{m=1}^{|E_n|-1}\sum_{i=1}^\infty \left\lfloor \frac{m}{q^{i}}\right\rfloor,$$ where $q$ is the size of the residue field of $k_\frak p$. When the equality is achieved for each $n$, we say that $(E_n)_{n \in \mathbb{N}}$ \textbf{equidistributes optimally}
 in $\O_{k_\frak p}$.
 It happens, for example, when $E_n$ are sets of the first $n$ elements of a sequence $(a_i)_{i\in\N}$ which is a $\frak p$-ordering (Definition \ref{d-pOrdering}). The $\frak p$-orderings were introduced by Manjul Bhargava in \cite{Bhargava1} in order to generalize the notion of the factorial to any Dedekind domain (or even subsets of Dedekind domains) and to extend the  classical results of P\'olya on integer valued polynomials in $\Q[t]$ to arbitrary Dedekind domains \cite[Theorem 14]{Bhargava1}. While it is easy to see that for a fixed finite set $P$ of primes $\frak p$ one can find a sequence $E_n$ that equidistributes optimally in $\O_{k_\frak p}$ for all $\frak p\in P$, it is not clear if there exists a sequence of sets  $E_n$ that equidistributes optimally for all primes $\frak p$ at the same time. It is certainly possible in $\Z$ because we can take $E_n=\left\{1,2,\ldots, n\right\}$. As the main result of this paper we prove that $k=\Q$ is the \textbf{only} number field for which $\O_k$ enjoys this property. As a corollary we answer the question of Bhargava \cite[Question 3]{Bhargava1} for rings of integers in number fields. Bhargava asked which Dedekind domains admit simultaneous $\frak p$-orderings. Our main result implies that $\Z$ is the only ring of integers in a number field where this is possible.
\subsection{$\mathfrak{p}$-orderings and equidistribution} 
\label{sectionporderings}

Let $A$ be a ring and let $I$ be an ideal of $A$.  We say that a finite subset $S\subseteq A$ is \textbf{almost uniformly distributed modulo $I$} if for any $a,b\in A$ we have 
\begin{equation}\label{e-ConditionAUE}\mid \left\lbrace s\in S\mid\: s-a\in I\right\rbrace\mid -\mid\left\lbrace s\in S\mid \: s-b \in I\right\rbrace\mid\in\left\lbrace -1, 0,1\right\rbrace .
\end{equation}
If $A/I$ is finite the condition (\ref{e-ConditionAUE}) is equivalent to the following  
\begin{equation}\label{e-ConditionAUE2}
\left||\left\lbrace s\in S\mid\: s-a\in I\right\rbrace| -\frac{|S|}{| A/I|}\right| <1.
\end{equation}
\begin{definition}\label{d-noptimal}
We call a finite subset $S\subseteq \O_k$ \textbf{$n$-optimal} if $|S|=n+1$ and $S$ is almost uniformly equidistributed modulo every power $\frak p^l, l\geq 1$ for every prime ideal $\frak p$ of $\O_k$.  
\end{definition}  The sequences of $n$-optimal sets are precisely the ones that equidistribute optimally in $\O_{k_\frak p}$ for all primes $\frak p$ at the same time.  
The main result of this paper determines the numbers fields $k$ where the rings of integers $\O_k$ admits arbitrarily large $n$-optimal sets.
\begin{thmx}\label{mt-noptimal}
Let $k$ be a number field different than $\Q$. Then there is a natural number $n_0$ such that there are no $n$-optimal sets in $\O_k$ for $n\geq n_0$.
\end{thmx} In particular, unless $k=\Q$ there are no infinite sequences of finite subsets that equidistribute optimally modulo all prime powers. Motivation for considering the $n$-optimal subsets comes from the theory of integer valued polynomials and from the study of $\mathfrak{p}$-orderings. We recall the definition of a $\frak p$-ordering in a subset of $\O_k$, following \cite{Bhargava1}. 
\begin{definition}\label{d-pOrdering}
Let $W \subseteq \O_k$ and let $\frak p$ be a non-zero proper prime ideal. A sequence $(a_i)_{i\in\N}\subseteq W$ is a \textbf{$\frak p$-ordering in $W$} if for every $n\in\N$ we have 
$$v_W(\frak p,n):=v_\frak p\left(\prod_{i=0}^{n-1}(a_i-a_n)\right)=\min_{s\in W} v_\frak p\left(\prod_{i=0}^{n-1}(a_i-s)\right),$$ where $v_\frak p$ stands for the additive $\frak p$-adic valuation on $k$. The value $v_W(\frak p,n)$ does not depend on the choice of a $\frak p$-ordering (\cite{Bhargava1}).
\end{definition} 

Bhargava defines the \textbf{generalized factorial} as the ideal $n!_W=\prod_{\frak p}\frak p^{v_W(\frak p,n)}$ where $\frak p$ runs over primes in $\O_k$. A sequence $(a_i)_{i\in\N}\subseteq W$ is called a \textbf{simultaneous $\frak p$-ordering} in $W$ if it is a $\frak p$-ordering in $W$ for every prime ideal $\frak p$. Simultaneous $\frak p$-orderings are also called Newton sequences \cite{CahenChabert1,CahenChabert2}.  A sequence $(a_i)_{i\in \N}\subseteq\O_k$ is a simultaneous $\frak p$-ordering in $\O_k$ if and only if the set $\left\{a_0,a_1,\ldots,a_n\right\}$ is $n$-optimal for every $n\in\N$ (see \cite[Proposition 2.6]{BFS2017}) . In \cite{Bhargava1,BhargavaAMM} Bhargava asks what are the subsets $W\subseteq\O_k$ (or more general Dedekind domains) admitting simultaneous $\frak p$-orderings and in particular for which $k$ the ring $\O_k$ admits a simultaneous $\frak p$-ordering. The last question was addressed by Melanie Wood in \cite{MWood} where she proved that there are no simultaneous $\frak p$-orderings in $\O_k$ if $k$ is imaginary quadratic. This result was extended in \cite[Theorem 16]{AdamCahen} to real quadratic number fields $\Q(\sqrt{d})$ except for possibly finitely many exceptions.
Existence of a simultaneous $\frak p$-ordering implies that there are $n$-optimal sets in $\O_k$ for all $n$. As a corollary of  Theorem \ref{mt-noptimal} we get:
\begin{corollary}\label{mc-pordering}
$\Q$ is the unique number field whose ring of integers admits a simultaneous $\frak p$-ordering.
\end{corollary} This answers \cite[Question 3]{Bhargava1} for rings of integers in number fields. We remark that our method is completely different from the approach of Wood or Adam and Cahen.
Having an upper bound on $n$ such that there exists an $n$-optimal set is \textit{a priori} stronger than non-existence of simultaneous $\frak p$-orderings because not every $n$-optimal set can be ordered into an initial fragment of a simultaneous $\frak p$-ordering. We do not know any example of a Dedekind domain that has $n$-optimal sets for every $n \in \mathbb{N}$ but has no simultaneous $\frak p$-orderings. We remark that the ring $\F_q[t]$ admits a simultaneous $\frak p$-ordering \cite[p. 125]{Bhargava1}. It would be interesting to know which finite extensions $F$ of $\F_q(t)$  have the property that $\O_F$ admits a simultaneous $\frak p$-ordering. 
\subsection{Test sets for integer valued polynomials.}
The notions of $\frak p$-ordering and $n$-optimal set are connected to the theory of integer valued polynomials. Let $P\in k[X]$ be a polynomial. We say that $P$ is \textbf{integer valued} if $P(\O_k)\subseteq \O_k$. 

A subset $E\subset \O_k$ is an \textbf{$n$-universal} set if the following  holds. A polynomial $P\in k[X]$ of degree at most $n$ is integer valued (on $\O_k$) if and only if $P(E)\subseteq \O_k$. It is easy to prove, using Lagrange interpolation, that $|S|\geq n+1$ for any $n$-universal set $S$. It was shown in \cite{PV2013, BFS2017} that if $|S|=n+1$, then $S$ is $n$-universal if and only if it is almost uniformly distributed modulo all powers of all prime ideals. In our terminology the latter is equivalent to $S$ being $n$-optimal.  It is proved in \cite{BFS2017} that in a Dedekind domain for every $n\in\N$ there exists an $n$-universal set of size $n+2$. It is therefore natural to ask whether there are $n$-universal sets of cardinality $n+1$ (i.e. $n$-optimal sets). For $k$ quadratic imaginary it was proven in \cite{BFS2017} that there is an $n_0=n_0(k)$ such that there are no $n$-optimal sets in $\O_k$ with $n>n_0$. This generalizes the analogous result for $k=\Q(\sqrt{-1})$ from \cite{PV2013}. For general quadratic number fields Cahen and Chabert  \cite[Theorem 3.12]{CahenChabert2} proved that there are no $2$-optimal sets, except possibly in $\Q(\sqrt{d}), d=-3,-1,2,3,5$ and $d\equiv 1\mod 8$.  
From our main result and \cite[Theorem 4.1.]{BFS2017} we deduce the following.
\begin{corollary}\label{mc-nuniversal}
Let $k\neq \Q$ be a number field. Then for $n\in\N$ sufficiently large the minimal cardinality of an $n$-universal set in $\O_k$ is $n+2$.
\end{corollary}

\subsection{ Average number of solutions of a unit equation}
One of our key technical ingredients in the proof of Theorem \ref{mt-noptimal} is the following bound, which can be interpreted as a bound on the average number of solutions of the unit equation (see \cite{Zannier2014}). To shorten notation we will write $\|x\|=|N_{k/\Q}(x)|$ for $x\in k$, where $N_{k/\Q}$ is the field extension norm.
\begin{thmx}\label{mt-Count1} Let $k$ be a number field of degree $N$ with $d$ Archimedean places and let $B\in \R$. Put $\kappa=\frac{1}{3}$ if $N=1$ or $\kappa=\min\left\{\frac{1}{2N(N-1)},\frac{1}{4N-1}\right\}$ otherwise. There are constants $\Xi_1,\Xi_2,\Xi_3,\Xi_4$ dependent only on $k$ and $B$ such that for every $a\in \O_k$ and $0< X\leq\|a\|e^{B}$ we have 
$$|\left\{x\in \O_k|\: \|x(a-x))\|\leq X^2\right\}|\leq \Xi_1 X^{1+\kappa}\|a\|^{-\kappa} +\Xi_2(\log X)^{2d-2}+ \Xi_3\logp\logp\logp\log \|a\| +\Xi_4,$$
where $\logp t:=\log t$ if $t>1$ and $\logp t=0$ otherwise.
\end{thmx}
 The traditional form of the unit equation is 
$$\alpha_1\lambda_1+\alpha_2\lambda_2=1 \textrm{ where } \alpha_1,\alpha_2\in k^\times$$
and the indeterminates $\lambda_1,\lambda_2$ are units of $\O_k$. We may consider an equivalent form of the unit equation
\begin{equation}\label{e-UnitEq}
\alpha_1\lambda_1+\alpha_2\lambda_2=\alpha_3 \textrm{ where }\alpha_1,\alpha_2,\alpha_3\in \O_k \setminus \{ 0 \}.
\end{equation} It is clear that the number of solutions depends only on the class of $(\alpha_1,\alpha_2,\alpha_3)$ in the quotient of the projective space $\mathbb P^2(k)/{(\O_k^\times)^3}.$ Let $\nu(\alpha_1,\alpha_2,\alpha_3)$ be the number of solutions of (\ref{e-UnitEq}). It was known since Siegel \cite{Siegel21} that $\nu(\alpha_1,\alpha_2,\alpha_3)$  is finite and Evertse \cite{Evertse} found an upper bound independent of $\alpha_1,\alpha_2,\alpha_3$
$$\nu(\alpha_1,\alpha_2,\alpha_3)\leq 3\times 7^{N}.$$
In fact, Evertse, Gy\"ory, Stewart and Tijdeman \cite{EGST88} showed that except for finitely many points $[\alpha_1,\alpha_2,\alpha_3]\in \mathbb P^2(k)/(\O_k^\times)^{3}$ the equation (\ref{e-UnitEq}) has at most two solutions.  Theorem \ref{mt-Count1} gives a quantitative control on the "average" number of solutions of (\ref{e-UnitEq}) as $\alpha_1,\alpha_2\in \O_k/\O_k^\times, \|\alpha_1\alpha_2\|\leq X^2$ and $\|\alpha_3\|$ is fixed with $\|\alpha_3\|$ not much smaller than $X$.
\begin{thmx}\label{t-AvUnit}
Let $k$ be a number field of degree $N$ with $d$ Archimedean places, let $B\in \R$ and put $\kappa=\frac{1}{3}$ if $N=1$ or $\kappa=\min\left\{\frac{1}{2N(N-1)},\frac{1}{4N-1}\right\}$ otherwise. There exist constants $\Xi_1,\Xi_2,\Xi_3,\Xi_4$ dependent only on $k$ and $B$ such that for every $\alpha_3\in \O_k, 0< X\leq  \|\alpha_3\|e^{B}$ we have 
$$\sum_{\substack{\alpha_1,\alpha_2\in \O_k/\O_k^{\times}\\ \|\alpha_1\alpha_2\|\leq X^2 }}\nu(\alpha_1,\alpha_2,\alpha_3)\leq \Xi_1 X^{1+\kappa}\|\alpha_3\|^{-\kappa}+ \Xi_2 (\log X)^{2d-2} +\Xi_3\logp\logp\logp\log \|\alpha _{3}\|+\Xi_4.$$
\end{thmx}

The number of terms in the sum is of order $X^2\log X$ so Theorem \ref{t-AvUnit} shows that the average value of $\nu(\alpha_1,\alpha_2,\alpha_3)$ is 
\begin{equation*} 
O(X^{\kappa-1}\|\alpha_3\|^{-\kappa}(\log X)^{-1}+(\log X)^{2d-3}X^{-2}+(\logp\logp\logp\log\|\alpha_3\|)X^{-2}(\log X)^{-1}). 
\end{equation*} 
 Unless $\|\alpha_3\|\gg e^{e^{e^{X^2\log X}}}$, this improves on average on the pointwise bound of Evertse, Gy\"ory, Stewart and Tijdeman \cite{EGST88}.
\subsection{Energy minimizing shapes}
Let $V=\R^{r_1}\times \C^{r_2}$ and write $\|v\|=\prod_{i=1}^{r_1}|v_i|\prod_{i=r_1+1}^{r_2}|v_i|^2.$ We introduce a notion of energy for compactly supported measures on $V$ of bounded density (see Definition \ref{d-Energy}). For any compactly supported finite measure $\nu$ on $V$, absolutely continuous with respect to the Lebesgue measure and of bounded density, the energy is given by
$$I(\nu):=\int_V\int_V\log\|x-y\|d\nu(x)d\nu(y).$$ 
This is not too far from the energy considered in the potential theory, but the singularities of the kernel make it a bit more difficult to handle. The probability measures of density at most $1$ which minimize the energy among all such measures play a special role in the proof of Theorem \ref{mt-noptimal}. We prove that they always are of the following form:
\begin{thmx}\label{thm-Shape}
Let $\nu$ be a compactly supported probability measure with density at most $1$ which minimizes the energy among all such measures. Then, up to translation, there exists an open bounded set $U\subset V$  such that $\nu=\Leb|_U$ and $\partial U \cap \{v\in V| \|v\|>0\}$ is a $C^1$ submanifold of $\{v\in V| \|v\|>0\}$.
\end{thmx}
Theorem \ref{thm-Shape} follows from Proposition \ref{p-EnergyMinimizers}.
\subsection{Outline of the proof}
To prove Theorem  \ref{mt-noptimal} we argue by contradiction. We assume that there exists a sequence $\mathcal S_{n_{i}}$ of $n_{i}$-optimal subsets where $n_{i}$ tends to infinity. Let $V:= k\otimes _{\Q}\R\simeq \R^{r_{1}}\times \C^{r_{2}}$. First we show (Theorem \ref{t-OptimalFrame}) that for each $n_{i}$ there exists a cylinder (see Definition \ref{d-Cylinder}) $\mathcal{C}_{n_{i}}\subseteq V$ of volume $O(n_{i})$ containing $S_{n_{i}}$. This fact was implicit in the proofs of Theorem \ref{mt-noptimal} for $k=\Q (\sqrt{-1})$ in \cite{PV2013} and for $k$ quadratic imaginary in \cite{BFS2017}. The argument in \cite{PV2013, BFS2017} relied on a technique called "discrete collapsing" \footnote{In  \cite{BFS2017} it was called simply "collapsing". We add the adjective discrete to distinguish it from the collapsing for measures used in the present work.} which crucially uses the fact that the norm $N_{k/\Q}$ is convex for any quadratic imaginary number field $k$. Finding a way to prove Theorem \ref{t-OptimalFrame} for a general number field $k$ is one of the main contributions of the present work. A key number-theoretical input is provided by Proposition \ref{p-CountingMain} which counts the number of $x\in\O_k$ such that $|N_{k/\Q}(x(a-x))|\leq X^2$ for some $X>0$ and $a\in \O_k$ subject to the condition $|N_{k/\Q}(a)|\geq Xe^{-B}$ where $B$ is a fixed real number. The proof of Proposition \ref{p-CountingMain} combines a variant of Ikehara's Tauberian theorem, counting points of $\O_k$ in thin cylinders and the Baker--W\"ustholz's theorem on linear forms in logarithms.

Let $\Delta _{k}$ be the field discriminant of $k$. From Theorem \ref{t-OptimalFrame} we deduce (Corollary \ref{c-CompactFrame}) that there exists a compact set $\Omega$ and sequences $(s_{n_i})_{i\in\N},(t_{n_i})_{i\in\N}\subset V$ with $\|s_{n_i}\|=n_i|\Delta_k|^{1/2}$ such that the sets $s_{n_i}^{-1}(\mathcal S_{n_i}-t_{n_i})$ are all contained in $\Omega$. Thus, it makes sense to look at subsequential weak-* limits of measures 
$$\mu_{n_i}:=\frac{1}{n_i}\sum_{x\in \mathcal S_{n_i}} \delta_{s_{n_i}^{-1}(x-t_{n_i})}.$$ Any such limit will be called a \textbf{limit measure}. It is always a probability measure supported on $\Omega$, absolutely continuous with respect to the Lebesgue measure and of density\footnote{By density we mean the Radon--Nikodym derivative with respect to the Lebesgue measure.} at most one\footnote{The reason why we introduced the factor $|\Delta_k|^{1/2}$ in the formula $\|s_{n_i}\|=n_i|\Delta_k|^{1/2}$ is to ensure that the limits have density at most $1$. } (see Lemma \ref{l-DensityLimitMeasure}). By passing to a subsequence if necessary we can assume that $\mu_{n_i}$ converges to a limit measure $\mu$. The measure $\mu$ contains the information about the asymptotic geometry of the sets $\mathcal{S} _{n_i}$. Our strategy is to exploit the properties of $n$-optimal sets to show that no such limit measure can exist. 
The energy formula \footnote{In \cite{BFS2017} the energy ideal was called volume.} for $n$-optimal sets (see \cite[Corollary 5.2]{BFS2017}) allows us to prove (Proposition \ref{p-LimitEnergy}) that for any limit measure $\mu$ we have $$I(\mu)=-\frac{1}{2}\log|\Delta_k|-\frac{3}{2}-\gamma_k+\gamma_\Q,$$ where $\gamma_k,\gamma_\Q$ are the Euler--Kronecker constants of $k$ and $\Q$ respectively (c.f. \cite{ihara2006}). We know that the norm of the product of differences in an $n$-optimal set must be minimal among the norms of products of differences in all subsets of $\O_k$ of cardinality $n+1$  (\cite[Corollary 5.2]{BFS2017}). In other words the energy of an $n$-optimal set is minimal among energies (\cite[Definition 2.3]{BFS2017}) of subsets of $\O _k$ of cardinality $n+1$. This is used to show that $\mu$ minimizes the energy $I(\mu)$ among all compactly supported probability measures of density bounded by one (Lemma \ref{l-LowerEnergyBound}). The last property forces strong geometric constraints on $\mu$. In Proposition \ref{p-EnergyMinimizers} we show that any such energy-minimizing measure must be of the form $\mu(A)=\Leb(A\cap U)$ where $\Leb$ is the Lebesgue measure on $V$ and $U$ is an open set of measure $1$ whose boundary satisfies certain regularity conditions. This part of the argument uses the collapsing procedure for measures (Definition \ref{d-Collapsing}) which is analogous to the discrete collapsing from \cite{BFS2017} and similar to the Steiner symmetrization. We remark that if the field $k$ is not imaginary quadratic or $\Q$, then there is no reasonable discrete collapsing procedure for subsets of $\O_k$. The passage from subsets of $\O_k$ to measures on $V$ seems crucial for this part of the argument.

At this point we have established that $\mu_{n_i}$ converges weakly-* to $\mu=\Leb|_U$ for some open subset $U$ of $V$ with sufficiently regular boundary. This is equivalent to saying that  $\mathcal S_{n_i}=(\O_k\cap (s_{n_i}U+t_{n_i})) \sqcup R_{n_i}$ where the remainder satisfies $|R_{n_i}|=o(n_i)$.  The idea for the last part of the proof is to show that for $n_i$ sufficiently large, there is a prime ideal $\frak p_{n_i}$ such that $\mathcal S_{n_i}$ fails to be almost uniformly equidistributed modulo $\frak p_{n_i}^{l}$ for a certain $l \in \N$. This part is analogous to the proofs in \cite{PV2013, BFS2017} but substantially harder since we do not know the shape of $U$ explicitly. This problem is solved by relating the almost uniform equidistribution of $\mathcal{S} _{n_i}$ with the lattice point discrepancy of $U$ (see \ref{d-Discrepancy}). If $\mathcal{S} _{n_i}$ were almost uniformly equidistributed modulo powers of all prime ideals, then the maximal discrepancy of $U$ would be strictly less than $1$ (Lemma \ref{l-LimitDiscrepancy}). On the other hand, we show (Lemma \ref{l-NoAEq}) that once $\dim_\R V\geq 2$ and $\p U$ is smooth enough, the maximal discrepancy of $U$ must be strictly greater than $1$. This is the only place in the proof where we use the assumption that $k\neq \Q.$ We deduce that there must be a prime $\frak p_{n_i}$ such that $\mathcal S_{n_i}$ is not uniformly equidistributed modulo $\frak p_{n_i}^{l}$ for a certain $l \in \N$. This contradicts the fact that $\mathcal S_{n_i}$ is $n_i$-optimal and concludes the proof.
\subsection{Notation} 
Let $k$ be a number field of degree $N$ and let $\O_k$ be the ring of integers of $k$. Numbers $r_1,r_2$ are respectively the number of real and complex places of $k$. Put $d=r_1+r_2$. The field $k$ is fixed throughout the article and so are the numbers $N,r_1,r_2,d$. We identify  $V=k\otimes_\Q \R$ with $\R^{r_1}\times \C^{r_2}.$ For $v=(v_1,\ldots,v_d)\in V$ define $\|v\|=\prod_{i=1}^{r_1}|v_i|\prod_{i=r_1+1}^{r_1+r_2} |v_i|^2.$ Write $|v|_i$ for the absolute value of $i$-th coordinate of $v\in V$ and let $k_{\nu_i}$ be the completion of $k$ with respect to the valuation $|\cdot|_i$. We write $V^\times=\left\{v\in V|\: \|v\|\neq 0\right\}$ and $\O_k^\times$ for the unit group of $\O_k$. Let $N_{k/\Q}\colon k\to \Q$ be the norm of the extension $k/\Q$. The field $k$ embeds in $V$ and $\|x\|=|N_{k/\Q}(x)|$ for every $x\in k$.  We write $\Delta_k$ for the field discriminant of $k$. The $\mathfrak{p}$-adic valuations are normalized so that $|p|_{\mathfrak{p}}=p^{-[k_{\mathfrak{p}}:\mathbb{Q}_{\mathfrak{p}}]}$ for every $\mathfrak{p}|p$. We use standard big-O and little-o notation. Write $\log$ for the principal value of the natural logarithm and $\logp t=\log t$ if $t>1$ and $\logp t=0$ otherwise. 
We will write $\Leb$ for the Lebesgue measure on $V$, which is the product of Lebesgue measures on the real and complex factors. For any measure $\mu$ and measurable sets $E,F$ we will write $\mu|_E(F)=\mu(E\cap F).$ We write $B_{\R}(x,R)$ ($B_{\C}(x,R)$) for the ball of radius $R$ around $x\in \R$ ($x\in\C$). We will write $\mathcal M^1(V)$  (resp. $\mathcal P^1(V)$) for the set of finite (resp. probability) measures $\nu$ on $V$ which are absolutely continuous with respect to the Lebesgue measure and such that the Radon--Nikodym derivative satisfies $d\nu(v)/d\Leb(v)\leq 1$ for almost every $v\in V$. For any real number $t$ we will write $\lfloor t\rfloor=\max\left\{z\in\Z|\, z\leq t\right\}.$ If $G$ is a group we will write $\widehat G$ for the group of unitary characters of $G$. If $E$ is a set, then we will write $|E|$ for its cardinality. 
\subsection{Structure of the paper}
In Section \ref{s-CountingProblem} we develop estimates on the number of lattice points in $\O_k$ satisfying certain norm inequalities. The goal is to prove Proposition \ref{p-CountingMain} and deduce Theorem \ref{t-MainPropCount}. These inequalities control the number of points $x\in \O_k$ for which the norm $N_{k/\Q}((x-y)(x-z))$ is bounded where $y,z$ are two far-away points in $\O_k$. In \ref{s-Aramaki} we recall the Aramaki--Ikehara Tauberian theorem. In \ref{s-ProofPropC} we prove Proposition \ref{p-CountingMain} modulo some lemmas relying on diophantine approximation techniques. The subsection \ref{s-Baker} completes the missing part of the proof using Baker--W\"ustholz inequalities on linear forms in logarithms. In \ref{s-AverageNumber} we explain briefly why Theorem \ref{t-AvUnit} follows from Theorem \ref{t-MainPropCount}.

Section \ref{s-Geometry} is devoted to the proof of Theorem \ref{t-OptimalFrame}. This is where we prove that $n$-optimal sets can be suitably renormalized. In \ref{s-Gennoptimal} we gather some basic observations on norms of differences in an $n$-optimal set and in \ref{s-ProofOptimalFrame} couple them with the results of Section \ref{s-CountingProblem} to prove Theorem \ref{t-OptimalFrame}.

The Section \ref{s-Collapsing} is devoted to the collapsing procedure for measures on $V$. We mainly study its effect on the energy.

In Section \ref{s-LimitMeasures} we define and study the properties of limit measures. In \ref{s-Density} we prove that they have density bounded by $1$. In \ref{s-Energy} we prove that the limit measures must minimize the energy in the class of compactly supported probability measures of density at most $1$. Next in \ref{s-EnergyMinimizers} we study the geometric properties of such energy minimizing measures and describe them in Proposition \ref{p-EnergyMinimizers}. The proof of Proposition \ref{p-EnergyMinimizers} crucially uses the collapsing procedure.

In Section \ref{s-NonExistence} we show that limit measures cannot exist. In the first part \ref{s-Discrepancy} we recall the notion of lattice point discrepancy and relate it to limit measures. In \ref{s-Proof} we prove Theorem \ref{mt-noptimal}.

Finally in the Appendix we provide a proof of a folklore result on density of measures and likely well known variant of the prime number theorem for number fields.
\subsection*{Acknowledgment}
We would like to thank Jakub Byszewski for helpful discussions and previous collaboration that lead us to the subject of this paper. His comments were especially helpful for the proof of Lemma \ref{l-CollapsingPairsR}.  The second author is grateful to her supervisor Anne-Marie Aubert for comments on the earlier version of this paper. M.F. acknowledges support from ERC Consolidator Grant 648017. A.S acknowledges the support of UK Engineering and Physical Sciences Research Council.

\section{Counting problem}\label{s-CountingProblem}

The main result of this section is Proposition \ref{p-CountingMain}. It is a key ingredient in the proof of Theorem \ref{t-OptimalFrame} on the shape of $n$-optimal sets in $\O_k$.  As a corollary of Proposition \ref{p-CountingMain} we get the following counting result, introduced as Theorem \ref{mt-Count1} in the introduction, that may be of independent interest.

\begin{theorem}\label{t-MainPropCount}
Let $k$ be a number field of degree $N$ with $d$ Archimedean places, let $B\in\R$ and put $ \kappa =\frac{1}{3}$ if $N=1$ and $\kappa=\min\left\{\frac{1}{2N(N-1)},\frac{1}{4N-1}\right\}$ otherwise. There exist constants $\Xi_1,\Xi_2,\Xi_3,\Xi_4$ dependent only on $k$ and $B$ such that for every $X>0$ and $a\in \O_k$ such that $ \|a\|\geq Xe^{-B}$ we have 
$$|\left\{x\in \O_k|\: \|x(a-x)\|\leq X^2\right\}|\leq \Xi_1 X^{1+\kappa}\|a\|^{-\kappa} +\Xi_2(\log X)^{2d-2}+\Xi_3\logp\logp\logp\log\|a\|+\Xi_4.$$ 
\end{theorem}

We give a brief argument on how to deduce Theorem \ref{t-MainPropCount} after the statement of Proposition \ref{p-CountingMain}. To state Proposition \ref{p-CountingMain} we need to introduce some notations and auxiliary objects.
\begin{definition}\label{d-GoodFD} A \textbf{ good fundamental domain } of $\O_k^\times$ in $V^{\times}$ is a set $\mathcal F$ which is a finite union of convex closed cones in $V^{\times}$ such that $\mathcal F/\R^\times$ is compact in the projective space $\mathbb P(V)$,  $\rm{ int }\mathcal F\cap \lambda (\rm{ int } \mathcal F)=\emptyset$ for every $\lambda\in \O_k^\times, \lambda\neq 1$ and $V^{\times}=\bigcup_{\lambda\in \O_k^\times}\lambda\mathcal F.$ For technical reasons we will also require that the boundary $\p \mathcal F$ does not contain any points of $\O_k$. \end{definition} 
These conditions imply that $\mathcal F\cap \O_k$ is a set of all non-zero representatives of $\O_k/\O_k^\times.$ We have the following elementary observation.
\begin{lemma}\label{l-cone}Let $\mathcal F$ be a good fundamental domain of $\O_k^\times$ in $V^\times$. Then there exists a constant $C_0>0$ such that every $v\in\mathcal F$ satisfies $C_0^{-1} \|v\|^{1/N} \leq |v|_i \leq C_0\|v\|^{1/N}$ for $i=1,\ldots,d$.
\end{lemma}
We will often use this lemma in the latter part of the proof and sometimes we shall do so without additional comment. Let $W_k$ be the torsion subgroup of $\O_k^\times$. The quotient $\O_k^\times/W_k$ is a free abelian group of rank $d-1$, by Dirichlet's unit theorem. Let $\xi_1,\ldots,\xi_{d-1}$ be a basis of a maximal torsion free subgroup of $\O_k ^{\times }$. Every element $\lambda\in \O_k^\times$ is uniquely expressed as a product $\lambda=w \xi_1^{b_1}\ldots \xi_{d-1}^{b_{d-1}}$ with $w\in W_k$ and $b_i\in \Z$ for $i=1,\ldots,d-1.$ We define an $l^\infty$ norm on $\O_k^\times$ by $\|\lambda\|_\infty:=\max_{i=1,\ldots, d-1}|b_i|.$ From now on we fix the basis $\xi_1,\ldots, \xi_{d-1}$ as well as the associated norm $\|\cdot\|_\infty$.  
\begin{lemma}\label{l-NormIneq}
There exists a constant $\alpha>0$ such that $\max_{i=1,\ldots ,d}\log |\lambda|_i\geq \alpha \|\lambda\|_\infty$ for every $\lambda\in\O_k^\times$.
\end{lemma}
\begin{proof}
Put $\|\lambda\|_0:=\max_{i=1,\ldots ,d}\log |\lambda|_i$. Both $\|\cdot\|_0,\|\cdot\|_\infty$ extend uniquely to norms on $\O_k^\times \otimes_\Z \R\simeq \R^{d-1}$. Since any two norms on $\R^{d-1}$ are comparable, there exists a constant $\alpha>0$ such that $\alpha\|\lambda\|_\infty\leq \|\lambda\|_0\leq \alpha^{-1}\|\lambda\|_\infty$ for every $\lambda\in\O_k^\times$.
\end{proof}
By definition if we are given a good fundamental domain $\mathcal F$, then every element $y\in \O_k$ except $0$ decomposes uniquely as $y=x\lambda$ for $x\in\mathcal F\cap\O_k,\lambda\in \O_k^\times$. Let us fix a good fundamental domain $\mathcal F$. For $a\in \O_k, a\neq 0$ and $X>0$ we define the set 
$$S(a,X)=\left\{(x,\lambda)\in (\mathcal F\cap \O_k)\times \O_k^\times |\: \|x(a-x\lambda^{-1})\|\leq X^2, \|x\|\leq X\right\}.$$ 
\begin{proposition}\label{p-CountingMain}
Let $k$ be a number field of degree $N$ with $d$ Archimedean places, let $B\in\R$ and put $\kappa=\frac{1}{3}$ if $N=1$ and $\kappa=\min\left\{\frac{1}{2N(N-1)},\frac{1}{4N-1}\right\}$ otherwise.  Choose a good fundamental domain $\mathcal F$. There exist constants $\Theta_1,\Theta_2,\Theta_3,\Theta_4$ dependent only on $k,\mathcal F$ and $B$ such that for every $X>0$ and $a\in \O_k$ such that $ \|a\|\geq Xe^{-B}$ we have
\begin{enumerate}
\item  
$|S(a,X)|\leq \Theta_1 X^{1+\kappa}\|a\|^{-\kappa} +\Theta_2(\log X)^{2d-2}+\Theta_3\logp\logp\logp\log\|a\|+\Theta_4.$ 
\item Suppose that $a\in \mathcal F$. For every $\varepsilon>0$ there exists $M$ such that 
\begin{align*}|\{(x,\lambda)\in S(a,X)|\: \| &\lambda\|_\infty\geq M \}|\leq \\ & \varepsilon X^{1+\kappa}\|a\|^{-\kappa} +\Theta_2(\log X)^{2d-2}+\Theta_3\logp\logp\logp\log\|a\|+\Theta_4. 
\end{align*} 
\end{enumerate}
\end{proposition}
The proof consists of dividing the set $S(a,X)$ in two parts $S_1,S_2$ where $S_1$ consists of pairs $(x,\lambda)$ where $\|\lambda\|_\infty$ is "not too big" compared to $\log \|a\|-\log\|x\|$ and $S_2$ is the complement of $S_1$.
 To estimate the size of $S_1$ we will use the Aramaki--Ikehara Tauberian theorem (Section \ref{s-Aramaki}) and to control $S_2$ we rely on Baker--W\"ustholz's theorem on linear forms in logarithms and counting integer points in cylinders (Section \ref{s-Baker}). Theorem \ref{t-MainPropCount} is an easy consequence of Proposition \ref{p-CountingMain}.
 \begin{proof}[Proof of Theorem \ref{t-MainPropCount}]
 It is enough to show that $|\left\{x\in \O_k|\: \|x(a-x)\|\leq X^2\right\}|\leq 2|S(a,X)|+2$. Note that the set $\left\{x\in \O_k|\: \|x(a-x)\|\leq X^2\right\}$ is invariant under the map $x\mapsto a-x$. The inequality $\|x(a-x)\|\leq X^2$ implies that either $\|x\|\leq X$ or $\|a-x\|\leq X$. For any such $x$ different than $0$ and $a$ there exists a pair $(y,\lambda)\in S(a,X)$ such that $\lambda^{-1} y=x$ or $\lambda^{-1} y=a-x$. This proves that $|\left\{x\in \O_k|\: \|x(a-x)\|\leq X^2\right\}|\leq 2|S(a,X)|+2$. Theorem \ref{t-MainPropCount} now follows from Proposition \ref{p-CountingMain} with $\Xi_i=2\Theta_i, i=1,2,3$ and $\Xi_4=2\Theta_4+2$.
 \end{proof}
 
\subsection{Aramaki--Ikehara theorem}\label{s-Aramaki}
We will need an extension of the classical Tauberian theorem of Wiener and Ikehara due to Aramaki \cite{Aramaki1988}. Our goal is Lemma \ref{l-IdealsEst} and it is the only result from this section that we will be using later. The result stated below differs from the original formulation in \cite{Aramaki1988}, we comment on that after the statement.
\begin{theorem}\label{t-Ararmaki}(\cite{Aramaki1988})
Let $Z(s)=\sum_{n\in \N}\frac{a_n}{n^s}$ be a Dirichlet series with non-negative real coefficients, convergent for $\Re(s)$ sufficiently large. Assume that $Z(s)$ satisfies the following conditions:
\begin{enumerate}
\item $Z(s)$ has a meromorphic extension to $\C$ with poles on the real line.
\item $Z(s)$ has the first singularity at $s=a>0$ and there exist constants $A_{j}\in \C$ for $j=1,\ldots ,p$ such that   
$$Z(s)-\sum_{j=1}^p\frac{A_j}{(j-1)!}\left(-\frac{d}{ds}\right)^{j-1}\frac{1}{s-a}$$ is holomorphic in $\left\{s\in \C\mid \Re(s)>a-\delta\right\}$ for some $\delta>0$. 
\item $Z(s)$ is of polynomial order of growth with respect to $\Im(s)$ in all vertical strips, excluding neighborhoods of the poles. 
\end{enumerate}
Then, there exists $\delta_0>0$ such that for all $X\geq 1$
$$\sum_{n\leq X}a_n=\sum_{j=1}^p\frac{A_j}{(j-1)!}\left.\left(\frac{d}{ds}\right)^{j-1}\left(\frac{X^s}{s}\right)\right|_{s=a}+ O(X^{a-\delta_0}).$$
\end{theorem} 
Originally Aramaki states the theorem only for the series of the form $Z(s):=\sum_{i=}^\infty \lambda_i^{-s}$ where $\lambda_i\in\mathbb R_{>0}$ are the eigenvalues (counted with multiplicity) of certain compact positive self-adjoint operator $P$ acting on a Hilbert space. We can recover Theorem \ref{t-Ararmaki} for Dirichlet series $\sum_{n\in \N}\frac{a_n}{n^s}$ with coefficients in $\mathbb N$ by considering an operator with eigenvalue $n$ appearing with multiplicity $a_n$, for each $n\geq 1$. 
The version with non-negative real coefficients does not follow directly from \cite{Aramaki1988}. However, the proof provided by Aramaki depends only on the series $Z(s)$ and never uses the spectral interpretation. In particular, the integrality of the coefficients $a_n$ is never used, so by repeating exactly the same argument as in \cite{Aramaki1988} one can prove Theorem \ref{t-Ararmaki}. 
\begin{corollary}\label{l-AramakiCor}
Let $(a_n)_{n\in \N}$ be a sequence of positive real numbers such that the Dirichlet series $Z(s)=\sum_{n=1}^\infty \frac{a_n}{n^s}$ satisfies the hypotheses of Theorem \ref{t-Ararmaki}. Then for every integer $m\geq 0$ and $X\geq 1$ we have  
\begin{enumerate} 
\item there exists $ \delta _{m}>0$ such that $$\sum_{n\leq X}a_n(\log n)^m=\sum_{j=1}^p \frac{A_j}{(j-1)!}\left.\left(\frac{d}{ds}\right)^{m+j-1}\left(\frac{X^s}{s}\right)\right|_{s=a}+O(X^{a-\delta _{m}}).$$
\item If $Z(s)$ has a simple pole at $1$ with residue $\rho$, then there exists $\delta >0 $ such that
$$\sum_{n\leq X}a_n(\log X-\log n)^{m}= m!\rho X+O(X^{1-\delta}).$$
\end{enumerate}
\end{corollary}
\begin{proof}
\begin{enumerate}
\item Note that $\sum_{n=1}^\infty \frac{a_n(\log n)^m}{n^s}=\left(-\frac{d}{ds}\right)^m Z(s). $ The derivative $\left(-\frac{d}{ds}\right)^m Z(s)$ is meromorphic on $\C$ with poles on the real line. Cauchy's integral formula implies that $\left(-\frac{d}{ds}\right)^m Z(s)$ is of polynomial order of growth with respect to $\Im(s)$ on vertical strips away from the poles. The desired formula follows from Aramaki theorem applied to $\left(-\frac{d}{ds}\right)^m Z(s).$
\item By the previous point we have $\sum_{n\leq X}a_n(\log n)^m=\rho \left.\left(\frac{d}{ds}\right)^m\left(\frac{X^s}{s}\right)\right|_{s=1}+O(X^{1-\delta _{m} }).$ We use this identity in the following computation:
\begin{align*}
\sum_{n\leq X} a_n(\log & X -\log n)^m= \\=&\sum_{l=0}^m{m\choose l}(-1)^{m-l}(\log X)^l\sum_{n\leq X}a_n(\log n)^{m-l}\\
=&\sum_{l=0}^m {m\choose l}(-1)^{m-l}\left.\left(\frac{d}{ds}\right)^l X^{s-1}\right|_{s=1} \rho\left.\left(\frac{d}{ds}\right)^{m-l}\left(\frac{X^s}{s}\right)\right|_{s=1} +O(X^{1-\delta } )\\
=&\rho \sum_{l=0}^m {m \choose l}\left.\left(-\frac{d}{ds}\right)^l X^{1-s}\right|_{s=1} \left.\left(-\frac{d}{ds}\right)^{m-l}\left(\frac{X^s}{s}\right)\right|_{s=1} +O(X^{1-\delta })\\
=&\rho \left.\left(-\frac{d}{ds}\right)^m\left( X^{1-s}\frac{X^s}{s}\right)\right|_{s=1}+ O(X^{1-\delta })\\
=&m!\rho X+O(X^{1-\delta })
\end{align*}
\end{enumerate} 
where $\delta = {\rm min} \left\lbrace \delta _{0}, \ldots , \delta _{m} \right\rbrace/2 $. 

\end{proof}
The following lemma is a key ingredient in the proof of Proposition \ref{p-CountingMain}.
\begin{lemma}\label{l-IdealsEst}
Let $\rho_k$ be the residue of the Dedekind zeta function $\zeta_k(s)$ at $s=1$, let $h_k$ be the class number of $k$ and let $w_k$ be the size of the torsion subgroup of $\O_k^\times$. For every $m \geq 0$ there exists $\delta_0>0$ such that for every $X\geq 1$ we have
\begin{enumerate}
\item $$\sum_{\substack{a\in \O_k/\O_k^\times\\ 0<N(a)\leq X}}\log N(a)^m=\frac{\rho_k}{h_k}X(\log X)^m+O(X^{1-\delta_0})$$
and
\item $$\sum_{\substack{a\in \O_k/\O_k^\times\\ 0<N(a)\leq X}}(\log X-\log N(a))^m=m!\frac{\rho_k}{h_k}X+O(X^{1-\delta_0}).$$
\end{enumerate}
\end{lemma}
\begin{proof}
Let $\chi_1,\ldots,\chi_{h_k}$ be the characters of the class group of $k$, with $\chi_1=1$. The L-functions $L(s,k,\chi _i)=\sum_{\frak a}\frac{\chi_i(\frak a)}{(N\frak a)^s}$ are entire for $i\geq 2$ and $L(s,k,1)$ is the Dedekind zeta function of $k$ with unique simple pole at $s=1$ with residue $\rho_k$. All of them are of polynomial growth on vertical strips. Consider the Dirichlet series 
$$G(s)=\sum_{\substack{a\in \O_k/\O_k^\times\\ 0<N(a)}}\frac{1}{N(a)^s}=\sum_{\substack{\frak a\\ {\rm principal}}}\frac{1}{(N\frak a)^s}=\frac{1}{h_k}\sum_{i=1}^{h_k}L(s,k,\chi_i).$$
It has non-negative coefficients and extends to a meromorphic function on $\C$ with a simple pole at $s=1$ with residue $\frac{\rho_k}{h_k}.$ Equalities (1),(2) follow from Corollary \ref{l-AramakiCor} applied to $G(s)$.
\end{proof}
\subsection{Proof of Proposition \ref{p-CountingMain}}\label{s-ProofPropC}
If $X<1$, then $S(a,X)$ is empty so in the sequel we assume $X\geq 1$. We adopt the following convention. The constants $C_i,B_i$ appearing in the inequalities successively throughout the proof are dependent on $k$ and $B$ alone. 
We structured the proof so that it should be clear that $C_i,B_i$ depend only on $k,B$ and the constants $C_j,B_j$ for $j<i$. We omit the computations of exactly how big $C_i,B_i$ should be in terms of $k$ and $B$. 
\begin{proof}
First consider the case $N=1$. Remark that we do not need this case of the proposition in order to prove Theorem \ref{mt-noptimal}. However for completeness and in order to prove Theorem \ref{mt-Count1} and \ref{t-AvUnit} we include the proof.

We have 
\begin{equation}\label{Nrowne1}\left| S(a,X)\right| =\left|\{ x\in \mathbb{Z}\setminus \{0 \} | \ \ \|x(x-a)\| \leq X^2,\ \ \|x\| \leq X\}\right| 
 \leq \left| \{ x\in \mathbb{Z}|\ \ \|x(a-x)\| \leq X^2\}\right|. 
\end{equation}
The inequality $\|x(a-x)\| \leq X^2$ holds if and only if $\frac{a^2}{4}-X^{2}\leq \left(x-\frac{a}{2}\right)^2 \leq X^{2}+\frac{a^2}{4}$.  
\\ First assume $\frac{a^2}{4}-X^2<0$. Since $|a|\geq Xe^{-B}$, 
\begin{equation*}
\left| S(a,X)\right| \leq 2 \sqrt{X^2+\frac{a^2}{4}}+2 \leq 2\sqrt{2}X+2 \leq 4\sqrt{2}X^{2}|a|^{-1}+2 \leq 4\sqrt{2} e^{-\frac{2}{3}B} X^{1+\frac{1}{3}} |a|^{-\frac{1}{3}} +2.   
\end{equation*}
Assume now $\frac{a^2}{4}-X^{2}\geq 0$. By (\ref{Nrowne1}), 
\begin{align*} 
\left|S(a,X)\right| \leq & 2\left(\sqrt{X^{2}+\frac{a^{2}}{4}}-\sqrt{\frac{a^2}{4}-X^2} \right)+2=2\left( \int_{-X^2}^{X^2} \left(\sqrt{\frac{a^2}{4}+t}\right)'dt \right) +2 \\
\leq & \frac{2X^2}{\sqrt{\frac{a^2}{4}-X^2}}+2. 
\end{align*}
If $X<\frac{|a|}{4}$, then $|S(a,X)|\leq 2X^{1+\frac{1}{3}}|a|^{-\frac{1}{3}}+2$. Finally, if $\frac{|a|}{4}\leq X \leq \frac{|a|}{2}$ by a trivial bound $|S(a,X)|\leq 2X+2 \leq 4X^{1+\frac{1}{3}}|a|^{-\frac{1}{3}}+2$. 
Point (2) of the proposition is trivial. 
From now on assume $N>1$. 
\\(1)  The problem is invariant under multiplying $a$ by $\O_k^\times$ so we may assume, without loss on generality, that $a\in \mathcal F$. Recall that  $\|a\|\geq Xe^{-B}$ and
\begin{align*}S:=S(a,X)=&\left\{(x,\lambda)\in (\mathcal F\cap \O_k)\times \O_k^\times | \|x(a-x\lambda^{-1})\|\leq X^2, \|x\|\leq X\right\}\\
=&\left\{(x,\lambda)\in (\mathcal F\cap \O_k)\times \O_k^\times | \log\left\|\lambda-\frac{x}{a}\right\|\leq 2\log X-\log\|a\|-\log\|x\|, \|x\|\leq X\right\}.\\
\end{align*}
Let $\alpha$ be the constant from Lemma \ref{l-NormIneq}. We define 
\begin{align*}S_1:=&\left\{(x,\lambda)\in  S| \|\lambda\|_\infty \leq \frac{2}{\alpha}\left(2\log X- \left(2-\frac{1}{2N}\right)\log \|x\|-\frac{1}{2N}\log \|a\|\right)\right\}\\
\end{align*}
and $S_2:=S\setminus S_1$. We start be estimating the size of $S_1$. We will use the fact that for non-negative $R$ the number of $\lambda\in \O_k^\times$ with $\|\lambda\|_\infty\leq R$ is at most $O(R^{d-1})+ |W_k|$. 
\begin{align*}
|S_1|\leq& \sum_{\substack{x\in \mathcal{F}\cap \O_k\\ \|x\|\leq X}}\left|\left\{ \lambda\in\O_k^\times |\|\lambda\|_\infty\leq \frac{2}{\alpha}\left(2\log X-\left(2-\frac{1}{2N}\right)\log\|x\|-\frac{1}{2N} \log\|a\|\right)\right\}\right|\\
=& \sum_{\substack{x\in \mathcal{F}\cap \O_k\\ \|x\|\leq X}}\left|\left\{ \lambda\in\O_k^\times |\|\lambda\|_\infty\leq \frac{(4N-1)}{N\alpha}\left(\frac{4N}{4N-1}\log X- \frac{1}{4N-1} \log\|a\|-\log\|x\|\right)\right\}\right|\\
\end{align*}
Put $\log Y=\frac{4N}{4N-1}\log X- \frac{1}{4N-1} \log\|a\|$. The summands in the last formula vanish unless $\|x\|\leq Y$ so we get
\begin{align*}
|S_1|\leq& \sum_{\substack{x\in \mathcal{F}\cap \O_k\\ \|x\|\leq Y}}\left|\left\{ \lambda\in\O_k^\times |\|\lambda\|_\infty\leq \frac{(4N-1)}{N\alpha}\left(\log Y-\log \|x\|\right)\right\}\right|\\
\leq& \sum_{\substack{x\in \mathcal{F}\cap \O_k\\ \|x\|\leq Y}}\left(C_1\left(\log Y-\log\|x\|\right)^{d-1}+C_2\right)\\
\leq& C_3Y=C_3 X^{1+\frac{1}{4N-1}}\|a\|^{-\frac{1}{4N-1}}.
\end{align*}
The last inequality uses Lemma \ref{l-IdealsEst}.  It remains to bound the size of $|S_2|$. 
\begin{lemma} \label{l-S2Decomp} Put $B_1:=\alpha^{-1}((\log X-\log\|a\|)N^{-1}+ 2\log C_0+\log 2)$ where $C_0$ is as in Lemma \ref{l-cone}. Let $(x,\lambda)\in S_2$. Then either $\|\lambda\|_\infty< B_1$ or  there exists $i\in\left\{1,\ldots,d\right\}$ such that 
\begin{equation}\label{e-ConditionS2i}
\log\left|\frac{x}{a}-\lambda\right|_i\leq -\frac{\alpha\|\lambda\|_\infty}{2N-2}-\left(\frac{1}{N}+\frac{1}{2N(N-1)}\right)(\log\|a\|-\log\|x\|)+\log 2. 
\end{equation}
\end{lemma}
\begin{proof}
Assume that $\|\lambda\|_\infty\geq B_1$ and that $(x,\lambda)\in S_2$. Let $j\in \left\{1,\ldots,d\right\}$ be such that $|\lambda|_j$ is maximal. By Lemma \ref{l-NormIneq}, we have  $\log |\lambda|_j \geq \alpha\|\lambda\|_\infty.$  Since $\|\lambda\|_\infty\ge B_1$ we have   $\log |\lambda|_j\geq (\log X-\log\|a\|)N^{-1}+2\log C_0+\log 2$. By Lemma \ref{l-cone}, $\log \left|\frac{x}{a}\right|_j\leq (\log\|x\|-\log\|a\|)N^{-1} + 2\log C_0\leq (\log X-\log\|a\|)N^{-1}+2\log C_0$. It follows that $|\lambda|_j\geq 2|\frac{x}{a}|_j$ so we have $\log |\frac{x}{a}-\lambda|_j\geq \log|\lambda|_j-\log 2.$ From this and the fact that $(x,\lambda)\in S_2$ we deduce that 

\begin{align*}\log|\lambda|_j \geq & \alpha\|\lambda\|_\infty\geq  \frac{\alpha}{2}\|\lambda\|_\infty+ \left(2\log X-\frac{4N-1}{2N}\log\|x\|-\frac{1}{2N}\log\|a\|\right)
\end{align*}
and
\begin{align}\label{e-MaxCoord}
\log\left|\frac{x}{a}-\lambda\right|_j\geq& \frac{\alpha}{2}\|\lambda\|_\infty+ \left(2\log X-\frac{4N-1}{2N}\log\|x\|-\frac{1}{2N}\log\|a\|\right)-\log 2.
\end{align}
At the same time $|\lambda|_j \geq 1$ because $\|\lambda\|=1$ so we also have $\log\left|\frac{x}{a}-\lambda\right|_j\geq -\log 2$. This observation is valid even if $B_1<0$.  By definition of $S$ we have 
\begin{equation*}
\log\left\|\frac{x}{a}-\lambda\right\|\leq 2\log X-\log\|a\|-\log\|x\|.
\end{equation*}
Let $f=1$ if $j> r_1$ and $f=0$ otherwise. Subtracting (\ref{e-MaxCoord}) we get
\begin{align*}
\sum_{i=1,i\neq j}^{r_1}\log\left|\frac{x}{a}-\lambda\right|_i+2\sum_{i=r_1+1,i\neq j}^d \log\left|\frac{x}{a}-\lambda\right|_i + f\log\left|\frac{x}{a}-\lambda\right|_j \leq \\ -\frac{\alpha}{2}\|\lambda\|_\infty -\frac{2N-1}{2N}\left(\log\|a\|-\log\|x\|\right)+\log 2.
\end{align*}
At least one term in the sum must be smaller or equal to the average. Therefore, for some $i$ we have 
\begin{equation}
\log\left|\frac{x}{a}-\lambda\right|_i\leq -\frac{\alpha\|\lambda\|_\infty}{2N-2} -\left(\frac{1}{N}+\frac{1}{2N(N-1)}\right)\left(\log\|a\|-\log\|x\|\right)+\frac{\log 2}{N-1}.
\end{equation}
This is slightly better than what we needed to prove.
\end{proof}
Put $S^0_2:=\left\{(x,\lambda)\in S_2| \|\lambda\|_\infty\leq B_1\right\}$ and for $i=1,\ldots,d$ let 
\begin{equation}\label{e-DefS2i}S^i_2:=\left\{(x,\lambda)\in S_2| \textrm{ inequality (\ref{e-ConditionS2i}) holds }\right\}.
\end{equation}
\begin{lemma}\label{l-S2-0Bound}
There is a constant $C_5$ dependent only on $k,B$ such that 
$$|S^0_2|\leq C_5 X^{1+\kappa} \|a\|^{-\kappa}.$$
\end{lemma}
\begin{proof}
If $B_1 \geq 0$, then the number of $\lambda$ satisfying $\|\lambda\|_\infty \leq B_1$, where $B_1$ is as in Lemma \ref{l-S2Decomp}, is at most $O(1+B_1)^{d-1}\leq O(X^{\kappa}\|a\|^{-\kappa})$ so there is a constant $C_4$ such that 
$$|S^0_2|\leq C_4 X^{\kappa}\|a\|^{-\kappa}\sum_{\substack{x\in\mathcal{F}\cap \O_k\\ \|x\|\leq X}}1\leq C_5 X^{1+\kappa}\|a\|^{-\kappa}.$$ The last inequality uses Lemma \ref{l-IdealsEst}. If $B_1<0$, then $S^0_2$ is empty. 
\end{proof}
We have the following estimate on $|S^i_2|$ for $i=1,\ldots,d$.
\begin{lemma}\label{l-S2iBound}
Let $\kappa'=\frac{1}{2N(N-1)}.$ There are constants $C_6,C_7,C_8,C_9$ dependent on $k,B$ alone such that for $i=1,\ldots, d$ we have
$$|S^i_2|\leq C_6 X^{1+\kappa'}\|a\|^{-\kappa'}+ C_7(\log X)^{2(d-1)} + C_8\logp\logp\logp\log\|a\|+C_9.$$
\end{lemma}
The proof of the Lemma \ref{l-S2iBound} relies on the Baker--W\"ustholz's bounds on linear forms in logarithms and it's quite independent from the rest. We postpone it to the next section and move on to prove Proposition \ref{p-CountingMain}. By Lemma \ref{l-S2Decomp}, we have $S_2=\bigcup_{i=0}^d S^i_2$ so 
\begin{align}
|S|\leq & |S_1|+\sum_{i=0}^d|S^i_2|\\
\leq&  C_3 X^{1+\frac{1}{4N-1}}\|a\|^{-\frac{1}{4N-1}} + C_5X^{1+\kappa}\|a\|^{-\kappa} + dC_6 X^{1+\kappa'}\|a\|^{-\kappa'}\\ &+dC_7(\log X)^{2(d-1)}+ dC_8\logp\logp\logp\log\|a\|+ dC_9.
\end{align}
As $\kappa= \min\left\{\frac{1}{4N-1}, \kappa'\right\}$ and $X\|a\|^{-1}\leq e^B$  we can deduce that
\begin{align}
|S|\leq & \Theta_1 X^{1+\kappa}\|a\|^{-\kappa} +\Theta_2(\log X)^{2(d-1)} + \Theta_3\logp\logp\logp\log\|a\|+\Theta_4, 
\end{align}
where $\Theta_1,\Theta_2,\Theta_3,\Theta_4$ depend only on $k,B$.  This proves the first part of Proposition \ref{p-CountingMain}.

We proceed to the proof of part (2). Let $M>0$. Put $S[M]:=\left\{(x,\lambda)\in S|\,\|\lambda\|_\infty\geq M\right\}$ and $S_1[M]:=S_1\cap S[M],S_2[M]:=S_2\cap S[M],S_2^i[M]:=S_2^i\cap S[M]$ for $i=0,\ldots ,d$. The proof of this part is reduced to the following lemmas.
\begin{lemma}\label{l-BoundS1M} For every $\delta>0$ there exists $M_1$ such that  for every $M\geq M_1$
$$|S_1[M]|\leq \delta X^{1+\kappa}\|a\|^{-\kappa}.$$
\end{lemma}
\begin{proof}
\begin{align*}
&|S_1[M]|\leq\\ &\sum_{\substack{x\in \mathcal{F}\cap \O_k\\ \|x\|\leq X}}\left|\left\{ \lambda\in\O_k^\times |M\leq \|\lambda\|_\infty\leq \frac{2}{\alpha}\left(2\log X-\left(2-\frac{1}{2N}\right)\log\|x\|-\frac{1}{2N} \log\|a\|\right)\right\}\right|\\
=& \sum_{\substack{x\in \mathcal{F}\cap \O_k\\ \|x\|\leq X}}\left| \left\{ \lambda\in\O_k^\times |M\leq \|\lambda\|_\infty\leq \frac{4N-1}{N\alpha}\left(\frac{4N}{4N-1}\log X- \frac{1}{4N-1} \log\|a\|-\log\|x\|\right)\right\}\right|\\
\end{align*}
The summands in the last formula vanish unless $M\leq \frac{(4N-1)}{N\alpha}(\frac{4N}{4N-1}\log X- \frac{1}{4N-1} \log\|a\|-\log\|x\|)$ i.e. $\log \|x\|\leq \frac{4N}{4N-1}\log X- \frac{1}{4N-1} \log\|a\|-\frac{N\alpha}{(4N-1)}M$. Put  $\log Y_M:=\frac{4N}{4N-1}\log X- \frac{1}{4N-1} \log\|a\|-\frac{MN\alpha}{(4N-1)}$. We get
\begin{align*}
|S_1[M]|\leq& \sum_{\substack{x\in \mathcal{F}\cap \O_k\\ \|x\|\leq Y_M}}\left|\left\{ \lambda\in\O_k^\times |\|\lambda\|_\infty\leq \frac{4N-1}{N\alpha}\left(\log Y_M-\log \|x\|\right)\right\}\right|\\
\leq& \sum_{\substack{x\in \mathcal{F}\cap \O_k\\ \|x\|\leq Y_M}}\left(C_1\left(\log Y_M-\log\|x\|\right)^{d-1}+C_2\right)\\
\leq& C_3Y_M=C_3 X^{1+\frac{1}{4N-1}}\|a\|^{-\frac{1}{4N-1}}e^{-\frac{MN\alpha}{4N-1}}.
\end{align*}
For the last inequality we have used Lemma \ref{l-IdealsEst}. As $\|a\|\geq Xe^{-B}$ we have 
\begin{equation*}|S_1[M]|\leq e^{-B\left(\frac{1}{4N-1}-\kappa\right)}X^{1+\kappa}\|a\|^{-\kappa}e^{-\frac{MN\alpha}{4N-1}}.\end{equation*} Clearly for $M\geq M_1$ sufficiently large we have $e^{-\frac{MN\alpha}{4N-1}-B\left(\frac{1}{4N-1}-\kappa \right)}\leq \delta$. The lemma is proven.

\end{proof}
We have the following analogue of Lemma \ref{l-S2iBound}.
\begin{lemma}\label{l-S2iMBound}
Let $\kappa'=\frac{1}{2N(N-1)}.$ For every $\delta>0$ and $i=1,\ldots,d$ there exists $M_2$ such that for every $M\geq M_2$  
$$|S^i_2[M]|\leq \delta X^{1+\kappa'}\|a\|^{-\kappa'}+C_7(\log X)^{2(d-1)}+C_8\logp\logp\logp\log\|a\|+C_9.$$
\end{lemma}
The proof is postponed to the next section. We will also need the following easy observation.
\begin{lemma}\label{l-S20M} There exists $M_3$ such that for every $M\geq M_3$ the set $S^0_2[M]$ is empty. \end{lemma}
\begin{proof}
It is enough to choose $M_3>B_1$. Recall that $B_1:=\alpha^{-1}((\log X-\log\|a\|)N^{-1}+ 2\log C_0+\log 2)$. By assumptions of Proposition \ref{p-CountingMain}, $\log X-\log \|a\|\leq B$, so $B_1$ is bounded independently of $X$ and $\|a\|$.
\end{proof}
We are ready to prove Proposition \ref{p-CountingMain} (2).  Choose $M$ such that $S_2^0[M]$ is empty, Lemma \ref{l-BoundS1M} and Lemma \ref{l-S2iMBound} hold with $\delta=\frac{\varepsilon}{e^{B}(d+1)}$. By Lemma \ref{l-S2Decomp}, we have:
$$S[M]=S_1[M]\cup S_2[M]=S_1[M]\cup S_2^0[M]\cup \bigcup_{i=1}^d S_2^i[M]$$
$$|S[M]|\leq \varepsilon X^{1+\kappa}\|a\|^{-\kappa}+ d C_7(\log X)^{2(d-1)}+dC_8\logp\logp\logp\log\|a\|+dC_9.$$ This concludes the proof of Proposition \ref{p-CountingMain}.
\end{proof}
\subsection{Linear forms in logarithms and the bound on $|S^i_2|$}\label{s-Baker}
The aim of this section is to show Lemma \ref{l-S2iBound} i.e. an upper bound on $|S^i_2|$ where $S^i_2$ is the set defined by (\ref{e-DefS2i}). Next we apply more or less the same argument to prove Lemma \ref{l-S2iMBound}. Our main tool is the Baker--W\"ustholz inequality on linear forms in logarithms \cite[Theorem 7.1]{BakerWustholz07}. 

We recall the definition of the logarithmic Weil height of an algebraic number. Let $K$ be a finite extension of $\Q$ and let $\omega\in K$. 
\begin{definition} The logarithmic Weil height of $\omega$ is defined as
$$h(\omega)=\frac{1}{[K:\Q]}\sum_{\nu\in\Sigma}a_\nu\max\left\{0,\log|\omega|_\nu\right\},$$ where $a_\nu=2$ if $\nu$ is a complex Archimedean place and $a_\nu=1$ otherwise and $\Sigma $ denotes the set of normalized valuations of $K$. The value of $h(\omega)$ does not depend on the choice of $K$.
\end{definition} 
The height enjoys the following sub-additivity properties: $h(xy)\leq h(x)+h(y), h(x/y)\leq h(x)+h(y), h(x^m)=mh(x).$ We also have $h(|x|)\leq h(x)$ for any algebraic complex number. For later use we define $h'(\omega):=\max\left\{h(\omega),1\right\}$. This definition agrees with the one from \cite[7.2]{BakerWustholz07} up to a constant depending only on $[\Q(\omega):\Q]$. 
\begin{theorem}[{\cite[Theorem 7.1]{BakerWustholz07}}] \label{t-BW} Let $\alpha_1,\ldots,\alpha_n\in \overline{\Q}\setminus \{ 0,1 \}$ and let $\log\alpha_i$ be the value of the main branch of logarithm for $i=1,\ldots,n$. Let $D=[\Q(\alpha_1,\ldots,\alpha_n):\Q]$. For every $b_1,\ldots,b_n\in \Z$ such that 
$$\Lambda:=b_1\log\alpha_1+\ldots b_n\log\alpha_n\neq 0$$ we have 
$$\log|\Lambda|\geq -C_{n,D}h'(\alpha_1)\ldots h'(\alpha_n)\max\left\{1,\log\max_{i=1,\ldots,n}\frac{|b_i|}{b}\right\},$$ where $C_{n,D}$ is a positive constant depending only on $n$ and $D$ and $b$ is the highest common divisor of $b_1, \ldots ,b_n$.
\end{theorem}
Recall that $\xi_1,\ldots,\xi_{d-1}$ form a basis of a maximal torsion free subgroup of $\O_k^{ \times }$, so every element $\lambda\in \O_k^\times$ can be written as $\lambda=w\xi_1^{b_1}\ldots \xi_{d-1}^{b_{d-1}}$ with $w$ torsion and $\|\lambda\|_\infty:=\max_{i=1,\ldots,d-1}|b_i|.$ 
\begin{corollary}\label{c-BW}
Let $x\neq y\in \mathcal F\cap \O_k$, let  $i\in\left\{1,\ldots,d\right\}$  and let $\lambda\in \O_k^\times$. Then 
$$\log\left|\log \left( \frac{x_i}{y_i}\lambda_i \right)\right|\geq -C_{10}(1+\log\|x\|+\log\|y\|)\max\left\{1,\log\|\lambda\|_\infty\right\},$$
where $C_{10}>0$ depends only on $k$ and the choice of $\xi_1,\ldots,\xi_{d-1}$.
\end{corollary}
\begin{proof}
As $x,y\in \mathcal{F}\cap \O_k$ the definition of Weil height with $K=k$ and Lemma \ref{l-cone} imply that $h(x)\leq \frac{1}{N}\log\|x\|+ \log C_0$ and similarly for $y$. We have $h'(\frac{x}{y})\leq 1+h(\frac{x}{y})\leq 1+h(x)+h(y)= O(1+\log\|x\|+\log\|y\|)$. Write $\lambda=w \xi_1^{b_1}\xi_2^{b_2}\ldots\xi_{d-1}^{b_{d-1}}$ with $w$ being a torsion element. Observe that by the definition of $\mathcal{F}$, $\frac{x_i}{y_i}\lambda _i \neq 1$. Theorem \ref{t-BW} yields 
\begin{equation*}\log\left|\log \left(\frac{x_i}{y_i}\lambda_i\right)\right|
 \geq -C_{d,N}h'\left(\frac{x}{y}\right)h'(\xi_1)\ldots h'(\xi_{d-1})\max\left\{1,\log \max_{j=1,\ldots, d-1}|b_j|\right\}.\end{equation*} 
Since $\max_{j=1,\ldots, d-1}|b_j|=\|\lambda\|_\infty$, the corollary follows.
\end{proof}

\begin{definition}\label{d-Cylinder}A cylinder in $V\simeq \R^{r_1}\times \C^{r_2}$ is a set $\mathcal C$ which is a coordinate-wise product of closed balls  $$\mathcal C=\prod_{i=1}^{r_1} B_{\R}(t_i, R_i)\times \prod_{i=r_1+1}^{d} B_{\C}(t_i, R_i),$$ with $t_i\in \R$ for $i=1,\ldots, r_1$, $t_i\in \C$ for $i=r_1+1,\ldots, d$ and $R_i\in \R_{\geq 0}$ for $i=1,\ldots, d$. 
\end{definition}
\begin{lemma}\label{l-Cylinder}
Let $\mathcal C$ be a cylinder. Then $|\mathcal C\cap \O_k|\leq 1+C_{11}\Leb(\mathcal C)$ where $C_{11}$ is a constant depending only on $k$. 
\end{lemma}
\begin{proof}
First we prove that any cylinder $\mathcal{C}'$ of volume strictly below $\pi^{r_2}4^{-r_2}$ cannot contain more than one point of $\O_k$. Write 
$$\mathcal C'=\prod_{i=1}^{r_1} B_{\R}(t_i', R_i')\times \prod_{i=r_1+1}^{d} B_{\C}(t_i', R_i').$$ If $x,y\in \mathcal C'$, then $|x-y|_i\leq 2R_i'$ for every $i=1,\ldots,d$. We deduce that $$\|x-y\|\leq \prod_{i=1}^{r_1}2R_i'\prod_{i=r_1+1}^{d}4R_i'^2=4^{r_2}\pi^{-r_2}\Leb\mathcal C'<1.$$ On the other hand if $x,y\in \O_k$ are distinct, then $\|x-y\|=|N_{k/\Q}(x-y)|\geq 1$. Hence $\mathcal C'$ can contain at most one point from $\O_k$. The lemma follows since we can cover $\mathcal C$ with at most $1+C_{11}\Leb(\mathcal C)$ cylinders of volume  $\frac{\pi^{r_2}4^{-r_2}}{2}.$
\end{proof}

\begin{lemma}\label{l-log}
For every $z\in\C$ with $|1-z|\leq \frac{1}{2}$ we have $\log|\log z|\leq \log|1-z|+\log 2$. 
\end{lemma}
\begin{proof}
Let $z=1-t$. Then $|t|\leq \frac{1}{2}$ and $\log z=-\sum_{n=1}^\infty \frac{t^n}{n}$. Hence $|\log z|\leq 2|t|$ and consequently $\log|\log z|\leq \log|1-z|+\log 2$.
\end{proof}
We can are ready to prove Lemma \ref{l-S2iBound}.
\begin{proof}[Proof of Lemma \ref{l-S2iBound}]
Recall that $\kappa':=\frac{1}{2N(N-1)}.$ For $\lambda\in \O_k^\times$ define 
\begin{align*}S_2^i(\lambda):=&\left\{x\in \mathcal F\cap \O_k| (x,\lambda)\in S^i_2\right\}\\
T^i:=&\left\{\lambda\in \O_k^\times |S_2^i(\lambda)\neq \emptyset\right\}.
\end{align*}
Put $\beta:=\frac{\alpha}{2N-2}$. By definition and by Lemma \ref{l-cone}, for every $x\in S_2^i(\lambda)$ we have 
\begin{align*}\log\left|\frac{x}{a}-\lambda\right|_i\leq& -\beta\|\lambda\|_\infty -\left(\frac{1}{N}+\kappa'\right)(\log\|a\|-\log\|x\|)+\log 2.\\
\log\left|x-a\lambda\right|_i\leq& -\beta\|\lambda\|_\infty +\log|a|_i-\left(\frac{1}{N}+\kappa'\right)(\log\|a\|-\log\|x\|)+\log 2\\
\leq& -\beta\|\lambda\|_\infty -\kappa'(\log\|a\|-\log X)+\frac{1}{N}\log X+\log 2+\log C_0.
\end{align*}
By Lemma \ref{l-cone}, the set $\left\{x\in \mathcal F\cap \O_k|\|x\|\leq X\right\}$ is contained in the cylinder 
\begin{equation*}\prod_{j=1}^{r_1}B_\R(0,C_0X^{1/N})\times \prod_{j=r_1+1}^d B_\C(0, C_0X^{1/N}). 
\end{equation*} Hence, $S_2^i(\lambda)$ is contained in the cylinder 
\begin{align*}\mathcal{C}^i(\lambda)= \prod_{j=1,j\neq i}^{r_1}B_\R(0,C_0X^{1/N})\times & \prod_{j=r_1+1,j\neq i}^d B_\C(0, C_0X^{1/N})\times \\ &B_{\mathbb K}(a_i\lambda_i, 2C_0X^{1/N +\kappa'}\|a\|^{-\kappa'}e^{-\beta\|\lambda\|_\infty}), 
\end{align*} where $\mathbb K=\R$ if $i=1,\ldots,r_1$ and $\mathbb K=\C$ otherwise.
We have 
\begin{equation*}\Leb(\mathcal{C}^i(\lambda))\leq C_{12}X^{1+\kappa'}\|a\|^{-\kappa'}e^{-\beta\|\lambda\|_\infty} \quad \textrm{if} \quad i=1,\ldots,r_1 
\end{equation*} and  
\begin{equation*}\Leb(\mathcal{C}^i(\lambda))\leq C_{12}X^{1+2\kappa'}\|a\|^{-2\kappa'}e^{-2\beta\|\lambda\|_\infty} \quad \textrm{if} \quad i=r_1+1,\ldots,r_2. 
\end{equation*} We work under assumption that $X\leq \|a\|e^B$ so in the second case we have 
\begin{equation*}C_{12}X^{1+2\kappa'}\|a\|^{-2\kappa'}e^{-2\beta\|\lambda\|_\infty}\leq e^{B\kappa'}C_{12}X^{1+\kappa'}\|a\|^{-\kappa'}e^{-\beta\|\lambda\|_\infty}. 
\end{equation*} 
By Lemma \ref{l-Cylinder}, we get 
$$|S_2^i(\lambda)|\leq 1+C_{11} \Leb(\mathcal{C}^i(\lambda))\leq 1+C_{13}X^{1+\kappa'}\|a\|^{-\kappa'}e^{-\beta\|\lambda\|_\infty}.$$
Hence
\begin{align}\label{e-S2iFirstBound} |S_2^i|\leq& \sum_{\lambda\in T^i}|S_2^i(\lambda)|\leq |T^i|+C_{13}X^{1+\kappa'}\|a\|^{-\kappa'}\sum_{\lambda\in \O_k^\times}e^{-\beta\|\lambda\|_\infty}\\
\leq& |T^i|+C_{14} X^{1+\kappa'}\|a\|^{-\kappa'}.
\end{align}
It remains to bound $|T^i|$. First we show that for every $\lambda\in T^i$ we have 
\begin{equation*} \|\lambda\|_\infty \leq C_{15}\log\|a\|\logp\log\|a\|+C_{16}
\end{equation*}  (equation (\ref{e-UpperLogBound})). Let $x \in S_2 ^i (\lambda )$. We have 
\begin{align*}\log\left|\frac{x}{a}-\lambda\right|_i\leq& -\beta\|\lambda\|_\infty -\left(\frac{1}{N}+\kappa'\right)(\log\|a\|-\log\|x\|)+\log 2.\\
\log\left|1-\frac{a}{x}\lambda\right|_i\leq& -\beta\|\lambda\|_\infty +\log|a|_i-\log|x|_i-\left(\frac{1}{N}+\kappa'\right)(\log\|a\|-\log\|x\|)+\log 2\\
\leq& -\beta\|\lambda\|_\infty -\kappa'(\log\|a\|-\log X)+\log 2+2\log C_0\\
\leq& -\beta\|\lambda\|_\infty +\kappa' B+\log 2 +2\log C_0=:-\beta\|\lambda\|_\infty + B_2.
\end{align*}
Here we define the constant $B_2=\kappa' B+\log 2 +2\log C_0$ to lighten the notation. It follows that for $\|\lambda\|_\infty\geq \frac{B_2+\log 2}{\beta}$ we will have $\left|1-\frac{a}{x}\lambda\right|_i\leq \frac{1}{2}$ and by Lemma \ref{l-log}
\begin{equation}\label{e-LogIneq}\log\left|\log\left(\frac{a_i}{x_i}\lambda _i\right)\right|\leq -\beta\|\lambda\|_\infty+B_2+\log 2.\end{equation}
Put $B_3=\max\left\{\frac{B_2+\log 2}{\beta}, 3\right\}$. For  $\|\lambda\|_\infty\geq B_3$ Corollary \ref{c-BW} yields 
\begin{equation}
-C_{10}(1+\log\|x\|+\log\|a\|)\log\|\lambda\|_\infty\leq -\beta\|\lambda\|_\infty+ B_2+\log 2.
\end{equation} The only thing we used is that $\|\lambda\|_\infty\geq 3$ so $\max\left\{1,\log\|\lambda\|_\infty\right\}=\log\|\lambda\|_\infty.$
Using inequalities $\log\|x\|\leq \log X\leq \log\|a\|+B$ we get 
\begin{equation*}
-C_{10}(1+B+2\log\|a\|)\log\|\lambda\|_\infty\leq -\beta\|\lambda\|_\infty+ B_2+\log 2.
\end{equation*}
For $\|\lambda \|_{\infty}\geq B_3\geq \frac{B_2+\log2}{\beta}$ we deduce that 
\begin{align}\label{e-UpperLogBound}
\|\lambda\|_\infty\leq C_{15}\log\|a\|\logp\log\|a\|+C_{16}.
\end{align} We proved this inequality under the assumption that $\|\lambda\|_\infty\geq B_3$ but by making $C_{16}$ bigger if necessary this inequality is also valid if $\|\lambda\|_\infty\leq B_3$. Inequality (\ref{e-UpperLogBound}) already implies a non-trivial upper bound of the form $|T^i|=O((\log\|a\|\logp\log\|a\|)^{d-1})+O(1)$. This is too weak for our purposes when $\|a\|$ is large. To get the desired bound we need to consider the relations between pairs $\lambda,\lambda'\in T^i$ with $B_3\leq \|\lambda\|_\infty\leq \|\lambda'\|_\infty$.
\begin{lemma}\label{l-PairsTw}Let $\lambda\neq\lambda'\in T^i$ with $B_3\leq \|\lambda\|_\infty\leq \|\lambda'\|_\infty$. Then 
$$\beta\|\lambda\|_\infty-B_2- 2\log 2\leq C_{10}(1+2\log X)\log 2\|\lambda'\|_\infty.$$
\end{lemma}
\begin{proof}
We stress that the estimate we are going to prove becomes interesting only when $\|a\|$ is very large compared to $X$. Let $x,x'\in \mathcal F\cap \O_k$ be such that $(x,\lambda),(x',\lambda')\in S_2^i$. Inequality (\ref{e-LogIneq}) yields 
$\left|\log\left(\frac{a_i}{x_i}\lambda _i\right)\right|\leq 2e^{-\beta\|\lambda\|_\infty+B_2}$ and the same for $(x',\lambda')$. Taking the difference we get 
$$\left|\log\left(\frac{x_i'}{x_i}\lambda _i\lambda_i'^{-1}\right)\right|= \left|\log\left(\frac{a_i}{x_i}\lambda_i\right)-\log\left(\frac{a_i}{x_i'}\lambda_i'\right)\right|\leq 2e^{B_2}(e^{-\beta\|\lambda\|_\infty}+e^{-\beta\|\lambda'\|_\infty})\leq 4e^{-\beta\|\lambda\|_\infty + B_2}.$$
In this way, we got rid of $a$. Using Corollary \ref{c-BW} we get 
\begin{align*}\label{e-LogDoubleIneq}
-\beta\|\lambda\|_\infty+B_2+ 2\log 2 \geq& \log \left|\log\left(\frac{x_i'}{x_i}\lambda _i\lambda_i'^{-1}\right)\right|\\ \geq& -C_{10}(1+\log\|x\|+\log\|x'\|)\max\left\{1,\log \|\lambda\lambda'^{-1}\|_\infty\right\}\\ \geq&-C_{10}(1+2\log X)\log 2\|\lambda'\|_\infty.
\end{align*}
Therefore 
$ \beta\|\lambda\|_\infty-B_2- 2\log 2\leq C_{10}(1+2\log X)\log 2\|\lambda'\|_\infty.$
\end{proof}
Let $B_4\geq \max\left\{9370,B_3\right\}$ be a constant dependent only on $C_{10}, B_3$ and $B_2$ such that whenever $\|\lambda\|_\infty\geq B_4(\log X)^2+ B_4$ we have $\beta\|\lambda\|_\infty-B_2-2\log 2\geq C_{10}(1+2\log X)\|\lambda\|_\infty^{1/2}.$
We divide the set $T^i$ into two parts: a "tame" part $T^i_t:=\left\{\lambda\in T^i| \|\lambda\|_\infty\leq B_4(\log X)^2+B_4\right\}$ and a "wild" part $T^i_w:=T^i\setminus T^i_t$. We have a simple estimate for $|T^i_t|$ 
\begin{equation}\label{e-TtBound}
|T^i_t|\leq C_{18}(\log X)^{2(d-1)} +C_{18,5}.
\end{equation}
The set $T_w^i$ is finite. Denote $L:=|T_w^i|$. Let us list the elements of $T^i_w$ as $\lambda_1,\ldots,\lambda_L$ in such a way that $\|\lambda_l\|_\infty\leq \|\lambda_{l+1}\|_\infty$ for $l=1,\dots,L-1.$ By Lemma \ref{l-PairsTw} and choice of $B_4$, we have 
$$ C_{10}(1+2\log X)\|\lambda_l\|_\infty^{1/2}\leq \beta\|\lambda_l\|_\infty-B_2- 2\log 2\leq C_{10}(1+2\log X)\log 2\|\lambda_{l+1}\|_\infty.$$
Therefore $\|\lambda_l\|_\infty^{1/2}\leq \log 2\|\lambda_{l+1}\|_\infty$ for $l=1,\ldots,L-1$. Since $\|\lambda_l\|_\infty\geq 9370 $ we have $(\log 2\|\lambda_l\|_\infty)^2\leq \|\lambda_l\|_\infty^{1/2}$ so 
$$(\log2\|\lambda_l\|_\infty)^2\leq \log 2\|\lambda_{l+1}\|_\infty.$$
Now an elementary induction shows that $\log2\|\lambda_L\|_\infty\geq (\log (2\times 9370))^{2^{L-1}}> e^{2^L}.$ Together with (\ref{e-UpperLogBound}) this yields 
$$e^{e^{2^L}}=e^{e^{2^{|T^i_w|}}}\leq 2\|\lambda_L\|_\infty\leq 2C_{15}\log\|a\|\logp\log\|a\|+2C_{16},$$
\begin{equation}\label{e-TwBound}
|T^i_w|\leq C_{19}\logp\logp\logp\log\|a\|+ C_{19,5}.
\end{equation}
By (\ref{e-TtBound}) and (\ref{e-TwBound}) we get 
\begin{equation}\label{e-TiFinal}|T^i|\leq C_{18}(\log X)^{2(d-1)}+C_{19}\logp\logp\logp\log\|a\|+C_{20}\end{equation} where 
$C_{20}=C_{18,5}+C_{19,5}.$ Together with (\ref{e-S2iFirstBound}) this gives Lemma \ref{l-S2iBound}.

\end{proof}
The proof of Lemma \ref{l-S2iMBound} is very similar.
\begin{proof}[Proof of Lemma \ref{l-S2iMBound}.]
We adopt notation from the proof of Lemma \ref{l-S2iBound}. By the same reasoning as in the proof of Lemma \ref{l-S2iBound} we get 
\begin{align}\label{e-S2iMFirst}
|S^i_2[M]|\leq& \sum_{\substack{\lambda\in T^i\\ \|\lambda\|_\infty\geq M}}|S^i_2(\lambda)|\leq |T^i|+C_{13}X^{1+\kappa'}\|a\|^{-\kappa'}\sum_{\substack{\lambda\in \O_k^{\times}  \\ \|\lambda\|_\infty\geq M}}e^{-\beta\|\lambda\|_\infty}.
\end{align}
For $M_2$ sufficiently large we have $$\sum_{\substack{\lambda\in \O_k \\ \|\lambda\|_\infty\geq M_2}}e^{-\beta\|\lambda\|_\infty}\leq \delta C_{13}^{-1}$$ so (\ref{e-S2iMFirst}) yields 
$|S^i_2[M_2]|\leq \delta X^{1+\kappa'}\|a\|^{-\kappa'}+|T^i|.$ By inequality (\ref{e-TiFinal}) we get 
$$|S^i_2[M_2]|\leq \delta X^{1+\kappa'}\|a\|^{-\kappa'}+C_{18}(\log X)^{2(d-1)}+C_{19}\logp\logp\logp\log\|a\|+C_{20}.$$ The lemma is proven.
\end{proof}
\subsection{Average number of solutions of unit equations}\label{s-AverageNumber}
For completeness we explain how Theorem \ref{t-AvUnit} follows from Theorem \ref{t-MainPropCount}.
\begin{proof}[Proof of Theorem \ref{t-AvUnit}]
Let $a=\alpha_3$. Assume that $\alpha_1\lambda_1+\alpha_2\lambda_2=\alpha_3$ for some $\lambda_1,\lambda_2\in\O_k^\times$ and $\alpha_1,\alpha_2, \alpha _3\in\O_k \setminus \{ 0\}$. If we put $x=\alpha_1\lambda_1$, then $\alpha_2\lambda_2=a-x$ and $\|x(a-x)\|=\|\alpha_1\alpha_2\|$. Hence, the sum $$\sum_{\substack{\alpha_1,\alpha_2\in \O_k/\O_k^{\times}\\ 0<\|\alpha_1\alpha_2\|\leq X^2 }}\nu(\alpha_1,\alpha_2,\alpha_3)$$ counts the number of $x\in\O_k\setminus\{0\}$ such  that $\|x(a-x)\|\leq X^2$. Theorem now follows from Theorem \ref{t-MainPropCount}.
\end{proof}
\section{Geometry of $n$-optimal sets. }\label{s-Geometry}
As before let $k$ be a number field of degree $N$ and let $d=r_1+r_2$ where $r_1, r_2$ are numbers of the real and complex places of $k$. Recall that $V\simeq \R^{r_1}\times \C^{r_2}$. The aim of this section is to show: 
\begin{theorem}\label{t-OptimalFrame}
There exists a positive constant $\Theta_5$ dependent only on $k$, with the following property. For every $n$-optimal set $\mathcal S\subset \O_k$ there exists a cylinder (see Definition \ref{d-Cylinder}) $\mathcal C$ of volume $\Theta_5 n$ such that $\mathcal S\subset \mathcal C$. 
\end{theorem}
We prove it in Section \ref{s-ProofOptimalFrame}. As an easy  consequence we get:
\begin{corollary}\label{c-CompactFrame} There exists a positive constant $A_3$ depending only on $k$, such that the cylinder $\Omega=B_\R(0,A_3)^{r_1}\times B_\C(0,A_3)^{r_2}$ has the following property. Let $\mathcal S$ be an $n$-optimal set in $\O_k$. Then there exist $t,s\in V$ with $\|s\|=n|\Delta_k|^{1/2}$ such that $s^{-1}(\mathcal{S}-t)\subset \Omega.$
\end{corollary}
\begin{proof}Let $\mathcal C$ the cylinder from Theorem \ref{t-OptimalFrame}. Let $t$ be the center of $\mathcal C$. We have 
$$\mathcal S- t\subset \mathcal C-t=\prod_{i=1}^{r_1}B_\R(0,v_i)\times \prod_{i=r_1+1}^d B_\C(0,v_i),$$ with $\prod_{i=1}^{r_1}(2v_i)\prod_{i=r_1+1}^d(\pi v_i^2)=\Leb(\mathcal C)\leq \Theta_5 n$. Let $A_3=(\Theta_5 |\Delta_k|^{-1/2}2^{-r_1}\pi^{-r_2})^{1/N}.$ Put  $s=(s_1,\ldots,s_d)$ where $s_i=v_i (n|\Delta_k|^{1/2}2^{r_1}\pi^{r_2}\Leb(\mathcal C)^{-1})^{1/N}$. Then $\|s\|=n|\Delta_k|^{1/2}$ and $s^{-1}(\mathcal{C}-t)\subset \Omega$ because $|s_i^{-1}v_i|=(\Leb(\mathcal C)n^{-1}|\Delta_k|^{-1/2}2^{-r_1}\pi^{-r_2})^{1/N}\leq A_3$.
\end{proof}

\subsection{Generalities on $n$-optimal sets} \label{s-Gennoptimal}
For a finite subset $F\subset \O_k$ we define the energy ideal \footnote{In \cite{BFS2017} it was called the volume. We decided to change the name to avoid confusion with the Lebesgue measure.} of $F$ as the principal ideal $\Vol(F)=\prod_{x\neq y\in F}(x-y)$. For $m \in \mathbb{N}$, $m \geqslant 0$ let $m!_k:=m!_{\mathcal{O} _{k}}$ be the generalized factorial in $\O_k$ in the sense of Bhargava \cite{Bhargava1} (see subsection \ref{sectionporderings}). We remark that $m! _{k}$ is an ideal of $\O_k$, not a number. By \cite[Proposition 2.6]{BFS2017}, a set $\mathcal S\subset \O_k$ of size $n+1$ is $n$-optimal if and only if 
\begin{equation}\label{e-volume}\Vol(\mathcal S)=\prod_{m=0}^{n}m!_k^2.
\end{equation}
Also by \cite[Proposition 2.6]{BFS2017} for every subset $F\subset \O_k$ of size $n+1$ we have \begin{equation}\label{e-volumeB2}N_{k/\Q}(\Vol(F))\geq\prod_{m=0}^{n}N_{k/\Q}(m!_k^2).\end{equation}
\begin{lemma}\label{l-PointBound}Let $\mathcal S\subset \O_k$ be an $n$-optimal set. Then for every $x\in \mathcal{S} $ we have 
$$\sum_{y\in \mathcal S\setminus\left\{x\right\}}\log \|x-y\|\leq \log N_{k/\Q}(n!_k)\leq n\log n + A_1 n,$$
\end{lemma} where $A_1\geq 1$ is a constant 
depending only on $k$.
\begin{proof}
By (\ref{e-volumeB2}) we have 
$$\log N_{k/\Q}(\Vol(\mathcal S))-2\sum_{y\in \mathcal S\setminus\left\{x\right\}}\log \|x-y\|=\log N_{k/\Q}(\Vol(\mathcal S\setminus \left\{x\right\}))\geq  \sum_{m=0}^{n-1}\log N_{k/\Q}(m!_k)^2.$$ Using formula (\ref{e-volume}) we get $\sum_{y\in \mathcal S\setminus\left\{x\right\}}\log \|x-y\|\leq \log N_{k/\Q}(n!_k).$ The second inequality in the lemma follows from \cite[Theorem 1.2.4]{Lam}.
\end{proof} We immediately get:
\begin{corollary}\label{c-DifferenceBound}
Let $\mathcal S$ be an $n$-optimal set. Then for every $x\neq y\in \mathcal S$ we have 
$\log \|x-y\|\leq n\log n + A_1 n.$
\end{corollary}
\begin{remark}A posteriori we know that the bound in the above corollary is very far off but it will be used to ensure that the quadruple-logarithmic error term from Proposition \ref{p-CountingMain} is negligible.\end{remark}
\subsection{Proof of Theorem \ref{t-OptimalFrame}}\label{s-ProofOptimalFrame}
 Our first goal is to give an upper bound on the norms of differences of pairs of elements in hypothetical $n$-optimal sets. 

\begin{lemma}\label{l-LogSumLB}
Let $B\in\R$, let  $F$ be a finite subset of $\O_k$ and let $x,y\in F$ be such that $\log|F| \leq \log \|x-y\|+B$. Then for every $0<\log X\leq \log\|x-y\|+B$ we have 
\begin{align*}\sum_{z\in F\setminus\left\{x,y\right\}}&(\log\|(x-z)(y-z)\|)\geq 2|F|\log X- \frac{2\Theta_1\|x-y\|^{-\kappa}}{1+\kappa}(X^{1+\kappa}-1) \\
& -\frac{2\Theta_2}{2d-1}(\log X)^{2d-1}-2\Theta_3\logp\logp\logp\log\|x-y\|\log X -2\Theta_4 \log X,
\end{align*} where $\kappa$ is as in Proposition \ref{p-CountingMain} and the constants $\Theta_i$ depend only on $k$ and $B$ .
\end{lemma}
\begin{proof} By translating $F$ if necessary we can assume that $x=0$. Put $a=y$. Then the leftmost sum takes the form 
$$\sum_{z\in F\setminus\left\{0,a\right\}}\log\|z(a-z)\|.$$ For $s\geq 1$ let $\mathcal{E}(a,s)=\left\{z\in \O_k\setminus\left\{0,a\right\}| \|z(a-z)\|\leq s^2\right\}.$ 
We have  
\begin{align*}\sum_{z\in F\setminus\left\{0,a\right\}}\log\|z(a-z)\| =& \sum_{z\in F\setminus\left\{0,a\right\}}\left(2\log X-\int_{\|z(a-z)\|}^{X^2}\frac{dt}{t}\right)\\
\geq& 2(|F|-2)\log X-\sum_{z\in F\setminus\left\{0,a\right\}}\int_1^{X^2} \mathbf{1}_{\mathcal{E}(a,t^{1/2})}(z)\frac{dt}{t}\\
=& 2(|F|-2)\log X-2\int_1^{X} |\mathcal{E}(a,t)\cap F|\frac{dt}{t}\\
\geq& 2(|F|-2)\log X-2\int_1^{X} |\mathcal{E}(a,t)|\frac{dt}{t}\\
\geq& 2(|F|-2)\log X-2\int_1^X\left(\Theta_1 t^{1+\kappa}\|a\|^{-\kappa }+\Theta_2(\log t)^{2d-2}\right.\\ 
& \left.+\Theta_3\logp\logp\logp\log\|a\|+\Theta_4+4\right)\frac{dt}{t}.
\end{align*}
The last inequality is an application of Theorem \ref{t-MainPropCount}. Upon evaluating the integral we get the desired inequality.
\end{proof}
\begin{lemma}\label{l-BallsnOptimal}
There exists a constant $\Theta_7$, dependent only on $k$, such that for every sufficiently large $n$ and every $n$-optimal set $\mathcal S$ we have $\log\|x-y\|\leq \log n+\Theta_7$ for every $x\neq y\in \mathcal S$. 
\end{lemma}
\begin{proof}
Let $n>0$ and let $A_1$ be the constant from Lemma \ref{l-PointBound}. We recall that $A_1\geq1$ and it depends only on $k$. Let $x\neq y\in \mathcal S$. Either $\log\|x-y\|\leq \log n+A_1$ or we apply  Lemma \ref{l-LogSumLB} with $F=\mathcal S$, $\log X=\log n+ 2A_1$ and $B= A_1 $.  In the latter case we get 
\begin{align*}\sum_{z\in \mathcal S\setminus\left\{x,y\right\}}\left(\log\|z-x\|+\log\|z-y\|\right)\geq 2(n+1)(\log n +2A_1)-\frac{2\Theta_1\|x-y\|^{-\kappa}}{1+\kappa}(X^{\kappa+1}-1)\\ -\frac{2\Theta_2}{2d-1} (\log X)^{2d-1}-2\Theta_3\logp\logp\logp\log\|x-y\|\log X
- 2\Theta_4\log X,
\end{align*}
where the constants $\Theta _i$ and $\kappa$ depend only on $k$. By Corollary \ref{c-DifferenceBound}, we have $\log \|x-y\|\leq n\log n+ A_1n$ so $2\Theta_3\logp\logp\logp\log\|x-y\|\log X=o(n)$. Similarly $\frac{2\Theta _2}{2d-1}(\log X)^{2d-1}=o(n)$ and $2\Theta _4\log X=o(n)$. Hence, for $n$ sufficiently large we have
\begin{align*}\sum_{z\in \mathcal S\setminus\left\{x,y\right\}}\left(\log\|z-x\|+\log\|z-y\|\right)\geq& 2n(\log n +2A_1)-\frac{2\Theta_1\|x-y\|^{-\kappa}}{1+\kappa}(ne^{2A_1})^{\kappa+1}-o(n)\\
\geq& 2n\log n +3nA_1-\frac{2\Theta_1\|x-y\|^{-\kappa}}{1+\kappa}(ne^{2A_1})^{\kappa+1}. 
\end{align*} 
By Lemma \ref{l-PointBound}, we get 
$$2n\log n+2A_1n\geq \sum_{z\in \mathcal S\setminus\left\{x,y\right\}}\left(\log\|z-x\|+\log\|z-y\|\right)+2\log\|x-y\|.$$ Of course $\log\|x-y\|\geq 0$ so we deduce that
\begin{align*}2n\log n + 2A_1n\geq & 2n\log n+3n A_1 -\frac{2\Theta_1\|x-y\|^{-\kappa}}{1+\kappa}(ne^{2A_1})^{\kappa+1},\\
\frac{2\Theta_1\|x-y\|^{-\kappa}}{1+\kappa}(ne^{2A_1})^{\kappa+1}\geq & A_1 n.
\end{align*}
Put
\begin{equation*}\Theta_6:=\frac{2\Theta_1e^{2A_1+2\kappa A_1}}{A_1(1+\kappa)}.\end{equation*}
We have $\Theta _6 \geq \|x-y\|^\kappa n^{-\kappa}$ and $\log \Theta_6\geq \kappa(\log \|x-y\|-\log n)$. Hence, for $n$ sufficiently large $\log\|x-y\|\leq \log n+\kappa^{-1}\log \Theta_6$ where $\Theta_6$ depends only on $k$. The lemma holds with $\Theta_7=\max\left\{\kappa^{-1}\log \Theta_6, A_1\right\}$. The constant $\Theta_7$ depends only on $k$.
\end{proof}
The second ingredient in the proof of Theorem \ref{t-OptimalFrame} is the following weaker version of Theorem \ref{t-OptimalFrame}.
\begin{lemma}\label{l-99prcfram}For every $\delta>0$ there exists a constant $\Theta_8=\Theta_8(\delta)$, dependent only on $k$ and $\delta$, such that for every $n$ sufficiently large and every $n$-optimal set $\mathcal S$ there exists a cylinder $\mathcal{C}_1$ of volume at most $n \Theta_8$ such that $|\mathcal{S}\cap \mathcal C_1|\geq (1-\delta)n.$
\end{lemma}
\begin{proof}
Assume $n>0$. We shall crucially use Proposition \ref{p-CountingMain} (2) together with Lemma \ref{l-BallsnOptimal}. In order to use Proposition \ref{p-CountingMain} (2) we fix a good fundamental domain $\mathcal F$ of $\O_k^\times$ in $V^\times$, a basis $\xi_1,\ldots, \xi_{d-1}$ of a maximal torsion free subgroup of $\O_k^\times$ and the associated norm $\|\cdot\|_\infty$ on $\O_k^\times$. Put $A_2=\gamma_\Q-\gamma_k-2$. First note that by the energy formula \cite[Corollary 5.2]{BFS2017} for large enough $n$ we have 
$$\log(N_{k/\Q}(\Vol(\mathcal S)))=\sum_{x\neq y\in \mathcal S} \log \|x-y\|=n^2\log n +n^2(\gamma_\Q-\gamma_k-\frac{3}{2})+o(n^2)\geq (n^2+n)(\log n +A_2).$$ Together with Lemma \ref{l-BallsnOptimal} this implies that there exists at least one pair $x,y\in \mathcal S$ such that $\log n+ A_2\leq \log\|x-y\|\leq \log n+\Theta_7$. Let us fix a pair $x_0,y_0$ with $\|x_0-y_0\|$ maximal among all pairs in $\mathcal S$. By translating $\mathcal S$ if necessary we may assume that $x_0=0$ and put $a=y_0$. Let $X=\|a\|$. Then $\log n+ A_2\leq \log X\leq \log n+\Theta_7$. The question is invariant under multiplying $\mathcal S$ by elements of $\O_k^\times$ so we may assume without loss of generality that $a\in \mathcal F$.  For every $z\in\mathcal S$ we have $\|z\|\leq X$ and $\|a-z\|\leq X$ so  $\|z(a-z)\|\leq X^2$. Therefore, with notation from Proposition \ref{p-CountingMain} we have \begin{equation}\label{e-Containment}\mathcal S\setminus\left\{0\right\}\subseteq \left\{x\lambda^{-1}| (x,\lambda)\in S(a,X)\right\}.
\end{equation}

Let $M>0$ be such that Proposition \ref{p-CountingMain} (2) holds with $\varepsilon=\frac{\delta}{2}e^{-\Theta_7}$ and $B=0$. The constant $M$ depends only on $k$ and $\delta$.  For $n$ sufficiently large we have 
\begin{align*}|S(a,X)[M]|\leq& \frac{\delta e^{-\Theta_7}}{2}X^{1+\kappa}\|a\|^{-\kappa}+\Theta_2(\log X)^{2d-2}+\Theta_3 \logp\logp\logp\log\|a\|+\Theta_4\\
\leq & \frac{\delta}{2} n + o(n)\leq  \delta n.
\end{align*}
Let $\mathcal S':=\mathcal S\setminus \left\{(x\lambda^{-1}| (x,\lambda)\in S(a,X)[M]\right\}$. By the inequality above, $\mathcal S'$ contains at least $(1-\delta)n$ elements. To prove the lemma, it is enough to show that  $\mathcal S'$ is contained in a cylinder of volume at most $n\Theta_8$. By (\ref{e-Containment}) we have $\mathcal S'\subseteq \left\{(x\lambda^{-1}| (x,\lambda)\in S(a,X)\setminus S(a,X)[M]\right\}\cup\left\{0\right\}.$
$$ S(a,X)\setminus S(a,X)[M]\subseteq \left\{(x,\lambda)\in (\mathcal{F}\cap \O_k) \times \O_k^{\times}| \ \  \|x\|\leq X, \|\lambda\|_\infty\leq M\right\}.$$ 
By Lemma \ref{l-cone}, we have a constant $C_0>0$ such that $C_0^{-1}\|x\|^{1/N}\leq |x|_i\leq C_0\|x\|^{1/N}$ for every $x\in \mathcal F$ an every $i=1,\ldots,d$.  Let $C_{21}:=\max_{\|\lambda\|_\infty\leq M}\max_{i=1,\ldots,d}|\lambda^{-1}|_i$. Therefore, for every $(x,\lambda)\in S(a,X)\setminus S(a,X)[M]$ and $i=1,\ldots,d$ we have  $|x\lambda^{-1}|_i\leq C_{21}C_0\|x\|^{1/N}\leq C_{21}C_0 X^{1/N}\leq C_{21}C_0e^{\Theta_7/N} n^{1/N} $. It follows that $\mathcal{S}'$ is contained in the cylinder
$$\mathcal C_1=B_{\R}(0,C_{21}C_0e^{\Theta_7/N}n^{1/N})^{r_1}\times  B_{\C}(0,C_{21}C_0e^{\Theta_7/N}n^{1/N})^{r_2}.$$
The volume of $\mathcal C_1$ is $n 2^{r_1}\pi^{r_2}e^{\Theta_7}C_0^NC_{21}^N=: n\Theta_8$ where $\Theta_8$ depends only on $k$ and $\delta$.
\end{proof}
We are ready to prove Theorem \ref{t-OptimalFrame}
\begin{proof}[Proof of Theorem \ref{t-OptimalFrame}]
Assume that $n$ is sufficiently large so that Lemma \ref{l-BallsnOptimal} holds and Lemma \ref{l-99prcfram} holds with $\delta=1/100$ and $\Theta_8=\Theta_8(1/100)$. Also for technical reasons we require $n\geq 4d, \Theta_8\geq 1$ and the constant $C_{11}$ from Lemma \ref{l-Cylinder} satisfies $C_{11}\geq 1$. This is not a problem since they can be always replaced by a bigger constants as long as these constants depend only on $k$. Let $\mathcal S$ be an $n$-optimal set and let $\mathcal{C}_1$ be a cylinder of volume $n\Theta_8$ containing at least $\frac{99n}{100}$ points of $\mathcal S$. Write 
$$\mathcal C_1=\prod_{i=1}^{r_1}B_{\R}(t_i,R_i)\times \prod_{i=r_1+1}^d B_{\C}(t_i,R_i)$$ 
with $t=(t_1,\ldots,t_d)\in V$. Note that  $2^{r_1}\pi^{r_2}\prod_{i=1}^{r_1}R_i\prod_{i=r_1+1}^d R_i^2=n\Theta_8.$ For a positive constant  $A>0$ (how big will be precised later) we put 
$\mathcal C_1^A=\prod_{i=1}^{r_1}B_{\R}(t_i,AR_i)\times \prod_{i=r_1+1}^d B_{\C}(t_i,AR_i)$. The idea of the proof is to show that for large $A$ (how large depends only on $k$) and every $y\not\in \mathcal C_1^A$ the intersection $\mathcal C_1\cap \left\{x\in V|\|x-y\|\leq ne^{\Theta_7}\right\}$ is too small to contain $\frac{99n}{100}$ points of $\mathcal S$. Then from Lemma \ref{l-BallsnOptimal} and Lemma \ref{l-99prcfram} we can deduce that $y\not \in \mathcal S$ and consequently that $\mathcal{S}\subset \mathcal C_1^A$. 

Let $C_{11}$ be the constant from Lemma \ref{l-Cylinder} and put  
\begin{equation*}A=\max\left\{2,e^{\Theta_7}(2N)^N 2^{r_1}\pi^{r_2}\Theta_8^{N-1} C_{11}^N+1\right\}. 
\end{equation*} Suppose that $y\in \mathcal S\setminus \mathcal C_1^A$. Since $y\not\in \mathcal C_1^A$ for every $x\in \mathcal C_1$ there exists a coordinate $i\in\left\{1,\ldots,d\right\}$ such that $|x-y|_i\geq (A-1)R_i$. Put $\iota=1$ if $i\in\left\{1,\ldots,r_1\right\}$ and $\iota=2$ otherwise. If additionally $\|x-y\|\leq ne^{\Theta_7}$, then we have 
\begin{align}\label{e-IneqCyl1}
\prod_{j=1,j\neq i}^{r_1}|x-y|_j\prod_{j=r_1+1,j\neq i}^d|x-y|_j^2\leq& ne^{\Theta_7}R_i^{-\iota}(A-1)^{-\iota}\\
\leq & (A-1)^{-\iota}\frac{e^{\Theta_7}2^{r_1}\pi^{r_2}}{\Theta_8}\prod_{j=1,j\neq i}^{r_1}R_j\prod_{j=r_1+1,j\neq i}^d R_j^2\\
\leq& \prod_{j=1,j\neq i}^{r_1}\frac{R_j}{2N\Theta_8C_{11}}\prod_{j=r_1+1,j\neq i}^d \left(\frac{R_j}{2N\Theta_8C_{11}}\right)^2. \label{e-IneqCyl2}
\end{align}
Hence, there exists $j\neq i$ such that $|x-y|_j\leq \frac{R_j}{2N\Theta_8C_{11}}$. Define
$$\mathcal C_1(j)=\prod_{l=1,l\neq j}^{r_1}B_{\R}(t_l,R_l)\times\prod_{l=r_1+1,l\neq j}^d B_{\C}(t_l,R_l)\times B_{k_{j}}\left(y_j,\frac{R_j}{2N\Theta_8C_{11}}\right)$$
and note that $\Leb(\mathcal C_1(j))=\frac{n\Theta_8}{(2N\Theta_8 C_{11})^{[k_{j}:\R]}}\leq \frac{n}{2N C_{11}}.$ From inequalities (\ref{e-IneqCyl1}-\ref{e-IneqCyl2}) we deduce that 
\begin{equation}\label{e-IneqCyl3}
\left\{x\in \mathcal C_1|\, \|x-y\|\leq ne^{\Theta_7}\right\}\subset \bigcup_{l=1}^d \mathcal C_1(l).
\end{equation}
By (\ref{e-IneqCyl3}) and Lemma \ref{l-Cylinder}, we get 
\begin{equation}\label{e-UpperBIntersesction}
|\left\{x\in \mathcal C_1\cap \O_k|\, \|x-y\|\leq ne^{\Theta_7}\right\}|\leq  d+ C_{11}\frac{dn}{2N C_{11}}\leq d+\frac{n}{2}\leq \frac{3n}{4}.
\end{equation}
Lemma \ref{l-BallsnOptimal} yields 
$\mathcal S\subseteq \left\{x\in\O_k|\, \|x-y\|\leq ne^{\Theta_7}\right\}$ so we have \begin{equation*}\mathcal S\cap \mathcal C_1\subseteq \left\{x\in \mathcal C_1\cap \O_k | \|x-y\|\leq ne^{\Theta_7}\right\}. 
\end{equation*} In particular $| \left\{x\in \mathcal C_1\cap \O_k|\, \|x-y\|\leq ne^{\Theta_7}\right\}|\geq |\mathcal S\cap \mathcal C_1|\geq \frac{99}{100}n$. This contradicts (\ref{e-UpperBIntersesction}).  Thereby we showed that for $n$ big enough, the set $\mathcal S\setminus \mathcal C_1^A$ is empty, that is $\mathcal S\subset \mathcal C_1^A$. The volume of $\mathcal C_1^A$ is $\Theta_8 A^Nn$ so by taking into account the cases of small $n$ we see that $\mathcal S$ is contained in a cylinder of volume $\Theta_5n$, where $\Theta_5$ depends only on $k$. Theorem \ref{t-OptimalFrame} is proven.
\end{proof} 
\section{Collapsing of measures}\label{s-Collapsing}
Write $\mathcal M^1(V),(\mathcal P^1(V))$ for the set of finite measures (resp. probability measures) $\nu$ on $V$ which are absolutely continuous with respect to the Lebesgue measure such that the density ${d\nu}/{d\Leb}$  is almost everywhere less or equal to $1$. For $i\in \left\{1,\ldots,d\right\}$ and $v_i\in \R$ or $\C$ (depending on whether $i$ corresponds to real or complex place) we will define an operation called collapsing $c_{i,v_i}:\mathcal M^1(V)\to \mathcal M^1(V)$ that has the following property: either $I(c_{i,v_i}(\nu))<I(\nu)$ or $\nu$ is of a very specific form. It is a version of the Steiner symmetrization (\cite{Krantz1999}), but for measures in $\mathcal M^1(V)$ instead of subsets of $V$. We shall make it precise in a moment. The operation of collapsing is the continuous analogue of the collapsing operation on subsets of $\O_k$ used in \cite{PV2013} and \cite{BFS2017} where it was defined for $k$ quadratic imaginary.  We remark that for number fields $k$ other than quadratic imaginary there is no reasonable discrete collapsing procedure for subsets of $\O_k$.  In this section we study the effect of collapsing on the energy of measures. Our goal is Corollary \ref{c-MinimalCollapsing} which says that the measures $\nu$ in $\mathcal P^1(V)$ that minimize the energy $I(\nu)$ are, up to translation, invariant under all collapsing operations.
\begin{definition}\label{d-Collapsing}
Let  $i\in \left\{1,\ldots,d\right\}$ and $v_i\in \R$ if $i\in\left\{1,\ldots, r_1\right\}$ or $v_i\in \C$ otherwise. Let $\nu\in \mathcal{M}^1(V)$ be a measure with density $f\in L^1(V)$. For $x=(x_1,\ldots,x_d)\in V$ define 
\begin{equation}
F_i(x):=\begin{cases} \frac{1}{2}\int_{\R}f(x_1,\ldots,x_{i-1},t,x_{i+1},\ldots,x_d)dt & \textrm{ if } i\in\left\{1,\ldots,r_1\right\}\\
\frac{1}{\sqrt{\pi}}\left(\int_{\C}f(x_1,\ldots,x_{i-1},t,x_{i+1},\ldots,x_d)dt\right)^{1/2} & \textrm{ if } i\in\left\{r_1+1,\ldots,d\right\}.
\end{cases}
\end{equation}
Let $h_i:V\to \R_{\geq 0}$ be given by 
\begin{equation}
h_i(t_1,\ldots,t_d):=\begin{cases} \mathbf 1_{B_{\R}(v_i,F_i(t_1,\ldots,v_i,\ldots,t_d))}(t_i) & \textrm{ if } i\in\left\{1,\ldots,r_1\right\}\\
 \mathbf 1_{B_{\C}(v_i,F_i(t_1,\ldots,v_i,\ldots,t_d))}(t_i) & \textrm{ if } i\in\left\{r_1+1,\ldots,d\right\}.
\end{cases}
\end{equation}
The collapsed measure $c_{i,v_i}(\nu)$ is given by the density $h_i$. By construction $c_{i,v_i}(\nu)$ is symmetric with respect to the subspace $\left\{x=(x_1,\ldots,x_d)\in V| x_i=v_i\right\}.$
\end{definition}
Let $V^i:=\{x=(x_1,\ldots,x_d)\in V| x_i=0\}.$ Collapsing is closely related to the Steiner symmetrization in the following way. If $V=\R^d$, then for any measurable subset $E\subseteq V$ we have $c_{i,0}(\Leb|_E)=\Leb|_{{\rm St}_i(E)}$ where ${\rm St}_i(E)$ is the Steiner symmetrization of $E$ with respect to the hyperplane $V^i$ (c.f. \cite{Krantz1999}).
For further use we introduce a symmetric bilinear form on $\mathcal M^1(V)\times \mathcal M^1(V)$ 
\begin{equation}\label{e-Bilin}\langle \nu,\nu'\rangle =\int_V\int_V \log\|x-y\|d\nu(x)d\nu'(y).
\end{equation} The integral converges as soon as $\nu,\nu'$ are finite signed measures with bounded density and with compact support.
\begin{definition}\label{d-Energy} Let $\nu$ be a compactly supported measure of bounded density. We define the energy of $\nu$ as $I(\nu)=\langle \nu,\nu\rangle.$  
\end{definition}We will also need a modified version of the bilinear form $\langle\cdot,\cdot\rangle$ defined as $$\langle \nu_1,\nu_2\rangle_\delta:=\int_V\int_V\frac{1}{2}\log(\|x-y\|^2+\delta^2)d\nu_1(x) d\nu_2(y) \textrm{ for } \delta\geq 0.$$ Note that  $\langle \nu_1,\nu_2\rangle_0=\langle \nu_1,\nu_2\rangle.$ 
\begin{lemma}\label{l-WeakContinuity}
Let $\Sigma$ be a compact subset of $\R$ or $\C$. The map $\mathcal M^1(\Sigma)\times \mathcal M^1(\Sigma)\ni (\nu_1,\nu_2)\to \langle \nu_1,\nu_2\rangle_\delta\in \R$ is continuous with respect to weak-* topology.
\end{lemma}
\begin{proof}
Let $f_1,f_2$ be the densities of $\nu_1,\nu_2$ respectively. We have 
$$\langle \nu_1,\nu_2\rangle_\delta=\frac{1}{2}\int_{\Sigma}\int_{\Sigma}f_1(x)f_2(y)\log(|x-y|^2+\delta^2)dxdy.$$ We claim that the map $(\nu_1,\nu_2)\mapsto f_1\times f_2\in L^2(\Sigma\times\Sigma)$ is weakly-* continuous on $\mathcal M^1(\Sigma)\times \mathcal M^1(\Sigma)$ with the weak topology on $L^2(\Sigma\times\Sigma)$. We sketch the proof. Let $h\in L^2(\Sigma\times \Sigma), \varepsilon>0$ and let $(\nu_1^i\times \nu_2^i)$ be a weakly-* convergent sequence in  $\mathcal M^1(\Sigma)\times \mathcal M^1(\Sigma)$, with densities $f_1^i\times f_2^i$. By Luzin's theorem there exists a subset $E\subseteq \Sigma\times\Sigma$, such that $\Leb(\Sigma\times\Sigma\setminus E)<\varepsilon$ and $h$ is continuous on $E$. Using the Radon property of the Lebesgue measure we can assume that $E$ is compact. We have 
\[ \int_\Sigma\int_\Sigma f_1^i(x)f_2^i(y)h(x,y)dxdy=\int_E f_1^i(x)f_2^i(y)h(x,y)dxdy + \int_{\Sigma\times\Sigma\setminus E}f_1^i(x)f_2^i(y)h(x,y)dxdy.\]
The first term on right hand side converges because $h(x,y)$ is continuous in $E$. Second term can be always bounded using the Cauchy-Schwartz inequality \[\left|\int_{\Sigma\times\Sigma\setminus E}f_1^i(x)f_2^i(y)h(x,y)dxdy\right|\leq \int_{\Sigma\times\Sigma\setminus E}|h(x,y)|dxdy\leq \varepsilon^{1/2}\|h\|_2.\] 
Taking $\varepsilon\to0$ we see that the inner product of $f_1^i\times f_2^i$ with $h$ always converges. This proves the continuity.
The function $(x,y)\mapsto \frac{1}{2}(\log|x-y|^2+\delta^2)$ is in $ L^2(\Sigma\times\Sigma)$ so the map $ (\nu_1,\nu_2)\mapsto \langle \nu_1,\nu_2\rangle_\delta$ is weakly-* continuous on $\mathcal M^1(\Sigma)\times \mathcal M^1(\Sigma)$. 
\end{proof}
\begin{definition} Let $\nu\in \mathcal M^1(\R)$. We define the "pushed" measure $s_t^-(\nu)\in \mathcal M^1(\R), t\in \R$ by 
\begin{align*}s_t^-(\nu):=\Leb|_{[t-\nu((-\infty,t]),t]}+\nu|_{[t,+\infty)}.\end{align*}
\end{definition}
\begin{lemma}\label{l-CollapsingPairsR}
Let $\nu_1,\nu_2\in \mathcal M^1(\R)$ be non-zero compactly supported measures, let $x\in \R$ and let $\delta\geq 0$. Then $\langle c_{1,x}(\nu_1),c_{1,x}(\nu_2)\rangle_\delta \leq \langle \nu_1,\nu_2\rangle_\delta $ and equality holds if and only if there exists $y\in \R$ such that $\nu_1,\nu_2$ are restrictions of the Lebesgue measure to intervals centered in $y$.
\end{lemma}
\begin{proof} 
Let $[-R,R]$ be an interval supporting both $\nu_1,\nu_2$.  Let $m_i=\nu_i(\R)$. The lemma is equivalent to the following statement: $$\langle \nu_1,\nu_2\rangle _{\delta }\geq \langle \Leb|_{[-m_1/2,m_1/2]},\Leb|_{[-m_2/2,m_2/2]}\rangle_\delta$$ and the equality holds if and only if $\nu_1,\nu_2$ are Lebesgue measures restricted to concentric intervals. 

Write $k(z)=\frac{1}{2}\log(z^2+\delta^2)$ for $z\in \R$. Let $f_i\in L^1(\R)$ be the density of $\nu_i$ and let $F_i(t):=\nu_i((-\infty,t])$. Put $P_i(x)=\int_\R k(x-y)d\nu_i(y).$ By Lemma \ref{l-WeakContinuity}, we may assume that $\langle \nu_1,\nu_2\rangle _{ \delta }$ is minimal among all pairs of measures $\nu'_1,\nu'_2$ subject to conditions $\supp( \nu_i')\subseteq [-R,R], \nu_i'(\R)=m_i$ and with densities at most $1$. We want to show that $\nu_1,\nu_2$ must be Lebesgue measures restricted to concentric intervals. 
Let $a_i:=\sup\{x\in \R|\: \nu_i((-\infty,x))=0\}, b_i:=\inf\{x\in \R|\: \nu_i((x,+\infty))=0\}, t_i:=\sup\{t\geq a_i|\: \nu_i([a_i,t])=t-a_i\}$ and $c_i:=\frac{1}{2}(a_i+t_i)$. Interval $[a_i,b_i]$ is the smallest interval containing the support of $\nu_i$ and $\nu_i$ restricted to $[a_i,t_i]$ is the Lebesgue measure. By minimality assumption, for every $\varepsilon>0$ we have \begin{align*}\langle s^-_{t_1+\varepsilon}(\nu_1)-\nu_1,\nu_2\rangle _{\delta } &\geq 0\\
\langle s^-_{t_2+\varepsilon}(\nu_2)-\nu_2,\nu_1\rangle _{ \delta } &\geq 0\end{align*} 
We claim that this forces $c_1=c_2$. For the sake of contradiction assume $c_1<c_2$. 
Consider two cases.\\
\textbf{Case 1}. Assume $a_{1}=t_{1}$. We want to show that $\langle s_{t_{1}+ \varepsilon}^{-}( \nu _{1}) - \nu _{1}, \nu _{2} \rangle _{\delta } <0$ which leads to the contradiction. We have 
\begin{align*} 
s_{t_{1}+ \varepsilon}^{-}( \nu _{1})- \nu_{1}=& \left( s_{ t_{1}+ \varepsilon}^{-}( \nu_{1}) - \nu _{1} \right) \mid _{[a_{1}, t_{1}+ \varepsilon - F_{1}(t_{1}+ \epsilon)]}+\left( s_{ t_{1}+ \varepsilon } ^{-}( \nu _{1})- \nu _{1} \right) \mid_{t_{1} + \varepsilon - F_{1}( t_{1} + \varepsilon), t_{1}+ \varepsilon]}\\=& - \nu _{1} \mid _{[a_{1}, t_{1}+ \varepsilon-F_{1}(t_{1}+ \varepsilon)]}+ \left( \Leb - \nu _{1} \right) \mid_{[t_{1} + \varepsilon - F_{1}( t_{1}+\epsilon ), t_{1}+ \epsilon ]}  
\end{align*} 
and 
\begin{equation*} 
\langle s_{t_{1}+ \varepsilon}^{-}( \nu _{1}) - \nu _{1}, \nu _{2} \rangle _{\delta} = \int _{t_{1}+ \varepsilon - F_{1}( t_{1}+ \varepsilon ) }^{t_{1}+ \varepsilon } P_{2}(x) d( \Leb - \nu _{1})(x)- \int_{a_{1}}^{t_{1}+ \varepsilon - F_{1}(t_{1}+ \epsilon )} P_{2}(x) d \nu _{1}(x).   
\end{equation*} 
Compute  
\begin{equation*} 
\int_{ t_{1} + \varepsilon - F_{1}(t_{1}+ \varepsilon )}^{ t_{1}+ \varepsilon} d( \Leb - \nu _{1} )(x)=F_{1}\left(t_{1}+ \varepsilon - F_{1}(t_{1}+ \varepsilon )\right)= \int _{a_{1}}^{t_{1}+ \varepsilon - F_{1}(t_{1}+ \varepsilon)} d \nu _{1} (x).  
\end{equation*} 
Let $\varepsilon $ be such that $a_{1} + \varepsilon < c_{2} $. Since the function $P_{2}(x)$ is strictly decreasing on the interval $[a_{1}, a_{1}+\varepsilon ]$ we have $\langle s_{t_{1}+ \varepsilon }^{-}(\nu _{1})- \nu _{1}, \nu _{2} \rangle _{\delta } < 0$.\\       
\textbf{Case 2.} Assume $a_{1}\neq t_{1}$. We can assume that $\varepsilon $ is small enough to satisfy $F_{1}(t_{1}+ \varepsilon ) > \varepsilon $. We have 
$$I_0:=\langle s^-_{t_1+\varepsilon}(\nu_1)-\nu_1,\nu_2\rangle _{\delta }=\int_{t_1}^{t_1+\varepsilon}(1-f_1(x))P_2(x)dx- \int_{a_1}^{t_1+\varepsilon-F_1(t_1+\varepsilon)} P_2(x)dx.$$ We can decompose $P_2(x)=\int_{a_2}^{t_2}k(y-x) dy+\int_{t_2}^{b_2}k(y-x)f_2(y)dy.$ Hence $I_0=I_1+I_2$ where \begin{align*}
I_1:=&\int_{t_1}^{t_1+\varepsilon} \int_{a_2}^{t_2}(1-f_1(x))k(y-x) dydx-\int_{a_1}^{t_1+\varepsilon-F_1(t_1+\varepsilon)} \int_{a_2}^{t_2} k(y-x) dydx\\
I_2:=&\int_{t_1}^{t_1+\varepsilon} \int_{t_2}^{b_2}(1-f_1(x))f_2(y) k(y-x) dydx-\int_{a_1}^{t_1+\varepsilon-F_1(t_1+\varepsilon)} \int_{t_2}^{b_2}f_2(y) k(y-x) dydx. 
\end{align*}
We sketch the computation showing that $I_0$ is negative for small enough $\varepsilon$. Let $K(t):= \int _{0}^{t} k(x) dx$. In what follows, $o(1)$ will mean a quantity converging to $0$ as $\varepsilon\to 0$. Rewrite $I_1$ as 
\begin{align*}
I_1=& \int_{t_1}^{t_1+\varepsilon} (1-f_{1}(x))(K(t_2-x)-K(a_2-x))dx- \int_{a_1}^{t_1+\varepsilon-F_1(t_1+\varepsilon)} (K(t_2-x)-K(a_2-x))dx\\=& \int_{t_1}^{t_1+\varepsilon} (1-f_{1}(x))(K(t_2-t_1)-K(a_2-t_1)+o(1))dx \\-&\int_{a_1}^{t_1+\varepsilon-F_1(t_1+\varepsilon)} (K(t_2-a_1)-K(a_2-a_1)+o(1))dx\\ =&(\varepsilon + F_{1}(t_{1})-F_{1}(t_{1}+ \varepsilon ) )(K(t_2-t_1)+K(a_2-a_1)-K(a_2-t_1)-K(t_2-a_1)+o(1)).
\end{align*}
We have \[K(t_2-t_1)+K(a_2-a_1)-K(a_2-t_1)-K(t_2-a_1)=\int_{a_2-t_1}^{t_2-t_1}k(x)dx -\int_{a_2-a_1}^{t_2-a_1}k(x)dx.\]
The function $k$ is  symmetric and increasing in $|x|$. It follows that $z\mapsto \int_{a_2-z}^{t_2-z}k(x)dx$ is increasing in $|c_2-z|$ unless $a_2=t_2$ in which case it is $0$. As $c_1<c_2$, we have $|c_2-t_1|<|c_2-a_1|$, so  $K(t_2-t_1)+K(a_2-a_1)-K(a_2-t_1)-K(t_2-a_1)<0$ unless $a_2=t_2$. 
Therefore, we either have  $a_2=t_2$ and $I_1=0$ or $I_1<0$ for $\varepsilon$ small enough. 

For $I_2$ we have 
\begin{align*}
I_2=&\int_{t_1}^{t_1+\varepsilon} (1-f_1(x))\int_{t_2}^{b_2}f_2(y)(k(y-t_1)+o(1))dydx\\-&\int_{a_1}^{t_1+\varepsilon-F_1(t_1+\varepsilon)} \int_{t_2}^{b_2}f_2(y)(k(y-a_1)+o(1))dydx.
\end{align*}
As soon as $c_1< c_2$, for $y\in [t_2,b_2]$ we have $|y-t_1|<|y-a_1|$ so $k(y-a_1)\geq k(y-t_1)+\kappa$ for some $\kappa>0$ and all $y\in [t_2,b_2]$. Therefore $I_2\leq (\varepsilon +F_{1}(t_{1})-F_1(t_1+\varepsilon))\nu_2([t_2,b_2])(-\kappa+o(1))$ and we deduce that either $\nu_2([t_2,b_2])=0$ and $I_2=0$ or $I_2<0$ for small enough $\varepsilon$.

Since $\nu_2([a_2,b_2])=m_2>0$, the cases $a_2=t_2$ and $\nu_2([t_2,b_2])=0$ are mutually exclusive.  In particular, for small $\varepsilon$ we would have $I_0<0$, contradicting the minimality of $\langle \nu_1,\nu_2\rangle _{ \delta }.$ We proved that $c_1=c_2$.
 
We proceed to show that $t_2=b_2$ which is equivalent to  $\nu_2([t_2,b_2])=0$. Assume that $\nu_2([t_2,b_2])>0$. Let us keep the notation for $K,I_0,I_1,I_2$ from the previous step. \textit{Case 1.} Assume $a_{1}\neq t_{1}$. We already know that $c_1=c_2$, so $K(t_2-t_1)+K(a_2-a_1)-K(a_2-t_1)-K(t_2-a_1)=0$ and consequently $I_1=o(\varepsilon -F_1(t_1+\varepsilon)+F_1(t_1)).$ 
For $y>t_2$ we will have $|y-t_1|<|y-a_1|,$ so $$\int_{t_2}^{b_2}f_2(y)k(y-t_1)dy-\int_{t_2}^{b_2}f_2(y)k(y-a_1)dy\leq -\kappa$$ for some $\kappa>0$. The upper bound on $I_2$ becomes $(\varepsilon -F_1(t_1+\varepsilon)+F_1(t_1))(-\kappa+o(1))$, so $I_1+I_2$ becomes negative for small enough $\varepsilon$. This contradicts the minimality of $\langle \nu_1,\nu_2\rangle _{\delta }$. \textit{Case 2.} Assume $a_{1}=t_{1}$. If $a_{2} \neq t_{2}$, then in the same way as in Case 2. of the proof of $c_1 =c_2$ we show that $\langle s_{t_{2}+ \varepsilon}^{-}( \nu _{2})- \nu _{2}, \nu_{1} \rangle _{\delta }=I_2 <0$. If $a_{2} =t _{2}$ then $a_1=a_2$, so the function $P_{2}(x)$ is strictly decreasing on the interval $[a_{1}, a_1+\varepsilon)$ for some small $\varepsilon$. Arguing as in the Case 1. of the proof of $c_1=c_2$ we get $\langle s_{t_{1}+ \varepsilon}^{-}( \nu _{1})- \nu _{1}, \nu _{2} \rangle _{\delta } <0$. Again this contradicts the minimality of $\langle \nu _{1}, \nu _{2} \rangle _{\delta } $. We proved that $t_2=b_2.$     

In the same way we show that $t_1=b_1$. We proved that $\nu_i$ is the Lebesgue measure restricted to $[a_i,b_i]$. Since $c_1=c_2$ the intervals $[a_1,b_1],[a_2,b_2]$ are concentric, as desired.

\end{proof}
\begin{remark}The proof above uses only the fact that $\log|z|$ is strictly increasing with $|z|$, symmetric with respect to the $y$-axis and that the singularity at $|z|=0$ is integrable. 
\end{remark} 
\begin{lemma}\label{l-CollapsingPairsC}
Let $\nu_1,\nu_2\in \mathcal M^1(\C)$ be non-zero compactly supported measures and $x\in \C$ . Then $\langle c_{1,x}(\nu_1),c_{1,x}(\nu_2)\rangle\leq \langle \nu_1,\nu_2\rangle$ and equality holds if and only if there is an $y\in \C$ such that $\nu_1,\nu_2$ are the restrictions of the Lebesgue measure to balls centered in $y$.
\end{lemma}
\begin{proof}
\textbf{ Step 1.} We define collapsing along a line $\ell$ in $\C$. First let us assume that $\ell$ is the real line $\R\subset \C$. Let $\nu$ be a finite measure on $\C$ of bounded density $f\in L^1(\C)$. For $x\in \R$ let $F(x)=\int_{-\infty}^{+\infty} f(x+iy)dy.$ We define $h\in L^1(\C)$ as $$h(x+iy)=\begin{cases} 1 & \textrm{ if } |y|\leq F(x)/2\\
0 & \textrm{ otherwise. }\end{cases}.$$ 
We write $c_{\R}(\nu)$ for the measure $h(x+iy)dxdy.$ Let $\nu_1,\nu_2\in \mathcal M^1(\C)$, we argue that $\langle \nu_1,\nu_2\rangle\geq \langle c_{\R}(\nu_1),c_{\R}(\nu_2)\rangle$ with an equality if and only if there exists $t\in\R$ such that $\nu_1,\nu_2$ are translates of $c_{\R}(\nu_1),c_{\R}(\nu_2)$ by $it$.  Let $f_1,f_2$ be the densities of $\nu_1,\nu_2$ respectively. For $x\in \R$ define $\nu_j^x\in \mathcal M^1(\R)$ by $d\nu_j^x(y)=f_j(x+iy)dy.$ We have 
\begin{align*}\langle \nu_1,\nu_2\rangle=&\int_\R\int_\R\int_\R\int_\R\log((x_1-x_2)^2+(y_1-y_2)^2)f_1(x_1+iy_1)f_2(x_2+iy_2)dx_1dy_1dx_2dy_2\\
=& 2 \int_\R\int_\R\langle \nu_1^{x_1},\nu_2^{x_2}\rangle_{|x_1-x_2|}dx_1dx_2\geq 2  \int_\R\int_\R\langle c_{1,0}(\nu_1^{x_1}),c_{1,0}(\nu_2^{x_2})\rangle_{|x_1-x_2|}dx_1dx_2\\
=& \int_\R\int_{-F(x_1)/2}^{F(x_1)/2}\int_\R\int_{-F(x_2)/2}^{F(x_2)/2}\log((x_1-x_2)^2+(y_1-y_2)^2)dx_1dy_1dx_2dy_2\\
=& \langle c_{\R}(\nu_1),c_{\R}(\nu_2)\rangle.
\end{align*} The inequality in the second line holds by Lemma \ref{l-CollapsingPairsR} with equality if and only if $\nu_1^{x_1},\nu_2^{x_2}$ are Lebesgue measures restricted to concentric intervals for almost every $x_1,x_2\in \R$. Call $t$ the common center of these intervals. If the equality holds, then $\nu_1,\nu_2$ are translates of $c_{\R}(\nu_1),c_{\R}(\nu_2)$ by $it$.

For $\ell\neq \R$ we choose any isometry $\iota$ of $\C$ such that $\iota(\ell)=\R$ and put $c_{\ell}(\nu)=\iota^{-1}(c_{\R}(\iota^*\nu)).$ Like before we have that $\langle \nu_1,\nu_2\rangle\geq \langle c_{\ell}(\nu_1),c_{\ell}(\nu_2)\rangle$ with an equality if and only if there exists $z\in \ell^{\bot}$ such that $\nu_1,\nu_2$ are translates of $c_{\R}(\nu_1),c_{\R}(\nu_2)$ by $z$. Equivalently we have $\langle \nu_1,\nu_2\rangle= \langle c_{\ell}(\nu_1),c_{\ell}(\nu_2)\rangle$ if and only if there exists a line $\ell'$ parallel to $\ell$ such that $\nu_j=c_{\ell'}(\nu_j)$ for $j=1,2$.

\textbf{ Step 2.} Let $m_i=\nu_i(\C)$ and let $B_1, B_2$ be closed balls of volumes $m_1,m_2$ respectively, centered at $0$. We show that for every $\nu_1,\nu_2\in \mathcal M^1(\C)$ compactly supported we have either $\langle \nu_1,\nu_2\rangle> \langle \Leb |_{ B_1}, \Leb |_{B_2} \rangle$ or $\nu_1,\nu_2$ are the Lebesgue measure restricted to concentric balls. 

Let $R>0$ be such that $\supp\, \nu_i\subseteq B_{\C}(0,R)$ for $i=1,2$. By Lemma \ref{l-WeakContinuity}, there exists a pair of measures $\nu_1',\nu_2'\in\mathcal M^1(\C)$ supported on $B_\C(0,R)$ with $\nu_1'(\C)=m_1,\nu_2'(\C)=m_2$ such that 
$$\langle\nu_1',\nu_2'\rangle =\min\left\{\langle\mu_1,\mu_2\rangle| \mu_1,\mu_2\in\mathcal M^1(B_\C(0,R)), \mu_1(\C)=m_1,\mu_2(\C)=m_2\right\}.$$
We either have $\langle \nu_1,\nu_2\rangle> \langle\nu_1',\nu_2'\rangle$ or we can assume that $\nu_i=\nu_i'$ for $i=1,2$.  Choose $z,w\in\C$ such that $\arg(z)-\arg(w)\not \in \pi\Q$. By Step $1$ and the choice of $\nu_1',\nu_2'$, we have $\langle c_{z\R}(\nu_1'),c_{z\R}(\nu_2')\rangle=\langle c_{w\R}(\nu_1'),c_{w\R}(\nu_2')\rangle=\langle \nu_1',\nu_2'\rangle.$ Hence, by Step 1 there exist lines $\ell_1,\ell_2$ parallel to $z\R,w\R$ respectively such that $\nu_i'=c_{\ell_j}(\nu_i')$ for $i=1,2$ and $j=1,2$. By translating $\nu_1',\nu_2'$ if necessary we may assume that $\ell_1=z\R, \ell_2=w\R$. Being collapsed implies that densities of $\nu_1,\nu_2$ are characteristic functions of measurable sets, so we have $\nu_i=\Leb|_{I_i}$ for some bounded measurable sets $I_i$. Let $s_i$ be the orthogonal reflection in $\ell_i$ for $i=1,2$. Since $\nu_1',\nu_2'$ are collapsed along $\ell_1,\ell_2$, they are invariant under the group $S$ of isometries generated by $s_1,s_2$. Since $\arg(z)-\arg(w)\not \in \pi\Q$, the group $S$ is dense in ${\rm O}(2)$ (the orthogonal group of $\C$ seen as $\R^2$). We deduce that $I_1,I_2$ must be (up to a measure $0$ set) closed balls $B_1,B_2$ respectively. 

\end{proof}
As an easy consequence of Lemma \ref{l-CollapsingPairsR} and Lemma \ref{l-CollapsingPairsC} we get
\begin{lemma}\label{l-CollapsingPairsV}Let $V=\R^{r_1}\times \C^{r_2}$. Let $\nu_1,\nu_2\in \mathcal M^1(V)$ be non-zero compactly supported measures and $v=(v_1,\ldots,v_d)\in V$. Let $i \in \lbrace 1, \ldots , d \rbrace$. Then $\langle c_{i,v_i}(\nu_1),c_{i,v_i}(\nu_2)\rangle\leq \langle \nu_1,\nu_2\rangle$ and equality holds if and only if there is an $w=(w _{1}, \ldots , w_{d})\in V$ such that $\nu_1=c_{i,w _{i} }(\nu_1)$ and $\nu_2=c_{i, w_{i}}(\nu_2)$.
\end{lemma}
\begin{proof}
Assume without loss of generality that $v=0$. We will first treat the case where $i$ corresponds to a real place. Write $V^i=\left\{v\in V| v_i=0\right\}$ and $e_i=(0,\ldots,0,1,0,\ldots, 0)\in V$ where the unique non-zero entry is placed on the $i$-th coordinate. Let $f_1,f_2$ be the densities of $\nu_1,\nu_2$.  For $x\in V^i$ and $j=1,2$ define the measure  $\nu_j^{x}$ on $\R$ as $d\nu_j^{x}(t)=f_j(x+te_i)dt.$ Note that for every $g\in L^1(V)$ we have $$\int_V g(v)d\nu_j(v)=\int_{V^i}\int_\R g(x+te_i)d\nu_j^x(t)dx.$$ By Lemma \ref{l-CollapsingPairsR}, we get
\begin{align*}
\langle \nu_1,\nu_2\rangle=&\int_{V^i}\int_{V^i}\langle \nu_1^{x},\nu_2^{y}\rangle dxdy\\
\geq& \int_{V^i}\int_{V^i}\langle c_{1,v_i}(\nu_1^{x}),c_{1,v_i}(\nu_2^{y})\rangle dxdy\\
=& \langle c_{i,v_i}(\nu_1),c_{i,v_i}(\nu_2)\rangle.
\end{align*} By Lemma \ref{l-CollapsingPairsR}, the equality holds if and only if there exists $w _{i} \in\R$ such that for all $x,y\in V^i$ the measures $\nu_1^x,\nu_2^y$ are the Lebesgue measure restricted to intervals centered in $w _{i} \in\R$. In that case we also have  $\nu_1=c_{i, w _{i} }(\nu_1)$ and $\nu_2=c_{i, w _{i} }(\nu_2).$ If $i$ corresponds to a complex case the proof is almost identical but we use Lemma \ref{l-CollapsingPairsC} in place of Lemma \ref{l-CollapsingPairsR}.
\end{proof}
\begin{lemma}\label{l-CollapsingAlternative}
Let $\nu\in \mathcal P^1(V)$ be a compactly supported measure and let $i\in\left\{1,\ldots,d\right\},v_i\in \R$ if $i\in\left\{1,\ldots,r_1\right\}$ or $v_i\in \C$ otherwise. Then either $I(c_{i,v_i}(\nu))<I(\nu)$ or $I(c_{i,v_i}(\nu))=I(\nu)$ and there exists $v_i'$ such that $\nu=c_{i,v_i'}(\nu).$
\end{lemma}
\begin{proof}
Use Lemma \ref{l-CollapsingPairsV} for $\nu_1=\nu_2=\nu$.
\end{proof}
As a consequence of Lemma \ref{l-CollapsingAlternative} we get:
\begin{corollary}\label{c-MinimalCollapsing}
Let $\nu\in \mathcal P^1(V)$ be a measure minimizing the energy $I(\nu)$ on $\mathcal P^1(V)$ among compactly supported measures. Then there exists $v=(v_1,\ldots,v_d)\in V$ such that $c_{i,v_i}(\nu)=\nu$ for every $i=1,\ldots,d.$
\end{corollary}
\section{Limit measures and energy}\label{s-LimitMeasures}
Let $(n_i)_{i=1}^\infty $ be an increasing sequence of natural numbers. Let $k$ be a number field and assume that $(\mathcal{S}_{n_i})_{i\in \N}$ is a sequence of $n_i$-optimal sets in $\O_k$. By Corollary \ref{c-CompactFrame}, there is a compact cylinder $\Omega$ and sequences $(t_{n_i})_{i\in\N}\subset V, (s_{n_i})_{i\in\N}\subset V$ such that $\|s_{n_i}\|=n_i|\Delta_k|^{1/2}$ and $s_{n_i}^{-1}(\mathcal S_{n_i}-t_{n_i})\subset \Omega.$ Define a sequence of measures
\begin{equation}\label{e-defLimit}
\mu_{n_i}:=\frac{1}{n_i}\sum_{x\in \mathcal{S}_{n_i}} \delta_{s_{n_i}^{-1}(x-t_{n_i})}.
\end{equation}
Since $\Omega$ is compact we can assume, passing to a subsequence if necessary, that $\mu_{n_i}$ converges weakly-* to a probability measure $\mu$. This observation uses crucially Corollary \ref{c-CompactFrame} and is the key step in the proof of Theorem \ref{mt-noptimal}. Such limit measures are the central object of study in this section. 
\begin{definition}A probability measure $\mu$ on $V$ is called a \textbf{limit measure} if it is a weak-* limit of a sequence of measures $\mu_{n_i}$ constructed as above.
\end{definition}
\subsection{Density of limit measures}\label{s-Density} Let $\nu$ be a probability measure on $V$, absolutely continuous with respect to the Lebesgue measure on $V$. The density of $\nu$ is the unique non-negative function $f\in L^1(V)$ such that $d\nu=f(t)dt$ where $dt$ is the Lebesgue measure. We say that $\nu$ is of density at most $D$ if $f(t)\leq D$ for Lebesgue-almost all $t\in V$.
\begin{lemma}\label{l-DensityLimitMeasure}Any limit measure $\mu$ on $V$ is of density at most $1$.
\end{lemma}
\begin{proof}
Let $(n_i)_{i\in \N}$ and let $(\mu_{n_i})_{i\in\N}$ be a sequence of measures defined as in (\ref{e-defLimit}) such that $\mu$ is the weak-* limit of $\mu_{n_i}$ as $i\to\infty$. By Lebesgue differentiation theorem it is enough to verify that $\mu(\mathcal C)\leq \Leb(\mathcal C)$ for every bounded cylinder $\mathcal C\subset V$. We have 
$$\mu_{n_i}(\mathcal C)=\frac{1}{n_i}|\mathcal{S}_{n_i}\cap (s_{n_i}\mathcal{C}+t_{n_i})|\leq \frac{1}{n_i}|\O_k\cap  (s_{n_i}\mathcal{C}+t_{n_i})|.$$
Put $\mathcal{C}_i=s_{n_i}\mathcal{C}+t_{n_i}$. Since $\|s_{n_i}\|=|\Delta_k|^{1/2}n_i$ the cylinder $\mathcal C_i$ has volume $|\Delta_k|^{1/2}n_i\Leb(\mathcal C).$ As $\O_k$ is a lattice of covolume $|\Delta_k|^{1/2}$ we get\footnote{This does not work for a general lattice $\Lambda\subset V$. However, we know that $\O_k$ is invariant under multiplication by $\O_k^\times$ so we can multiply $\mathcal C_i$ by an element of $\O_k^\times$ so that it becomes "thick" in every direction.} $|\O_k\cap \mathcal C_i|=|\Delta_k|^{-1/2}\Leb(C_i)+o(\Leb(C_i)).$ Hence
$$\mu(\mathcal{C})=\lim_{i\to\infty}\mu_{n_i}(\mathcal{C})\leq \lim_{i\to\infty}\frac{1}{n_i}\left(n_i\Leb(\mathcal C)+o(n_i \Leb(\mathcal C))\right)=\Leb(\mathcal C).$$
\end{proof}
\subsection{Energy of limit measures}\label{s-Energy}
 Let $\nu$ be a finite measure on $V$ and write $\Delta(V)=\left\{(v,v)| \, v\in V\right\}\subset V\times V$. We define the off-diagonal energy $I'(\nu)$ as
\begin{align*}
I'(\nu)=&\int_{V\times V\setminus \Delta(V)} \log\|x-y\|d\nu(x)d\nu(y)\\ 
\end{align*} provided that the integral converges. The integral converges for all compactly supported measures of bounded density. We shall compute the energy of the limit measures and show that they minimize the energy among all compactly supported probability measures of density at most one.
\begin{proposition}\label{p-LimitEnergy}
Let $k$ be a number field, let $V=k\otimes_\Q \R$ and suppose that $\mu$ is a limit measure on $V$.  Then
  $I(\mu)=-\frac{1}{2}\log|\Delta_k|-\frac{3}{2}-\gamma_k+\gamma_\Q$ where $\gamma_k,\gamma_\Q$ are Euler--Kronecker constants of $k,\Q$ respectively.

\end{proposition}
\begin{proof} Let us fix a sequence $(\mu_{n_i})_{i\in \N}$ of measure defined as in (\ref{e-defLimit}) such that $\mu$ is the weak-* limit of $\mu_{n_i}$ as $i\to\infty$. Observe that by the energy formula \cite[Corollary 5.2]{BFS2017} 
$$\sum_{x\neq y\in\mathcal S_{n_i}}\log\|x-y\|=n_i^2\log n_i+n_i^2(-\frac{3}{2}-\gamma_k+\gamma_{\Q})+o(n_i^2).$$
We have
\begin{align*}I'(\mu_{n_i})=&\frac{1}{n_i^2}\sum_{x\neq y\in \mathcal S_{n_i}}\log\|s_{n_i}^{-1}(x-y)\| \\ = & -\frac{n_i+1}{n_i}\log\|s_{n_i}\|+\frac{1}{n_i^2}(n_i^2\log n_i+n_i^2(-\frac{3}{2}-\gamma_k+\gamma_\Q)+o(n_i^2))\\=&-\frac{1}{2}\log|\Delta_k|-\frac{3}{2}-\gamma_k+\gamma_\Q+o(1).
\end{align*}
Our task is reduced to proving that $\lim_{i\to\infty}I'(\mu_{n_i})=I(\mu).$ This doesn't simply follow from the weak-* convergence because the logarithm is not continuous in $0$. We remedy that by approximating the logarithm by a well chosen family of continuous functions.

 Let $T>0$. For $x>0$ put $\log^T x:=\max\left\{-T,\log x\right\}$ and let $\log^T 0:=-T$. For any compactly supported probability measure $\nu$ on $V$ put:
\begin{equation}\label{e-TEneregy}
I_T(\nu)=\int_{V \times V}\log^T\|x-y\|d\nu(x)d\nu(y).
\end{equation} 
Note that we integrate over the diagonal as well. The function $\log^T$ is continuous so we get $\lim_{i\to\infty}I_T(\mu_{n_i})=I_T(\mu).$ On the other hand, by Lebesgue dominated convergence theorem we have $\lim_{T\to\infty}I_T(\mu)=I(\mu)$ so $I(\mu)=\lim_{T\to\infty}\lim_{i\to\infty}I_T(\mu_{n_i}).$ We estimate the difference $I_T(\mu_{n_i})-I'(\mu_{n_i}).$ 
\begin{align}\label{e-TDifference}
I'(\mu_{n_i})-I_T(\mu_{n_i})=& \frac{T(n_i+1)}{n_i^2}+\frac{1}{n_i^2}\sum_{\substack{x\neq y\in \mathcal S_{n_i}\\ \|s_n^{-1}(x-y)\|\leq e^{-T}}}\left( \log\|s_n^{-1}(x-y)\|+T\right)\\
 =& \frac{T(n_i+1)}{n_i^2}+\frac{1}{n_i^2}\sum_{\substack{x\neq y\in \mathcal S_{n_i}\\ \|x-y\|\leq \|s_{n_i}\|e^{-T}}}\left( \log\|x-y\|-\log\|s_{n_i}\|+T\right)
\end{align}
Hence 
\begin{equation}\label{eq:55}
- \frac{1}{n_i^2}\sum_{\substack{x\neq y\in \mathcal S_{n_i}\\ \|x-y\|\leq \|s_{n_i}\|e^{-T}}}(\log\|s_{n_i}\|-\log\|x-y\|-T)  \leq I'(\mu_{n_i})-I_T(\mu_{n_i})\leq \frac{T(n_i+1)}{n_i^2}. 
\end{equation}
We proceed to estimate the left hand side. Note that by Corollary \ref{c-CompactFrame} and our choice of $s_{n_i},t_{n_i}$ there is a compact cylinder $\Omega=B_{\R}(0,A)^{r_1}\times B_{\C}(0,A)^{r_2}$ such that $\mathcal S_{n_i}\subset s_{n_i}\Omega+t_{n_i}$. Let $\Omega'=B_{\R}(0,2A)^{r_1}\times B_{\C}(0,2A)^{r_2}.$ Then for every $x,y\in\mathcal S_{n_i}$ we have $x-y\in s_{n_i}\Omega'$. Hence
\begin{align}
\sum_{\substack{x\neq y\in \mathcal S_{n_i}\\ \|x-y\|\leq \|s_{n_i}\|e^{-T}}} (\log\|s_{n_i}\|-\log\|x-y\|-T)\leq \sum_{x\in S_{n_i}}\sum_{\substack{0\neq z\in s_{n_i}\Omega'\cap \O_k\\ 
\|z\|\leq\|s_{n_i}\|e^{-T}}}(\log\|s_{n_i}\|-\log\|z\|-T) \\
=(n_i+1) \sum_{\substack{0\neq z\in s_{n_i}\Omega'\cap \O_k\\ 
\|z\|\leq\|s_{n_i}\|e^{-T}}}(\log\|s_{n_i}\|-\log\|z\|-T) .\label{e-SumCylinder}
\end{align}
Let us fix a good fundamental domain $\mathcal F$ of $\O_k^\times$ acting on $V^\times$ (see Definition \ref{d-GoodFD}) and a basis $\xi_1,\ldots,\xi_{d-1}$ of a maximal torsion free subgroup of $\O_k^\times$ together with the associated norm $\|\cdot\|_\infty$ on $\O_k^\times$ (see above Lemma \ref{l-NormIneq}). We can write $s_{n_i}=v\lambda_0$ with $\lambda_0\in \O_k^\times, v\in \mathcal F$ and $\|v\|=\|s_{n_i}\|=n_i|\Delta_k|^{1/2}.$ Put $A_4:=2|\Delta_k|^{1/2N}C_0 A$. By Lemma \ref{l-cone}, we have 
$$\lambda_0^{-1}s_{n_i}\Omega'=v\Omega'\subseteq \Omega'':=B_{\R}(0,n_i^{1/N}A_4)^{r_1}\times B_{\C}(0,n_i^{1/N}A_4)^{r_2}.$$ By Lemma \ref{l-NormIneq} and Lemma \ref{l-cone}, for every $x\in \mathcal F$ and $\lambda\in\O_k^\times$ such that $x\lambda\in \Omega''$ we have $$\|\lambda\|_\infty \leq \alpha^{-1}(\log (n_i^{1/N}\|x\|^{-1/N}A_4 C_0))=:C_{22}(\log n_i-\log \|x\|)+ C_{23}.$$
We can estimate the sum in (\ref{e-SumCylinder}) by 
\begin{align*}
 &\sum_{\substack{0\neq z\in s_{n_i}\Omega'\cap \O_k\\ \|z\|\leq\|s_{n_i}\|e^{-T}}}(\log\|s_{n_i}\|-\log\|z\|-T) = \sum_{\substack{0\neq z\in \lambda_0^{-1}s_{n_i}\Omega'\cap O_k\\ \|z\|\leq\|s_{n_i}\|e^{-T}}}(\log\|s_{n_i}\|-\log\|z\|- T)\\
 \leq&  \sum_{\substack{x\in\mathcal F\cap\O_k\\ \|x\|\leq n_i|\Delta_k|^{1/2}e^{-T}}}(\log\|s_{n_i}\|-\log\|x\|-T) |\left\{\lambda\in\O_k^\times| \|\lambda\|_\infty\leq C_{22}(\log n_i-\log\|x\|)+C_{23}\right\}| .
\end{align*}
Choose $T$ sufficiently large so that we have $C_{22}(\log n_i-\log \|x\|)+C_{23}\leq 2C_{22}(\log n_i-\log \|x\|)$ for every $x$ satisfying $\|x\|\leq n_i|\Delta_k|^{1/2}e^{-T}$ and $T>\frac{1}{2}\log \Delta_k$. For such $T$, we can bound the last expression by 
\begin{align*}
 \leq & \sum_{\substack{x\in\mathcal F\cap \O_k\\ \|x\|\leq n_i|\Delta_k|^{1/2}e^{-T}}}(\log\|s_{n_i}\|-\log\|x\|-T)|\left\{\lambda\in\O_k^\times| \|\lambda\|_\infty\leq 2C_{22}(\log n_i-\log\|x\|)\right\}|\\
  \leq & C_{24} \sum_{\substack{x\in\mathcal F\cap \O_k\\ \|x\|\leq n_i|\Delta_k|^{1/2}e^{-T}}}(\log\|s_{n_i}\|-\log\|x\|-T) (\log n_i-\log\|x\|)^{d-1}\\
  = & C_{24}\sum_{\substack{x\in\mathcal F\cap\O_k\\ \|x\|\leq n_i|\Delta_k|^{1/2}e^{-T}}}(\log n_i-\log\|x\|+\frac{1}{2}\log|\Delta_k|-T)(\log n_i-\log\|x\|)^{d-1}.
\end{align*} Put $Y=n_i|\Delta_k|^{1/2}e^{-T}.$ For $n_i$ big enough we can estimate the last expression by
\begin{align*}
&C_{24}\sum_{\substack{x\in\mathcal F\cap \O_k \\ \|x\|\leq Y}}(\log Y-\log\|x\|)(\log Y-\log\|x\| + (T-\frac{1}{2}\log|\Delta_k|))^{d-1}\\
= &C_{24}\sum_{i=0}^{d-1}{ d-1\choose i}(T-\frac{1}{2}\log|\Delta_k|)^i\sum_{\substack{x\in\mathcal F\cap\O_k \\ \|x\|\leq Y}}(\log Y-\log\|x\|)^{d-i}\\
\leq & C_{25}T^{d-1} e^{-T}n_i.
\end{align*} For the last inequality we have used Lemma \ref{l-IdealsEst}. Note that the condition $Y\geq 1$ used in that lemma is satisfied for $n_i$ big enough. The constant $C_{25}$ depends only on $k$. We combine above inequalities with (\ref{eq:55}) to get 
\begin{align}
-\frac{n_i+1}{n_i}C_{25}T^{d-1}e^{-T}  \leq I'(\mu_{n_i})-I_T(\mu_{n_i})\leq \frac{T(n_i+1)}{n_i^2},\\
|I'(\mu_{n_i})-I_T(\mu_{n_i})|\leq \frac{T(n_i+1)}{n_i^2}+\frac{n_i+1}{n_i}C_{25}T^{d-1}e^{-T}.
\end{align}
It follows that for any $T$ sufficiently large we have $\limsup_{i\to \infty}|I'(\mu_{n_i})-I_T(\mu_{n_i})|\leq C_{25}T^{d-1}e^{-T}$. Consequently
$$I(\mu)=\lim_{T\to\infty}\lim_{i\to\infty}I_T(\mu_{n_i})=\lim_{i\to\infty}I'(\mu_{n_i})=-\frac{1}{2}\log|\Delta_k|-\frac{3}{2}-\gamma_k+\gamma_\Q.$$ The proposition is proved.
\end{proof}
\begin{lemma}\label{l-LowerEnergyBound}
For every compactly supported\footnote{ The assumption on the support makes the proof easier but the statement should remain valid without it.}  probability measure $\nu$ on $V$ with density at most $1$ we have $-\frac{1}{2}\log|\Delta_k|-\frac{3}{2}-\gamma_k+\gamma_\Q\leq I(\nu).$ 
\end{lemma}
\begin{proof}
 Let $\nu$ be a compactly supported probability measure on $V$ of density at most $1$. Lemma \ref{l-AccumMeasures} (in the appendix) affords a sequence $E_n$ of subsets of $\O_k$ such that $|E_n|=n+1$ and the measures  $$\nu_{E_n,n}:=\frac{1}{n}\sum_{x\in E_{n}}\delta_{n^{-1/N}|\Delta_k|^{-1/2N} x}$$ converge weakly-* to $\nu$. Put $\log^* t=\log t$ of $t>0$ and $\log^*0=0$. For every finite measure $\mu$ on $V$ we have $I'(\mu)=\int_V\int_V \log^*\|x-y\|d\mu(x)d\mu(y),$ provided that the integral converges. The function $(x,y)\mapsto \log^*\|x-y\|$ is lower semi-continuous on $V \times V$ so 
\begin{align}&\limsup_{n\to\infty}I'(\nu_{E_n,n})\leq I'(\nu)= I(\nu),\\
&\limsup_{n\to\infty}\frac{1}{n^2}\sum_{x\neq y\in E_n}(\log\|x-y\|-\log n-\frac{1}{2}\log|\Delta_k|)\leq I(\nu).
\end{align}
By \cite[Corollary 5.2]{BFS2017}, we get 
$$-\frac{1}{2}\log|\Delta_k|-\frac{3}{2}-\gamma_k+\gamma_\Q \leq I(\nu).$$
\end{proof} 
In conjunction with Proposition \ref{p-LimitEnergy} it follows that any limit measure realizes the minimal energy among all compactly supported probability measures of density at most $1$.
\subsection{Measures of minimal energy}\label{s-EnergyMinimizers}
This entire section is devoted to the proof of the following:
\begin{proposition}\label{p-EnergyMinimizers}
Let $\nu$ be a compactly supported probability measure on $V$ of density at most $1$ which is realizing the minimal energy among all such measures. Then, there exists an open set $U$ and $v\in V$ such that 
\begin{enumerate}
    \item $\nu={\rm Leb}|_U$,
    \item $(\p U-v)\cap V^\times $ is a codimension $1$ submanifold of $V^\times$ of class $C^1$,
    \item $\lambda (\overline U-v)\subset U-v$ for every $0\leq\lambda<1$.
\end{enumerate}
\end{proposition}
We remark that the vast majority of the proof is getting around the issue of the singularity of the logarithm at $0$. Had the energy been defined by a smooth kernel, the proof would be settled almost immediately be the analogue of Lemma \ref{lem-Upreimage} and the implicit function theorem. 
\begin{proof}
We start by setting up some notation for the proof. Let $j\neq m\in\{1,\ldots,d\}$ and $s\in k_{\nu_j},t\in k_{\nu_m}$. We write 
\begin{align*}
V^j:=&\{ (v_1,\ldots,v_d)\in V| v_j=0\},\\
V^{m,j}:=&\{(v_1,\ldots,v_d)\in V| v_m=0,v_j=0\},\\
v^j:=&(v_1,\ldots,v_{j-1},0,v_{j+1},\ldots,v_d)\in V^j,\\
v^{j}(s):=&(v_1,\ldots,v_{j-1},s,v_{j+1},\ldots,v_d),\\
v^{m,j}:=&(v_1,\ldots,v_{m-1},0,v_{m+1},\ldots,v_{j-1},0,v_{j+1},\ldots,v_d)\in V^{m,j},\\
v^{m,j}(s,t):=&(v_1,\ldots,v_{m-1},s,v_{m+1},\ldots,v_{j-1},t,v_{j+1},\ldots,v_d).\\
\end{align*}
We will write $dv^{j}, dv^{m,j}$ for the Lebesgue measures on $V^j,V^{m,j}$ respectively. Recall that for $\mu,\mu'\in \mathcal M^1(V)$ we defined the symmetric bilinear form $\langle \mu,\mu'\rangle:=\int_{V}\int_{V} \log \|x-y\|d\mu(x)d\mu'(y).$

One of the key tools used to prove Proposition \ref{p-EnergyMinimizers} is the collapsing procedure $c_{j,v_j}:\mathcal M^1(V)\to\mathcal M^1(V)$ (see Definition \ref{d-Collapsing}), introduced and studied in the Section \ref{s-Collapsing}. 
By Corollary \ref{c-MinimalCollapsing}, there exists $v=(v_1,\ldots,v_d)\in V$ such that $c_{j,v_j}(\nu)=\nu$ for every $j=1,\ldots,d.$ Translating the measure $\nu$ by $-v$ we will assume from now on that $c_{j,0}(\nu)=\nu$ for all $j=1,\ldots,d$.
Let $F_j,h_j$ be the functions defined as in Definition \ref{d-Collapsing}. Since our measure is already collapsed with respect to all coordinates, the function $h_1=h_2=\ldots=h_d$ is the density of $\nu$. Choose $U$ such that $\nu=\Leb|_U$. This condition determines $U$ up to a measure $0$ set. Then 
\begin{equation}\label{eqn-Fj} U=\{v^j(s)|v^j\in V^j, |s|< F_j(v^j)\},\end{equation} which holds modulo a measure $0$ set for $j=1,\ldots, d$.
At this point the functions $F_j$ are defined only modulo measure $0$ sets. For technical reasons we redefine the values at $0$ as 
\[F_j(0):=\esssup_{v\in V^j}F_j(v^j).\]
All these values are finite because we assumed that $\nu$ is compactly supported. Together with (\ref{eqn-Fj}), this implies that 
\begin{equation}\label{eqn-UCylinder} U\subseteq \prod_{j=1}^{r_1} B_{\mathbb R}(0,F_j(0))\times \prod_{j=r_1+1}^d B_{\mathbb C}(0,F_j(0))\end{equation}
modulo measure $0$ set. 
Let us also define 
\begin{equation}\label{eqn-UjCont} U^j:=U\cap V^j=\{v^j\in V^j| F_j(v^j)>0\}\subseteq \{v^j\in V^j|\, |v_m|\leq F_m(0) \text{ for }m\neq j\}.\end{equation}
The last inclusion holds modulo measure $0$ set. 
We will often integrate over these sets in the sequel, so it is important to keep in mind that they are bounded. 
Let $P_V,P_j, j=1,\ldots,d$ be the functions on $V$ defined by 
\begin{align}\label{eq-Pdefinitions} P_j(x)&:=\begin{cases}\int_V \log|x-y|_jd\nu(y) &\textrm{ for }j=1,\ldots, r_1,\\ \int_V \log|x-y|^2_j d\nu(y) &\textrm{ for }j=r_1+1,\ldots, d,\end{cases}\\ P_V(x)&:=\int_V \log\|x-y\|d\nu(y)=\sum_{
j=1}^d P_j(x_j).\end{align} Clearly $P_j(x)$ depends only on the $j$--th coordinate of $x$ so it makes sense to abuse the notation and write $P_j(x)=P_j(x_j)$. Using (\ref{eqn-Fj}) we can express $P_j$ in terms of $F_j$. Let $x=(x_1,\ldots,x_{r_1}, s_{r_1+1}e^{i\theta_{r_1+1}},\ldots, s_{d}e^{i\theta_{d}})=(x_1,\ldots ,x_d)\in V$. We have 
\begin{equation}\label{e-iPotential2}
P_j(x_j)=\begin{cases} \int_{V^j} \left(\int_{-F_j(v^j)}^{F_j(v^j)}\log|x_j-t_j|dt_j\right) dv^j & \textrm{ for } j=1,\ldots,r_1,\\
 \int_{V^j} \left(\int_{0}^{F_j(v^j)}\int_0^{2\pi}t_j\log|x_j-t_je^{i\tau _j}|^2d\tau _j dt_j\right) dv^j& \textrm{ for } j=r_1+1,\ldots,d.
\end{cases}
\end{equation}
We note that the functions $F_j|_{V^j}$ are compactly supported modulo measure $0$ set so the integration happens essentially over a compact set. The function $P_V$ is defined so that $$\langle \nu,\mu\rangle =\int_V P_V(x)d\mu(x)$$ for every $\mu\in\mathcal M^1(V)$.
The regularity properties of $P_j$ play a key role in the proof of the proposition. We will prove a sequence of lemmas that consecutively improve our grasp on the analytic properties of the functions $P_j$ and $F_j$. 
\begin{lemma}\label{l-DIdentities}
\begin{enumerate}
\item For every $x\in \R$, $T\geq0$ we have 
\begin{align*}
\int_{-T}^T \log|x-t|dt=&\begin{cases}
(x+T)\log (x+T)-(x-T)\log (x-T)-2T &\textrm{ if}\ \ x\geq T\\ 
-(x-T)\log(T-x)+(x+T)\log (x+T)-2T &\textrm{ if}\ \ -T \leq x \leq T\\
(x+T)\log (-x-T)+(T-x)\log (T-x)-2T &\textrm{ if} \ \ x\leq -T
\end{cases}\\
\frac{d}{dx}\left(\int_{-T}^T \log|x-t|dt\right)=& \log(|T+x|)-\log(|T-x|)\ \ \textrm{for}\ \ x\neq T,\ \ x\neq -T.
\end{align*}
\item Write $dxdy$ for the Lebesgue measure on $\C$ in coordinates $z=x+iy$. For every $se^{i\theta}\in \C, s>0, T\geq 0$ we have 
\begin{align*}\int_{B_\C(0,T)}\log|se^{i\theta}-z|^2dxdy=&\begin{cases} 2\pi T^2\log T-\pi T^2+\pi s^2& \textrm{ if } s\leq T\\
2\pi T^2\log s & {\rm otherwise,}\end{cases}\\
\frac{d}{ds}\left(\int_{B_\C(0,T)}\log|se^{i\theta}-z|^2dxdy \right)=&\begin{cases} 2\pi s & \textrm{ if } s\leq T\\
\frac{2\pi T^2}{s} & {\rm otherwise.}\end{cases} \end{align*}
\end{enumerate}
\end{lemma}
\begin{proof}
This a routine calculation.
\end{proof}

\begin{lemma}\label{lem-PjProperties} The functions $P_j$ satisfy the following properties:
\begin{enumerate}
    \item \label{lem-Pj1} Let $j=1,\ldots, d$. The function $P_j$ is continuous.
    \item \label{lem-Pj2} Let $j=1,\ldots, r_1$. The function $P_j$ is strictly increasing on $[0,+\infty)$ and strictly decreasing on $(-\infty,0]$. 
    \item \label{lem-Pj3} Let $j=r_1+1,\ldots,d$. The function $P_j(s e^{i\theta})$ is strictly increasing in $s\in[0,+\infty)$ and does not depend on $\theta$. Moreover,
    \begin{align*} 
    \frac{d}{ds_j}P_j\left(s_je^{i\theta _j}\right)=\int _{V^j, s_j\leq F_j(v^j)}2\pi s_jdv^j+\int _{V^j,s_j\geq F_j(v^j)}\frac{2\pi F_j(v^j)^2}{s_j}dv^j. 
    \end{align*}
    \item \label{lem-Pj4}Let $j=1,\ldots, r_1.$ The function $P_j$ is almost everywhere differentiable and the derivative is square integrable on bounded intervals. Moreover,
    \[\frac{d}{d x_j} P_j(x_j)=\int_{V^j}\left(\log|F_j(v^j)+x_j|-\log|F_j(v^j)-x_j|\right) dv^j,\] for all $x_j\in\mathbb R$ where the right hand side converges.
    \item \label{lem-Pj5} Let $j=r_1+1,\ldots,d.$ The function $P_j$ is $C^1$.
\end{enumerate}
\end{lemma}
\begin{proof} 

(\ref{lem-Pj1}) By Lemma \ref{l-DIdentities}, for every $v^j\in V^j$ the integrals 
\[ \int_{B_{\R}(0,F_j(v^j))}\log|x_j-t_j|dt_j \quad\text{ and }\quad \int_{0}^{F_j(v^j)}t_j\int_0^{2\pi}\log|x_j-t_je^{i\theta_j}|^2d\theta_j dt_j\]
are continuous functions of $x_j$. By the formula (\ref{e-iPotential2}) $P_j$ is continuous, being a convergent integral of a family of continuous functions.

(\ref{lem-Pj2}) By Lemma \ref{l-DIdentities} and formula (\ref{e-iPotential2}), the function $P_j(x)$ is a convergent integral of functions which are strictly increasing on $[0,+\infty)$ if $F_j(v^j)>0$ or $0$ if $F_j(v^j)=0.$ Since $F_j(v^j)$ is non-zero on a set of positive measure (e.g. (\ref{eqn-Fj})), we deduce that $P_j(x)$ is strictly increasing on $[0,+\infty)$. Analogous argument shows that it is strictly decreasing on $(-\infty,0]$.

(\ref{lem-Pj3}) This follows from the formula $(\ref{eqn-Fj})$ and Lemma \ref{l-DIdentities}.  

(\ref{lem-Pj4}) By (\ref{eqn-Fj}),
\[P_j(x_j)=\int_{V^j}\left(\int_{-F_j(v^j)}^{F_j(v^j)} \log| x_j - t|dt\right)dv^j.\]
By Lemma \ref{l-DIdentities}, $\frac{d}{dx_j}\left(\int _{-F_j(v^j)}^{F_j(v^j)}\log|x_j-t|dt\right)=\log(|F_j(v^j)+x_j|)-\log(|F_j(v^j)-x_j|)$ for $x_j\neq \pm F_j(v^j).$ This expression is non-negative (resp. non-positive) for all $x_j\geq 0$ (resp. $x_j \leq 0$), whenever it is defined. Therefore, for all $x_j\in\R$ 
\begin{equation}\label{e-iPotentialDerR}
\frac{d}{d x_j} P_j(x_j)=\int_{V^j}\left(\log|F_j(v^j)+x_j|-\log|F_j(v^j)-x_j|\right) dv^j,
\end{equation} whenever the integral is finite. We show that the right hand side is square integrable on bounded intervals. Without loss of generality we can restrict to symmetric intervals. For $[-a,a]\subset \mathbb R$ we have
\begin{align*}\left(\int_{-a}^a \left(\frac{d}{dx_j}P_j(x_j)\right)^2dx_j\right) ^{1/2}=&\left(\int_{-a}^{a}\left(\int_{V^j}\left(\log|F_j(v^j)+x_j|-\log|F_j(v^j)-x_j|\right) dv_j\right)^2dx_j\right)^{1/2}\\
\text{ and by Minkowski's inequality: }&\\
\leq &\int_{V^j}\left(\int_{-a}^{a}\left(\log|F_j(v^j)+x_j|-\log|F_j(v^j)-x_j|\right)^2 dx_j\right)^{1/2}dv^j\\
\leq &2\int_{U^j}\left(\int_{-a}^{a}\left(\log|F_j(v^j)+x_j|\right)^2 dx_j\right)^{1/2}dv^j.
\end{align*}
The last term is finite because $\int (\log |t|)^2dt=t(\log^2|t|-2\log|t|+2)+\textrm{const.}$ and the primitive is continuous in $t$.

(\ref{lem-Pj5}) Since $\frac{d}{d\theta _j}P_j(s_je^{i\theta_j})=0$, it is enough to show that $\frac{d}{ds_j}P_j(s_je^{i\theta _j})$ is continuous in $s_j$.  By Lemma \ref{lem-PjProperties} (\ref{lem-Pj3}), 
\begin{align*} 
\frac{d}{ds_j}P_j(s_je^{i\theta _j})=&\int _{V^j,F_j(v^j)\leq s_j}\frac{2\pi F_j(v^j)^2}{s_j}dv^j+\int _{V^j,F_j(v^j)\geq s_j}2\pi s_j dv^j \\
=& \int_{V^j,F_j(v^j)\leq s_j}2\pi s_j\left( \frac{F_j(v^j)}{s_j}\right)^2dv^j+\int_{V^j,F_j(v^j)\geq s_j}2\pi s_jdv^j\\
=&\int_{V^j}2\pi s_j \min \left\{1,\left(\frac{F_j(v^j)}{s_j}\right)^2\right\} dv^j.
\end{align*} 
The last term is finite and continuous in $s_j$. 
\end{proof} 

\begin{lemma}\label{lem-Upreimage}
There exists a unique $\alpha\in \mathbb R$ such that $U=P_V^{-1}((-\infty,\alpha))$ up to a measure $0$ set. After redefining $U=P_V^{-1}((-\infty, \alpha ))$ , we have $\partial U=P_V^{-1}(\{\alpha\})$.  
\end{lemma} 
\begin{proof} 
Let $t\in\mathbb R$ and choose $j=1,\ldots,d.$ If $j$ corresponds to a real coordinate, we have 
\begin{align*}\Leb(P_V^{-1}(\{t\}))=&\int_{V^j} \Leb\{v_j\in \mathbb R| P_V(v^j(v_j))=t\}dv^j\\=& \int_{V^j} \Leb\left\{v_j\in \mathbb R| P_j(v_j)=t-\sum_{m\neq j}P_m(v_m)\right\}dv^j.\end{align*}
By Lemma \ref{lem-PjProperties} (\ref{lem-Pj2}), the set $\left\{v_j\in \mathbb R| P_j(v_j)=t-\sum_{m\neq j}P_m(v_m)\right\}$ has at most two points, hence the integral is zero. The case when $j$ corresponds to a complex coordinate is handled similarly. We conclude that 
\[\Leb(P_V^{-1}(\{t\}))=0 \text{ for every }t\in\mathbb R.\]
The function $P_V$ is continuous, so $t\mapsto \Leb(P_V^{-1}((-\infty,t)))$ is strictly increasing on $[\inf_{v\in V} P_V(v), +\infty)$. By the above identity, it is also continuous in $t$. It follows that there exists a unique $\alpha\in [\inf_{v\in V} P_V(v), +\infty)$, such that $\Leb(P_V^{-1}((-\infty,\alpha)))=1.$ Let $U':=P_V^{-1}((-\infty,\alpha))$. We claim that $U=U'$, modulo a measure zero set. Note that the function $P_V$ is proper, so $U'$ is a bounded set.

Let $\nu'={\rm Leb}|_{U'}\in\mathcal P^1(V)$. The set $U'$ is chosen to be concentrated precisely where $P_V$ is the smallest. Therefore, $\langle \nu,\nu'\rangle\leq \langle\nu,\nu\rangle$ with equality if and only if $\nu=\Leb|_{U'}$. 
Let $\varepsilon\geq 0$ be small and put $\nu_\varepsilon= (1-\varepsilon)\nu+\varepsilon \nu'$. This measure is compactly supported and is in $\mathcal P^1(V)$. Since, $\nu$ is supposed to be energy minimizing among compactly supported probability measures on $V$ of density at most $1$, we have:
$$ I(\nu_\varepsilon)=(1-\varepsilon)^2 I(\nu)+2\varepsilon(1-\varepsilon)\langle \nu,\nu'\rangle +\varepsilon^2 I(\nu')\geq  I(\nu).$$ We deduce that 
$$\left.\frac{d}{d\varepsilon}I(\nu_\varepsilon)\right|_{\varepsilon=0}=2(\langle \nu,\nu'\rangle- \langle \nu,\nu\rangle)\geq 0.$$ We already know that $\langle \nu,\nu'\rangle-\langle\nu,\nu\rangle\leq 0$ so $\langle \nu,\nu'\rangle= \langle\nu,\nu\rangle$. It follows that  $\nu=\Leb|_{U'}$, which is equivalent to the equality $U=U'$ modulo measure $0$ set. 

Put $U=P_V^{-1}((-\infty,\alpha ))$. We want to show that $\partial U=P_V^{-1}(\{\alpha \})$. By Lemma \ref{lem-PjProperties} (\ref{lem-Pj1}), $P_V$ is continuous and $\partial U\subseteq P_V^{-1}(\{\alpha \})$. To prove the reverse inclusion, take $v=(v_1,\ldots , v_n) \in P_V^{-1}(\{\alpha \})$. Since $0\in U$, $v\neq 0$. Let $l$ be such that $v_l\neq 0$. Take an open neighborhood $W$ of $V$. There exists $\varepsilon >0$ such that $v'=(v_1, \ldots,v_{l-1},(1-\varepsilon )v_l,v_{l+1}, \ldots ,v_d)\in W$. By Lemma \ref{lem-PjProperties} (\ref{lem-Pj2}), (\ref{lem-Pj3}), $P_V(v')<P_V(v)=\alpha$. We showed that $W\cap U$ is non-empty for any open neighborhood $W$ of $v$. It follows that $v\in \partial U$.


\end{proof} 
The functions $F_j$ and the set $U$ were determined only up to a measure $0$ set. Thanks to the above lemma we can \textbf{redefine them exactly}. We put $U:=P^{-1}_V((-\infty,\alpha)),$ $U^j:=U\cap V^j$ and use the identity (\ref{eqn-Fj}) to define $F_j(v_j)$, for every $v^j\in V^j$. The new definition of $F_j$ agrees with the old one modulo measure $0$ sets, but has the advantage of being defined everywhere.

\begin{lemma}\label{lem-FjProperties}Let $P_j^{-1}$ be the inverse of $P_j$ on $[0,+\infty)$ (both in the real and the complex cases). The functions $F_j$ satisfy the following properties:
\begin{enumerate}
    \item\label{lem-Fj1} For every $v^j\in U^j$ $$F_j(v^j)=P_j^{-1}\left(\alpha-P_V(v^j)+P_j(0)\right).$$
    \item\label{lem-Fj2} For $v^j\in U^j$ $$\frac{d}{dv_m}F_j(v^j)=-\frac{d}{dv_m}P_m(v_m) \left(\left. \frac{d}{dt}P_j(t)\right|_{t=F_j(v^j)}\right)^{-1},$$ for $m\neq j$ and when the right hand side is defined.
    \item\label{lem-Fj2.5}We have $F_j(v^j)\leq F_j(0)$ for every $v^j \in V^j$.
     \item\label{lem-Fj3} The function $F_j$ is continuous on $V^j$. 
\end{enumerate}
\end{lemma}
\begin{proof}
(\ref{lem-Fj1}) Recall that $U^j=\{v^j\in V^j| F_j(v^j)>0\}$. By (\ref{eqn-Fj}) and Lemma \ref{lem-PjProperties} (\ref{lem-Pj1}), $P_V(v^j(F_j(v^j)))=\alpha.$ Therefore, 
\begin{align*}P_V(v^j(F_j(v^j)))=P_j(F_j(v^j))+\sum_{m\neq j} P_m(v_m)=\alpha,\\
P_j(F_j(v^j))=\alpha-\sum_{m\neq j} P_m(v_m)=\alpha-P_V(v^j)+P_j(0).\end{align*}
The identity follows.

(\ref{lem-Fj2}) We differentiate both sides of point (\ref{lem-Fj1}). 

(\ref{lem-Fj2.5}) If $v^j\not\in U^j$, then $F_j(v^j)=0$. Let $v^j_0\in U^j, v^j_0\neq 0$. Like in the first point, we have $P_V(v^j_0(F_j(v^j_0)))=\alpha$. Let $s:=F_j(v^j_0)$. By Lemma \ref{lem-PjProperties} (\ref{lem-Pj2}), (\ref{lem-Pj3}), the value of $P_V(v^j(s))$ decreases as the coordinates of $v^j$ approach $0$. In particular, $P_V(0^j(s))<\alpha$ so $0^j(s)\in U$. By (\ref{eqn-Fj}), $F_j(0)=F_j(0^j(0))> s.$

(\ref{lem-Fj3}) Since $P_j$ is strictly increasing and continuous on $[0,+\infty)$ (Lemma \ref{lem-PjProperties}), the inverse is continuous and strictly increasing as well. On $U^j$, the function $F_j$ is continuous as a composition of continuous functions and it is zero on $V^j\setminus U^j$. To prove continuity it is enough to show that for every sequence $(v^j_n)\subseteq U^j$ such that $\lim_{n\to\infty} v^j_n=w^j\in V^j\setminus U^j$ we have $\lim_{n\to\infty} F_j(v^j_n)=0$. For the sake of contradiction, suppose this is not true. Passing to a sub-sequence if necessary we can assume that $\lim_{n\to\infty} F_j(v^j_n)=\ell>0$ for some $\ell\in (0,F_j(0)]$. We have $P_V(v^j_n(F_j(v^j_n)))=\alpha$. By continuity, $P_V(w^j(\ell))=\alpha.$ On the other hand $w^j=w^j(0)\not\in U$, so $P_V(w^j(0))\geq \alpha$. By Lemma \ref{lem-PjProperties} (\ref{lem-Pj2}), (\ref{lem-Pj3}), we conclude that $P_V(w^j(\ell))>\alpha$. This is a contradiction.  
\end{proof}

\begin{lemma}\label{lem-dPjLower} Let $\varepsilon>0$. There exists a positive constant $A_\varepsilon>0$ with the following properties
\begin{enumerate}
    \item\label{lem-dPL1} Let $j=1,\ldots, r_1$. For almost every $x_j\in \R$ such that $\varepsilon<|x_j|\leq F_j(0)$, we have $$\left| \frac{d}{dx_j}P_j(x_j)\right|\geq A_\varepsilon.$$
    \item\label{lem-dPL2} Let $j=r_1+1,\ldots, d$ and write $x_j=s_je^{i\theta_j}, s_j\in[0,+\infty), \theta _j\in [0,2\pi).$ For almost every $x_j\in\C$ such that $\varepsilon <|x_j|\leq F_j(0),$ we have 
    $$\left| \frac{d}{ds_j}P_j(x_j)\right|\geq A_\varepsilon.$$
\end{enumerate}
\end{lemma}
\begin{proof}
(\ref{lem-dPL1}) Let $\varepsilon>0$ and let $\varepsilon<|x_j|\leq F_j(0)$. Assume  that $x_j>\varepsilon$ (i.e. $x_j$ is positive). We can always restrict to a smaller $\varepsilon$, so for convenience we also assume that $1-2\varepsilon \Leb (U^j)>0$. By Lemma \ref{lem-PjProperties}, 
\begin{align*}
\frac{d}{dx_j}P_j(x_j)=&\int_{U^j}(\log |F_j(v^j)+x_j|-\log |F_j(v^j)-x_j|)dv^j\\
\geq & \int_{U^j}\frac{|F_j(v^j)+x_j|-|F_j(v^j)-x_j|}{|F_j(v^j)+x_j|}dv^j\\
=& 2\int_{U^j}\frac{\min\left\{F_j(v^j),x_j\right\}}{|F_j(v^j)+x_j|}dv^j\geq \frac{1}{F_j(0)}\int_{U^j}\min\left\{F_j(v^j),x_j\right\}dv^j\\
\geq&  \frac{1}{F_j(0)}\int_{U^j}\min\left\{F_j(v^j),\varepsilon\right\}dv^j,
\end{align*}provided that the integrals converge (this happens for almost every $x_j$). To estimate the last expression we go back to the definition of $F_j$ (see Definition \ref{d-Collapsing}). It implies that \begin{equation}\label{e-Step4P46}\int_{U^j}2 F_j(v^j)dv^j=\nu(V)=1.\end{equation} Let $E_j:=\left\{v^j\in U^j|\, F_j(v^j)\geq \varepsilon \right\}$. For every $v^j\in U^j$ we have $F_j(v^j)\leq F_j(0)$ so (\ref{e-Step4P46}) yields $F_j(0)\Leb(E_j)+\varepsilon(\Leb(U^j)-\Leb(E_j))\geq \frac{1}{2}.$ Therefore, $\Leb(E_j)\geq \frac{1-2\varepsilon \Leb(U^j)}{2(F_j(0)-\varepsilon)}.$ We get 
\begin{align*}
 \frac{d}{dx_j} P_j(x_j)\geq \frac{1}{F_j(0)}\int_{U^j}\min\left\{F_j(v^j),\varepsilon\right\}dv^j\geq& \frac{\varepsilon}{F_j(0)}\Leb(E_j)\geq\frac{\varepsilon(1-2\varepsilon \Leb(U^j))}{2F_j(0)(F_j(0)-\varepsilon)}=:A_{\varepsilon ,j}>0.
\end{align*} 
The final lower bound is positive and depends only on $\varepsilon, F_j(0)$ and $\Leb(U^j)$. The function $P_j$ is even so for $x_j<0$ we have $\frac{d}{dx_j}P_j(x_j)=-\frac{d}{d(-x_j)}P_j(-x_j)\leq -A_{\varepsilon ,j} $. 

(\ref{lem-dPL2}) Let $\varepsilon >0$ and let $\varepsilon <s_j \leq F_j(0)$. Like in the real case we can restrict to a smaller $\varepsilon$ if necessary so for technical reasons assume $1-\pi \varepsilon ^2 \Leb (U^j)>0$. By Lemma \ref{lem-PjProperties},
\begin{align*} 
\frac{d}{ds_j}P_j(x_j)=&\int_{V^j, F_j(v^j)\leq s_j}\frac{2\pi F_j(v^j)^2}{s_j}dv^j+\int _{V^j, F_j(v^j)\geq s_j}2\pi s_j dv^j\\
=&\int _{U^j,F_j(v^j)\leq s_j}\frac{2\pi F_j(v^j)^2}{s_j}dv^j+\int _{U^j, F_j(v^j)\geq s_j}2\pi s_j dv^j\\ \geq & 2\pi \int_{U^j} \min \left\{ \frac{F_j(v^j)^2}{s_j},s_j\right\}dv^j.
\end{align*}
To estimate the last expression, analogously as before, consider $E_j:=\{ v^j \in U^j | \ \ F_j(v^j) \geq \varepsilon \}$. By Definition \ref{d-Collapsing}, 
\begin{equation*} 
\int _{U^j}\pi F_j(v^j)^2dv^j=\nu (V)=1.
\end{equation*} 
Therefore $1\leq \pi\varepsilon ^2 (\Leb(U^j)- \Leb(E_j))+\pi F_j(0)^2\Leb (E_j)$ and $\Leb(E_j)\geq \frac{1-\pi \varepsilon ^2 \Leb (U ^j)}{\pi F_j(0)^2-\pi\varepsilon ^2}$. We get 
\begin{align*} 
\frac{d}{ds_j}P_j(x_j)\geq & 2\pi \int _{E_j}\min \left\{ \frac{F_j(v^j)^2}{F_j(0)}, \varepsilon \right\} dv^j\geq2\pi \frac{\varepsilon^2}{F_j(0)} \Leb (E_j)\\ \geq & 2\pi \frac{\varepsilon ^2 (1-\pi \varepsilon ^2 \Leb (U^j))}{F_j(0)(\pi F_j(0)^2-\pi \varepsilon ^2)}=:A_{ \varepsilon ,j}>0.   
\end{align*} 
Take $A_{\varepsilon }= \min \{A_{ \varepsilon ,j} | \ \ j=1, \ldots , d\}$. 
\end{proof}

\begin{lemma}\label{lem-Substitution}
Let $m,j\in \{1,\ldots, d\}$ and $m\neq j$. Let $s\geq 0$. We have
\[F_m(v^{m,j}(0,F_j(v^{m,j}(s,0))))=\min\{s, F_m(v^{m,j}(0,0))\}.\]
\end{lemma}

\begin{proof} 
The proof starts with a simple claim. Let $\alpha$ be the constant from Lemma \ref{lem-Upreimage}.


\textbf{ Claim 1.} For any $x\geq 0$ such that $F_j(v^{m,j}(x,0))>0$ we have $P_V(v^{m,j}(x,F_j(v^{m,j}(x,0))))=\alpha.$
Similarly, for any $y\geq 0$ such that $F_m(v^{m,j}(0,y))>0$ we have $P_V(v^{m,j}(F_m(v^{m,j}(0,y)),y))=\alpha.$ The identities are symmetric, so it will be enough to prove the first one. By (\ref{eqn-Fj}), $v^{m,j}(x,t)\in U$ for every $0\leq t< F_j(v^{m,j}(x,0))$. In particular, $P_V(v^{m,j}(x,t))<\alpha$. On the other hand $v^{m,j}(x,F_j(v^{m,j}(x,0)))\not\in U$, so 
$P_V(v^{m,j}(x,F_j(v^{m,j}(x,0))))\geq\alpha.$ The desired equality follows now from the continuity of $P_V$, established in Lemma \ref{lem-PjProperties} (\ref{lem-Pj1}).

We are ready to prove the lemma. Let $s\geq 0$. Consider two cases. 

\textbf{Case 1.} $F_j(v^{m,j}(s,0))=0$. By (\ref{eqn-Fj}), the point $v^{m,j}(s,0)$ is not in $U$, so again by (\ref{eqn-Fj}), $s\geq F_m(v^{m,j}(0,0))$. It follows that 
\[F_m(v^{m,j}(0,F_j(v^{m,j}(s,0))))=F_m(v^{m,j}(0,0))=\min\{s, F_m(v^{m,j}(0,0))\}.\]

\textbf{Case 2.} $F_j(v^{m,j}(s,0))>0$. By Claim 1, $P_V(v^{m,j}(s,F_j(v^{m,j}(s,0))))=\alpha.$ Put $t=F_j(v^{m,j}(s,0))$. If $F_m(v^{m,j}(0,t))=0$, then by (\ref{eqn-Fj}) $v^{m,j}(0,t)\notin U$ and $P_{V}(v^{m,j}(0,t))\geq \alpha $. By Lemma \ref{lem-PjProperties} (\ref{lem-Pj2}), $P_{V}(v^{m,j}(0,t)) <P_{V}(v^{m,j}(s,t))=\alpha $ if $s\neq 0$. Therefore, $s=0$ and the lemma holds. If $F_m(v^{m,j}(0,t))>0$, then by Claim 1, $P_V(v^{m,j}(F_m(v^{m,j}(0,t)),t))=\alpha =P_V(v^{m,j}(s,t))$. By Lemma \ref{lem-PjProperties} (\ref{lem-Pj2}), $s=F_m(v^{m,j}(0,t))$. 
It remains to justify why $F_m(v^{m,j}(0,t))\leq F_m(v^{m,j}(0,0))$. This follows immediately from Lemma \ref{lem-FjProperties} (\ref{lem-Fj1}) and Lemma \ref{lem-PjProperties} (\ref{lem-Pj2}),(\ref{lem-Pj3}). 

\end{proof} 

\begin{lemma}\label{lem-PC1}
Let $j=1,\ldots, r_1$. The function $P_j$ is of class $C^1$ on the interval $(-F_j(0),F_j(0)).$
\end{lemma}
\begin{proof}
Let $\varepsilon>0$. By Lemma \ref{lem-FjProperties} (\ref{lem-Fj3}), the functions $F_j$ are continuous so there exists $\delta>0$ be such that the set $D_0:=\prod_{j=1}^{r_1}[-\delta,\delta]\times \prod_{j=r_1+1}^d B_{\mathbb C}(0,\delta)$ satisfies $F_j(v^j)> F_j(0)-\varepsilon/2 $ for every $v^j\in D_0$. 
Choose a cover of $U^j=\bigcup_{l=0}^L D_l$ by non-empty sets of the form 
\[ D_l=\left[\prod_{m=1}^{j-1}[a_{m,l},b_{m,l}] \times\{0\}\times\prod _{m=j+1}^{r_1}[a_{m,l},b_{m,l}]\times \prod_{m=r_1+1}^d \{z\in\C|\quad a_{m,l}\leq|z|\leq b_{m,l}\}\right]\cap U^j \]
where $l\geq 1$, $a_{m,l}<b_{m,l}$ are such that $D_p$ and $D_q$ have disjoint interiors for $p\neq q$. 

By Lemma \ref{lem-PjProperties} (\ref{lem-Pj4}),
\[\frac{d}{d x_j} P_j(x_j)=\int_{U^j}\left(\log|F_j(v^j)+x_j|-\log|F_j(v^j)-x_j|\right) dv^j,\] for all $x_j\in \R$ such that  the right hand side converges.
To shorten notation, let us write $H(x,t):=\log|x+t|-\log|x-t|$ for $x,t\in\C$. The function $H$ is well defined outside the set $\{(x,t)\in \C^2|\ \ x=t\ \ \textrm{or}\ \  x=-t\}$. 

Using the decomposition of $U$ into sets $D_l$ we get 
\[\frac{d}{d x_j} P_j(x_j)=\sum_{l=0}^L \int_{D_l} H(F_j(v^j),x_j)dv^j= \sum_{l=0}^L \mathcal{I}_l(x_j),\] where $\mathcal{I}_l(x_j):=\int_{D_l} H(F_j(v^j),x_j)dv^j$ provided that the integrals converges. To prove the lemma, we will show that each $\mathcal{I}_l$ converges and is a continuous function on $(-F_j(0)+\varepsilon,F_j(0)-\varepsilon).$

First, let us handle $\mathcal{I}_0(x_j)$. Let $v^j\in D_0$. Since $F_j(v^j)> F_j(0)-\varepsilon/2$, the function $H(F_j(v^j),x_j)$ is continuous in $x_j\in(-F_j(0)+\varepsilon,F_j(0)-\varepsilon).$ We conclude that $\mathcal{I}_0(x_j)$ is continuous on $(-F_j(0)+\varepsilon,F_j(0)-\varepsilon)$, being a convergent integral of continuous functions.

Fix $l\in\{1,\ldots,L\}$. The interior of $D_l$ is disjoint with the interior of $D_0$, so one of the following occurs: 
\begin{itemize}
    \item there exists an $m\in\{1,\ldots,r_1\}$ different than $j$, such that $[a_{m,l},b_{m,l}]\cap (-\delta,\delta)=\emptyset$. This means that $m\in\{1,\ldots, r_1\}$ and $a_{m,l}\geq\delta$ or $b_{m,l}\leq -\delta$ The argument in each case is completely symmetric, so let us assume that $a_{m,l}\geq\delta$.
    \item there exists an $m\in\{r_1+1,\ldots,d\}$ such that $a_{m,l}\geq\delta.$
\end{itemize}
In this part of the proof $l$ is fixed so we drop it from the indices to lighten the notation.
Denote by $D^{m,j}$ the image of $D$ under the projection of $V$ onto $V^{m,j}$. By the definition of $F_j$'s,
\begin{align*}\mathcal{I}(x_j)=&\int_{D^{m,j}}\left(\int_{a_m}^{\min\{b_m,F_m(v^{m,j})\}}H(F_j(v^{m,j}(s,0)),x_j)ds\right)dv^{m,j}  \quad \textrm{if}\ \ j=1,\ldots,r_1,\\
\mathcal{I}(x_j)=&\int_{D^{m,j}}\left(\int_{a_m}^{\min\{b_m,F_m(v^{m,j})\}}\int_0^{2\pi}s H(F_j(v^{m,j}(se^{i\theta},0)),x_j)d\theta ds\right)dv^{m,j} \ \ \textrm{if}\ \  j=r_1+1,\ldots,d.
\end{align*}
We now consider two cases depending on whether $m$ corresponds to a real or to a complex coordinate. 

\textbf{ Case }$m\in\{1,\ldots,r_1\}$. 
Consider the substitution $$s=F_m(v^{m,j}(0,w)), \quad \textrm{where}\ \ w\in [F_j(v^{m,j}(b_m,0)),F_j(v^{m,j}(a_m,0))].$$ By Lemma \ref{lem-FjProperties} (\ref{lem-Fj1}),(\ref{lem-Fj2}),(\ref{lem-Fj3}), Lemma  \ref{lem-PjProperties} (\ref{lem-Pj2}),(\ref{lem-Pj4}) and Lemma \ref{lem-dPjLower} the function $F_m(v^{m,j}(0,w))$ is continuous, almost everywhere differentiable and non-increasing in $w\in [0,\infty)$. By Lemma \ref{lem-Substitution}, Lemma \ref{lem-FjProperties}(\ref{lem-Fj2}), Lemma \ref{lem-PjProperties} (\ref{lem-Pj4}) and integrating by substitution we obtain 
\begin{align*}\mathcal{I}(x_j)=&\int_{D^{m,j}}\left(\int_{F_j(v^{m,j}(b_m,0))}^{F_j(v^{m,j}(a_m,0))}H(w,x_j) \left| \frac{d}{dw} F_m(v^{m,j}(0,w))\right|dw\right)dv^{m,j}\\
=&\int_{D^{m,j}}\left(\int_{F_j(v^{m,j}(b_m,0))}^{F_j(v^{m,j}(a_m,0))}H(w,x_j) \left| \frac{d}{dw}P_j(w)\right| \left| \frac{d}{dt} P_m(t)|_{t=F_m(v^{m,j}(0,w))}\right|^{-1}dw\right)dv^{m,j}.
\end{align*}
First we establish the convergence. By Lemma \ref{lem-dPjLower}, Lemma \ref{lem-FjProperties} (\ref{lem-Fj2.5}), Lemma \ref{lem-PjProperties}(\ref{lem-Pj4}) and the Cauchy-Schwartz inequality,
\begin{align*}
&\int_{D^{m,j}}\left(\int_{F_j(v^{m,j}(b_m,0))}^{F_j(v^{m,j}(a_m,0))}|H(w,x_j)| \left| \frac{d}{dw}P_j(w)\right| \left| \frac{d}{dt} P_m(t)|_{t=F_m(v^{m,j}(0,w))}\right|^{-1}dw\right)dv^{m,j}\\
\leq& A_{\delta /2}^{-1}\int_{D^{m,j}}\left(\int_{F_j(v^{m,j}(b_m,0))}^{F_j(v^{m,j}(a_m,0))}|H(w,x_j)| \left| \frac{d}{dw}P_j(w)\right|dw\right)dv^{m,j}\\
\leq& A_{\delta /2}^{-1}\int_{D^{m,j}}\left(\int_{F_j(v^{m,j}(b_m,0))}^{F_j(v^{m,j}(a_m,0))}H(w,x_j)^2dw\right)^{1/2}\left(\int_{F_j(v^{m,j}(b_m,0))}^{F_j(v^{m,j}(a_m,0))}\left| \frac{d}{dw}P_j(w)\right|^2dw\right)^{1/2}dv^{m,j}\\
\leq& A_{\delta /2}^{-1}\Leb(D^{m,j})\left(\int_0^{F_j(0)}H(w,x_j)^2dw\right)^{1/2}\left(\int_0^{F_j(0)}\left| \frac{d}{dw}P_j(w)\right|^2dw\right)^{1/2}<\infty.\\
\end{align*} To prove continuity, let $x_j'\in (-F_j(0)+\varepsilon,F_j(0)-\varepsilon)$ be close to $x_j$. Similarly
\begin{align*}|\mathcal{I}(x_j)-\mathcal{I}(x_j')|\leq& A_{\delta /2}^{-1}\int _{D^{m,j}}\left(\int_{F_j(v^{m,j}(b_m,0))}^{F_j(v^{m,j}(a_m,0))}|H(w,x_j)-H(w,x_j')| \left| \frac{d}{dw}P_j(w)\right|dw\right)dv^{m,j}\\
\leq A_{\delta /2}^{-1}\Leb&(D^{m,j})\left(\int_0^{F_j(0)}(H(w,x_j)-H(w,x_j'))^2dw\right)^{1/2}\left(\int_0^{F_j(0)}\left| \frac{d}{dw}P_j(w)\right|^2dw\right)^{1/2}.\\
\end{align*}
The continuity follows because $\lim_{x_j'\to x_j}\int_0^{F_j(0)}(H(w,x_j)-H(w,x_j'))^2dw=0.$

\textbf{ Case} $m\in\{r_1+1,\ldots, d\}$.
We consider the same substitution $$s=F_m(v^{m,j}(0,w)),\quad w\in [F_j(v^{m,j}(b_m,0)),F_j(v^{m,j}(a_m,0))],$$ with the caveat that we use it to parametrize only the radial component of the complex coordinate. Note that the function $F_m$ does not depend on the angles of complex coordinates. By Lemma \ref{lem-Substitution}, Lemma \ref{lem-FjProperties}(\ref{lem-Fj2}) and integrating by substitution we obtain 
\begin{align*}
\mathcal{I}(x_j)=&2\pi \int_{D^{m,j}}\left(\int_{F_j(v^{m,j}(b_m,0))}^{F_j(v^{m,j}(a_m,0))}F_m(v^{m,j}(0,w))H(w,x_j) \left| \frac{d F_m(v^{m,j}(0,w))}{dw}\right|dw\right)dv^{m,j}.\end{align*}
By Lemma \ref{lem-FjProperties} (\ref{lem-Fj2.5}),
\begin{align*}2\pi \int _{D^{m,j}}\left( \int _{F_{j}(v^{m,j}(b_m,0))}^{F_j(v^{m,j}(a_m,0))} \left| F_m(v^{m,j}(0,w))H(w,x_j)\left|\frac{dF_m(v^{m,j}(0,w))}{dw}\right|\right|dw\right)dv^{m,j}\\ \leq 2\pi F_m(0)\int_{D^{m,j}}\left(\int_{F_j(v^{m,j}(b_m,0))}^{F_j(v^{m,j}(a_m,0))}|H(w,x_j)| \left| \frac{d F_m(v^{m,j}(0,w))}{dw}\right|dw\right)dv^{m,j}.\\
\end{align*}
As in the real case the last term is finite. Therefore $\mathcal{I}(x_j)$ is convergent. 
\\Similarly for $x_j,x_j'\in (-F_j(0)+\varepsilon,F_j(0)-\varepsilon)$ we have 
\begin{align*}
    |\mathcal{I}(x_j)-\mathcal{I}(x_j')|\leq 2\pi F_m(0)\int_{D^{m,j}}\left(\int_{F_j(v^{m,j}(b_m,0))}^{F_j(v^{m,j}(a_m,0))}|H(w,x_j)-H(w,x_j')| \left| \frac{d F_m(v^{m,j}(0,w))}{dw}\right|dw\right)dv^{m,j}.\\
\end{align*}
Again as in the real case $\lim_{x_j'\to x_j} |\mathcal{I}(x_j)-\mathcal{I}(x_j')|=0$ and $\mathcal{I}(x_j)$ is continuous. 

We proved that the functions $\mathcal{I}_0,\ldots, \mathcal{I}_L$ are well defined and continuous on the interval $(-F_j(0)+\varepsilon,F_j(0)-\varepsilon)$ so the same must hold for $\frac{d}{d x_j} P_j(x_j)$. We prove the lemma by letting $\varepsilon\to 0.$
\end{proof}

We are now ready to prove the proposition. Cases $V=\R$ and $V=\C$ are handled by Lemma \ref{l-CollapsingPairsR} and Lemma \ref{l-CollapsingPairsC} respectively. Assume now $d\geq 2$. Recall that we already reduced the proof to the case $v=0$. By Lemma \ref{lem-Upreimage}, the boundary $\partial U$ is precisely the set $P_V^{-1}(\{\alpha\})$. By Lemma \ref{lem-PC1} and Lemma \ref{lem-PjProperties} (\ref{lem-Pj5}), the functions $P_j$ are all of class $C^1$ on $B_{k_{\nu_j}}(0,F_j(0))$. It follows that $P_V$ is $C^1$ on $\mathcal{B}:=\prod_{l=1}^{r_1}B_{\R}(0,F_l(0))\times \prod_{l=r_1+1}^{d}B_{\C}(0,F_l(0))$. By Lemma \ref{lem-dPjLower} the derivative of $P_V$ is non-zero on $\mathcal{B}\cap V^\times$. Using the implicit function theorem \cite[Thm 9.28]{babyrudin} for $P_V-\alpha $ on the set $V^{\times}\cap \mathcal{B}$ we deduce that $P_V^{-1}(\{\alpha\})\cap V^\times \cap \mathcal{B}=P_V^{-1}(\{ \alpha \}) \cap V^{\times}$ is a $C^1$ sub-manifold. The assertion $\lambda \overline U\subset U$ for $0\leq \lambda<1$ follows immediately from Lemma \ref{lem-PjProperties} (\ref{lem-Pj2}),(\ref{lem-Pj3}) and  Lemma \ref{lem-Upreimage}.\end{proof}
\begin{remark}
The methods developed in Sections \ref{s-Collapsing} and \ref{s-EnergyMinimizers} are likely to work in greater generality. We expect that for any type of kernel $K\colon \mathbb R^n\times \mathbb R^n \to \mathbb R$ satisfying the following properties:
\begin{itemize}
    \item $K(x+w,y+w)=K(x,y)$ for $x,y,w\in \mathbb R^n$,
    \item $K(x,y)=K(y,x)$ for $x,y\in \mathbb R^n$,
    \item $K(0,(x_1,\ldots,x_n))$ is strictly increasing in each $x_i>0$ for $i=1,\ldots,n$,
    \item $K(x,y)$ is $C^1$ outside the subspaces $x_i=y_i$ for $i=1,\ldots,n$,
    \item $K$ has at worst logarithmic singularities on the subspaces $x_i=y_i$ for $i=1,\ldots,n$,
\end{itemize}
the compactly supported measures $\mu\in\mathcal M^1(\mathbb R^n)$ minimizing the energy $\int_{\mathbb R^n}\int_{\mathbb R^n}K(x,y)d\mu(x)d\mu(y)$ satisfy the conclusion of Proposition \ref{p-EnergyMinimizers}. 
\end{remark}
\section{Non-existence of $n$-optimal sets.}\label{s-NonExistence}
\subsection{Discrepancy and almost equidistribution}\label{s-Discrepancy}
Let $\mu$ be any limit measure on $V=k\otimes_\Q \R$. In this section we study the discrepancy of the sets $U$ such that $\mu=\Leb|_U$ which are provided by Proposition \ref{p-EnergyMinimizers}. The \textbf{lattice point discrepancy} is defined as follows (see \cite[Section 6.2]{HuxleyBook}).  
\begin{definition}\label{d-Discrepancy} Let $U$ be a bounded measurable subset of $V$. For $t\in V^\times, v\in V$ let $N_t(U,v):=|(tU+v)\cap \O_k |$ and define the discrepancy 
$$D_t(U,v):=N_t(U,v)-|\Delta_k|^{-1/2}\Leb(U)\|t\|$$ and the maximal discrepancy 
$$D_t(U):=\esssup_{v\in V} |D_t(U,v)|.$$ 
\end{definition}
We will use the following technical property of the maximal discrepancy. 
\begin{lemma}\label{l-DiscrepancyCont}
Let $U$ be a bounded measurable subset of $V$ of positive Lebesgue measure and such that $\p U$ has zero Lebesgue measure. Then either $D_t(U)<1$ for all $t\in V^\times$  or there exists $\delta>0$, a non-empty open subset $T\subseteq V^\times$ and a non-empty open subset $W$ of $V$ such that $|D_t(U,v)|> 1+\delta$ for all $t\in T, v\in W$. 
\end{lemma}
\begin{proof}
Let $E=\bigcup_{x\in \O_k}\left[\bigcup_{t\in V^\times}\left[(x-t \p U)\times \left\{t\right\}\right]\right]\subseteq V\times V^\times$ and for every $t\in V^\times$ let  $E_t:=\left\{v\in V| (v,t)\in E\right\}=\bigcup_{x\in \O_k}(x-t \p U)$. Because the set $U$ is bounded, the unions defining $E$ and $E_t$ are locally finite. We deduce that $E$ and $E_t$ are closed and $E_t$ has measure $0$ for every $t\in V^\times$. The function $(v,t)\mapsto N_t(U,v)$ is locally constant on $(V\times V^\times) \setminus E$ so it is constant on the connected components of $(V\times V^\times) \setminus E$. In particular, for every $t\in V^\times$ the function $v\mapsto D_t(U,v)$ is constant on the connected components of $V\setminus E_t$. 

Assume that $D_{t_0}(U)\geq 1$ for some $t_0\in V^\times$. The set of values of $D_{t_0}(U,v)$ is discrete because $D_{t_0}(U,v)\in \mathbb N-|\Delta_k|^{-1/2}\Leb(U)\|t_0\|$. We deduce that there exists a connected component $Q_{t_0}$ of $V\setminus E_{t_0}$ such that $D_{t_0}(U,v)\geq 1$ or $D_{t_0}(U,v)\leq -1$ for all $v\in Q_{t_0}$. Assume that the first inequality holds. Fix a point $v_0\in Q_{t_0}$. Let $Q$ be the unique connected component of $(V\times V^\times)\setminus E$ containing $(v_0,t_0)$. For $\varepsilon>0$ let $Q^\varepsilon:=\left\{(v,t)\in Q| \|t\|<\|t_0\|-\varepsilon\right\}$, for $\varepsilon$ small enough it is a non-empty open set because $(v_0,t_0)$ lies in the interior of $Q$. Choose open sets $T\subseteq V^\times, W\subseteq V$ such that $W\times T\subseteq Q^\varepsilon$. For every $(v,t)\in Q
^\varepsilon$ we have 
\begin{align*} D_t(U,v)=N_t(U,v)-|\Delta_k|^{-1/2}\Leb(U)\|t\|=&N_{t_0}(U,v_0)-|\Delta_k|^{-1/2}\Leb(U)\|t\|\\>& D_{t_0}(U,v_0)+\varepsilon |\Delta_k|^{-1/2}\Leb(U).
\end{align*} We deduce that  $D_t(U,v)>1+\delta$ with $\delta=\varepsilon \Leb(U)|\Delta_k|^{-1/2}$ for $t\in T$ and $v\in W$. In the case $D_{t_0}(U,v)\leq -1$ the same argument works with $Q^{\varepsilon }=\left\{(v,t)\in Q| \|t\|>\|t_0\|+\varepsilon\right\}.$ 
\end{proof}

We show that if $\mu$ is a limit measure and $U$ is the open set provided by Proposition \ref{p-EnergyMinimizers}, then $U$ must have suspiciously low maximal discrepancy. 

\begin{lemma}\label{l-LimitDiscrepancy} Let $\mu$ be a limit measure on $V$ and let $U$ be a non-empty open bounded subset of $V$ such that $\p U$ is Jordan measurable of Jordan measure $0$ and $\mu=\Leb|_U$. Then $U$ satisfies $D_t(U)<1$ for all $t\in V^\times$.
\end{lemma} 
\begin{proof}
We argue by contradiction. Assume that for some $t_0\in V^\times$ we have $D_{t_0}(U)\geq 1$. By Lemma \ref{l-DiscrepancyCont}, there exist non-empty open sets $T\subseteq V^\times, W\subseteq V$ and $\delta>0$ such that $|D_t(U,v)|>1+\delta$ for every $t\in T, v\in W$. By making $W$ smaller if necessary we may assume that it is an open cylinder in $V$, similarly by taking smaller $T$ if necessary we may assume that there exists $\kappa>1$ such that $\kappa^{-1}\leq \|t\|\leq\kappa$ for all $t\in T$.  Let $(n_i)_{i\in \N}$ be a sequence of natural numbers, let $(\mathcal S_{n_i})_{i\in\N}$ be a sequence of $n_{i}$-optimal sets and let $(t_{n_i})_{i\in\N}\subset V, (s_{n_i})_{i\in \N}\subset V^\times, \|s_{n_i}\|=n_i|\Delta_k|^{1/2}$ be such that the measures $\mu_{n_i}$ defined as in (\ref{e-defLimit}) converge weakly-* to $\mu$. Translating $\mathcal S_{n_i}$ by appropriate elements of $\O_k$ we may assume that $t_{n_i}=0$. This will simplify considerably the formulas in the proof. By \cite[Corollary 2.4]{BFS2017} the sets $\mathcal S_{n_i}$ are \textbf{almost uniformly distributed} modulo powers of every prime ideal $\frak p$ of $\O_k$. This means that for every prime $\frak p$ of $\O_k$, $ l\in \N$ and $a\in \O_k$ we have 
$$\left|| \left\{x\in \mathcal S_{n_i}| x-a\in \frak p^l\right\}|-\frac{n_i+1}{N\frak p^l}\right|<1.$$
In order to get a contradiction we will exhibit prime powers $\frak p_{n_i}^l$ for all sufficiently large $n_i$ such that $\mathcal S_{n_i}$ fails to be almost uniformly equidistributed modulo $\frak p_{n_i}^l$. 

Let $E_{n_i}:=(s_{n_i}U)\cap \O_k$ and put $R_{n_i}=(\mathcal S_{n_i}\setminus E_{n_i})\cup (E_{n_i} \setminus \mathcal S_{n_i})$. Since $\mu_{n_i}$ converges weakly-* to ${\rm Leb}|_U$ and the boundary $\p U$ has Jordan measure $0$ we can deduce that $|R_{n_i}|=o(n_i).$ The set $T^{-1}$ is open so by Corollary \ref{c-AngularEq} (in the appendix) for $n_i$ sufficiently large there exists an $\varpi_{n_i}\in s_{n_i}T^{-1}\cap \O_k$ such that the principal ideal $\frak p_{n_i}^l:=\varpi_{n_i}\O_k$ is a prime power. For every $x\in \O_k\cap (-\varpi_{n_i} W)$ we have
$$\left|\left\{y\in E_{n_i}| x-y\in \frak p_{n_i}^{l}\right\}\right|=\left|(s_{n_i}U)\cap (\varpi_{n_i}\O_k+x)\right|=N_{s_{n_i}\varpi_{n_i}^{-1}}(U, -\varpi_{n_i}^{-1}x).$$ 
Since $s_{n_i}\varpi_{n_i}^{-1}\in T$,  $-\varpi_{n_i}^{-1}x\in W$ and $\|s_{n_i}\varpi_{n_i}^{-1}\|=n_i |\Delta_k|^{1/2}(N\frak p_{n_i}^l)^{-1}$ 
we get \begin{equation}\label{e-BadEqui1} \left| \left|\left\{y\in E_{n_i}| x-y\in \frak p_{n_i}^l\right\}\right|- \frac{n_i}{N\frak p_{n_i}^l}\right|=|D_{s_{n_i}\varpi_{n_i}^{-1}}(U, -\varpi_{n_i}^{-1}x)|>1+\delta \textrm{ for all } x\in -\varpi_{n_i}W.
\end{equation}
We showed that for $n_{i}$ big enough $E_{n_i}$ fails "badly" to be almost uniformly equidistributed modulo $\frak p_{n_i}^l$. From this we want to deduce that $\mathcal S_{n_i}$ is not almost uniformly equidistributed modulo $\frak p_{n_i}^l$. Call $x\in (-\varpi_{n_i}W)\cap \O_k$ \textbf{bad} if $(x+\varpi_{n_i}\O_k)\cap R_{n_i}\neq \emptyset$ and \textbf{good} otherwise. For good points we have $(x+\frak p_{n_i}^l)\cap E_{n_i}=(x+\frak p_{n_i}^l)\cap \mathcal S_{n_i}.$  Our next goal is to prove that for $n_i$ sufficiently large there exists at least one \footnote{ In fact we will show that most of them are good.} good element in $-\varpi_{n_i}W$. Let us estimate the number of bad elements of $(-\varpi_{n_i}W)\cap \O_k$. By Lemma \ref{l-Cylinder}, for every $r\in R_{n_i}$ we have $|(\varpi_{n_i}\O_k-r)\cap(-\varpi_{n_i}W)|=|(r\varpi_{n_i}^{-1}-W)\cap \O_k|=O(1).$ Hence we have at most $O(|R_{n_i}|)=o(n_i)$ bad elements. On the other hand  $|-\varpi_{n_i}W\cap \O_k|=\|\varpi_{n_i}\|\Leb(W)|\Delta_k|^{-1/2}+ o(\|\varpi_{n_i}\|).$ We chose $\varpi_{n_i}\in s_{n_i}T^{-1}$ so that $\kappa^{-1} n_i|\Delta_k|^{1/2}\leq \|\varpi_{n_i}\|\leq \kappa n_i|\Delta_k|^{1/2}$ therefore the number of good elements is at least $\Leb(W)n_i\kappa^{-1}-o(n_i)$. We infer that for $n_i$ sufficiently large there exists at least one good element $x\in (-\varpi_{n_i}W)\cap \O_k$. Let $x\in (-\varpi_{n_i}W)\cap \O_k$ be such a good element. We have 
$E_{n_i}\cap (x+\frak p_{n_i}^l)=\mathcal S_{n_i}\cap(x+\frak p_{n_i}^l)$ so by (\ref{e-BadEqui1}) we get 
\begin{equation}\label{e-BedEqui2}
\left| \left|\left\{y\in \mathcal S_{n_i}| x-y\in \frak p_{n_i}^l\right\}\right|- \frac{n_i+1}{N\frak p_{n_i}^l}\right|=\left|D_{s_{n_i}\varpi_{n_i}^{-1}}(U, -\varpi_{n_i}^{-1}x)-\frac{1}{N\frak p_{n_i}^l}\right|>1+\delta-\frac{1}{N\frak p_{n_i}^l}.
\end{equation}
We know that $N\frak p_{n_i}^l=\|\varpi_{n_i}\|\geq \kappa^{-1}|\Delta_k|^{1/2}n_i$ so for $n_i$ sufficiently large (\ref{e-BedEqui2}) implies that $\mathcal S_{n_i}$ is not almost uniformly equidistributed modulo $\frak p_{n_i}^{l}$. This is a contradiction because $n$-optimal sets are almost uniformly equidistributed modulo all prime powers.
\end{proof} On the other hand we have the following lower bound on the discrepancy.
\begin{lemma}\label{l-NoAEq} Assume that $V= \R^{r_1}\times \C^{r_2}$ with $r_1+2 r_2>1$.  Let $U$ be an open bounded subset of $V$  such that $\p U\cap V^\times$ is a submanifold of $V^\times$ of class $C^1$ and $\lambda\overline U\subset U$ for every $0\leq \lambda<1$. Then there exists $t\in V^\times$ such that $D_t(U)>1$.
\end{lemma}
\begin{proof}
We claim that there exists $t_0\in V^\times, v_0\in V$ such that $(t_0\p U+v_0)\cap \O_k$ contains at least $3$ points\footnote{ This is of course not true if $V=\R$ and $U$ is an interval.}.  Let $N=\dim_\R V$. For every $v\in V^\times$ let us identify the tangent space $T_v V^\times$ with $V$ in the obvious way. For every point $p\in \p U\cap V^\times$ the tangent space $T_p \p U$ is a codimension $1$ subspace of $V$. Call a (real) codimension $1$ subspace $H$ of $V$ singular if we have $H\subset V\setminus V^\times$. 

\textbf{ Step 1.} We show that there is $x_0\in \p U\cap V^\times$ such that $T_{x_0} \p U$ is not singular.
Write ${\rm Gr}_{N-1}(V)$ for the space parametrizing all $(N-1)$-dimensional real vector subspaces of $V$. The map $\phi(p):= [T_p \p U]\in {\rm Gr}_{N-1}(T_p V^\times)\simeq {\rm Gr}_{N-1}(V)$ is continuous on $\p U\cap V^\times$ because the latter is a $C^1$-submanifold. Let $M$ be any connected component of $\p U\cap V^\times$. Either there exists a point $p\in M$ such that $T_p \p U$ is nonsingular or we can assume that the image  $\phi(M)$ consists solely of singular subspaces. First assume that the latter is true. The set of singular subspaces in ${\rm Gr}_{N-1}(V)$ has $r_1$ elements, each corresponding to a real coordinate of $V$. Hence $\phi(M)=H$ for some singular subspace $H$. We deduce that $M$ is contained in a hyperplane $H'$ parallel to $H$. In particular the $r_1$ singular subspaces and $H'$ cut out a connected component of $U\cap V^\times$ which is a bounded region of $V$. This is a contradiction because $r_1+1\leq N+1$ and the only way $N+1$ (real) codimension $1$ hyperplanes can cut out a bounded region in $N$-dimensional space is when they are pairwise non-parallel. Therefore there exists $x_{0}\in M$ such that $T_{x_{0}} \p U$ is nonsingular. 

\textbf{ Step 2.} Let $x_{0}$ be as in Step 1. We construct a continuous map $\gamma\colon [0,1]\to \p U$ such that $\gamma(0)=x_0$ and $\|\gamma(s)-\gamma(0)\|>0$ for $0<s\leq 1$.
Choose any smooth complete Riemannian metric on $\p U\cap V^\times$. Choose a vector $w\in T_{x_0}\p U$ such that $\|w\|\neq 0$. Let $\gamma: [0,+\infty)\to \p U \cap V^{\times }$ be the unique geodesic ray such that $\gamma(0)=x_0$ and $\frac{d}{dt}\gamma(t)|_{t=0}=w$. We have $\gamma (t)= \gamma (0)+tw+tr(t)$ where $r(t)\to 0$ when $t \to 0$. Therefore $\|\gamma (t) - \gamma (0) \|=\prod _{i=1}^{r_{1}}|tw_{i}+tr(t)_{i}|\prod _{i=r_{1}+1}^{d}|tw_{i}+tr(t)_{i}|^{2}=t^{n} \|w\| +O(t^{n} \textrm{max}_{i=1, \ldots , d} | r(t)|_{i})$ and for $t$ small enough we have $\|\gamma(s)-\gamma(0)\|>0$ for every $s\leq t$. Up to reparametrizing  $\gamma$ we may assume $t=1$.

\textbf{ Step 3.} We show that there exists $s_1>0$ such that $((\gamma(s_1)-x_0)\O_k+ x_0)\cap \p U$ contains at least one point $x_1$ other than $x_0, \gamma(s_1)$. First note that as $s$ approaches $0$ the norm $\|\gamma(s)-x_0\|$ tends to $0$. Hence $$\lim_{s\to 0}|(\gamma(s)-x_0)\O_k+ x_0)\cap U|=+\infty.$$ Let $s_1:=\inf\left\{s>0| |(\gamma(s)-x_0)\O_k+ x_0)\cap U|\leq |(\gamma(1)-x_0)\O_k+ x_0)\cap U|\right\}.$  The equality above ensures that $s_1>0$. The intersection $((\gamma(s_1)-x_0)\O_k+ x_0)\cap \p U$ must contain another point except $x_0$ and $\gamma(s_1)$ because otherwise the function $s\mapsto |((\gamma(s)-x_0)\O_k+ x_0)\cap U|$ would be constant in an open neighborhood of $s_1$, contradicting the definition of $s_1$.

\textbf{ Step 4.} Put $v_0=-x_0(\gamma(s_1)-x_0)^{-1}$ and  $t_0= (\gamma(s_1)-x_0)^{-1}$. Then $(t_0\p U+v_0)\cap \O_k$ contains at least $3$ distinct points $p_1,p_2,p_3$. Indeed, we may take $p_1=0, p_2=1$ and $p_3=(x_1-x_0)(\gamma(s_1)-x_0)^{-1}$ where $s_1,x_1$ are provided by Step 3.

\textbf{ Step 5.} We will show that for every small enough $\varepsilon>0$ there exists an open neighborhood $W$ of $v_0$  such that $((1-\varepsilon)t_0 U+v_1)\cap \O_k=(t_0 U+v_0)\cap \O_k$ and  $((1+\varepsilon)t_0 U+v_1)\cap \O_k\supset ((t_0 U+v_0)\cap \O_k)\sqcup \left\{p_1,p_2,p_3\right\}$ for every $v_1\in W$. 

Choose $\varepsilon>0$ such that $((1-\varepsilon)t_0 U+v_0)\cap \O_k=(t_0 U+v_0)\cap \O_k$ and $((1+\varepsilon)t_0 U + v_0)\cap \O_k \supset ((t_0 U+v_0)\cap \O_k)\sqcup \left\{p_1,p_2,p_3\right\}$.  The desired conditions are satisfied once $\varepsilon$ is small enough because $(1-\varepsilon)t_0 \overline U\subset t_0 U$ and $t_0\overline{U}\subset (1+\varepsilon)t_0U$. Conditions $((1-\varepsilon)t_0 U+v_1)\cap \O_k=(t_0 U+v_0)\cap \O_k$ and  $((1+\varepsilon)t_0 U+v_1)\cap \O_k\supset ((t_0 U+v_0)\cap \O_k)\sqcup \left\{p_1,p_2,p_3\right\}$ hold for $v_1$ in an open neighborhood of $v_0$.

\textbf{ Step 6.} We show that for small enough $\varepsilon>0$  we have either $D_{(1-\varepsilon)t_0}(U)>1$ or $D_{(1+\varepsilon)t_0}(U)>1$. By Step 5 for every $v_1\in W$ we have $N_{(1+\varepsilon)t_0}(U,v_1)-N_{(1-\varepsilon)t_0}(U,v_1)\geq 3$ so $D_{(1+\varepsilon)t_0}(U,v_1)-D_{(1-\varepsilon)t_0}(U,v_1)\geq 3- |\Delta_k|^{-1/2}\Leb(U)\|t_0\|((1+\varepsilon)^N-(1-\varepsilon)^N)$. By choosing $\varepsilon$ small enough we can ensure that $D_{(1+\varepsilon)t_0}(U,v_1)-D_{(1-\varepsilon)t_0}(U,v_1)\geq \frac{5}{2}$. The set $W$ is open so it has positive measure. We deduce that  $D_{(1-\varepsilon)t_0}(U)+D_{(1+\varepsilon)t_0}(U)\geq \frac{5}{2}$. One of them must be bigger than $1$ so Step 6 and the lemma follows.
\end{proof}
\begin{remark} Had we known that the boundary of $U$ is $C^4$, a much stronger lower bound on the discrepancy $D_t(U)$ could be extracted from classical results on lattice point counting \cite{HuxleyBook}. 
\end{remark}
\subsection{Proof of the main theorem}\label{s-Proof}
 In this section we prove the main result.
\begin{proof}[Proof of Theorem \ref{mt-noptimal}]
We argue by contradiction. As before $V=k\otimes_\Q \R=\R^{r_1}\times \C^{r_2}.$ Assume that there is a sequence $(n_i)_{i\in \N}\subseteq \N$ with $n_i\to\infty$ such that for every $i\in\N$, there exists an $n_i$-optimal set $\mathcal S_{n_i}$. By Corollary \ref{c-CompactFrame} there exists a compact cylinder $\Omega\subset V$ and sequences $(s_{n_i})_{i\in\N},(t_{n_i})_{i\in \N}\subset V$ such that $\|s_{n_i}\|=n_i|\Delta_k|^{1/2}$ and $s_{n_i}^{-1}(\mathcal{S}_{n_i}-t_{n_i})\subset \Omega$.  Put \begin{equation}\label{e-defLimit2}
\mu_{n_i}:=\frac{1}{n_i}\sum_{x\in \mathcal{S}_{n_i}} \delta_{s_{n_i}^{-1}(x-t_{n_i})}.
\end{equation}
Those measures are supported in $\Omega$. Since $\Omega$ is compact we may assume, passing to a subsequence if needed, that $\mu_{n_i}$ converges weakly-* to a probability measure $\mu$. This is a measure that we called in Section \ref{s-LimitMeasures} a limit measure. By Lemma \ref{l-DensityLimitMeasure}, the measure $\mu$ is absolutely continuous with respect to the Lebesgue measure on $V$ with the density bounded by $1$.  By Proposition \ref{p-LimitEnergy} we have $I(\mu)= -\frac{1}{2}\log|\Delta_k|-\frac{3}{2}-\gamma_k+\gamma_\Q$ where $\gamma_k,\gamma_\Q$ are Euler--Kronecker constants of $k,\Q$ respectively. Recall that $\mathcal P^1(V)$ is the set of probability measures on $V$ absolutely continuous with respect to the Lebesgue measure and of density at most $1$. By Lemma \ref{l-LowerEnergyBound}, the measure $\mu$ realizes the minimal energy among all compactly supported probability measures in $\mathcal P^1(V)$. Therefore by Proposition \ref{p-EnergyMinimizers} there exists an open set $U$ of Lebesgue measure $1$ such that $\p U\cap V^\times$ is a $C^1$-submanifold of $V^\times$ , $\lambda \overline U\subset U$ for $0<\lambda<1$ and up to translation $\mu=\Leb|_U$.  By Lemma \ref{l-NoAEq} applied to $U$, there exists $t\in V^\times$ such that $D_t(U)>1$. On the other hand  Lemma \ref{l-LimitDiscrepancy} yields  $D_t(U)<1$ for every $t\in V^\times$. This yields the desired contradiction.
\end{proof}

\section{Appendix} 
\subsection{Measure theory}\label{s-AppMeasures}
\begin{lemma}\label{l-AccumMeasures} Let $\nu$ be a probability measure on $V$ of density at most $1$. Then there exists a sequence of subsets $(E_{n})_{n\in \N}$ of $\O_k$ such that $|E_{n}|=n+1$ and the sequence of measures 
\begin{equation}\label{e-DensityEq1}\nu_{E_n,n}:=\frac{1}{n}\sum_{x\in E_{n}}\delta_{n^{-1/N}|\Delta_k|^{-1/2N} x}\end{equation} converges weakly-* to $\nu$.
\end{lemma}
\begin{proof}
The proof is based on a sequence of reductions to easier problems. Since the set of compactly supported probability measures on $V$ is dense among all probability measures we can, without loss of generality, assume that $\nu $ is compactly supported. First note that if we manage to find a sequence of sets $E_{n}\subset \O_k$ such that the measures $\nu_{E_n,n}$ converge weakly-* to $\nu$ and $n^{-1/N}E_n$ are all contained in a bounded set, then $|E_{n}|=n+o(n)$.  Removing or adding $o(n)$ points to each $E_{n}$ does not affect the weak-* limit. The proof is reduced to finding any sequence $(E_{n})$ of finite subsets of $\O_k$ such that $\nu_{E_n,n}$ converges weakly-* to $\nu$ and $n^{-1/N}E_n$ is uniformly bounded. Let $P\subseteq\mathcal M^1(V)$ be the set of  finite measures for which this is possible. 

\textbf{Step 1.} We prove that $P$ is a closed convex subset of $\mathcal M^1(V)$. The fact that $P$ is closed is immediate by definition. Thus, to prove that it is convex we only need to show that for every $\nu,\nu'\in P$ we have $\frac{1}{2}(\nu+\nu')\in P$. Fix a set $a_1,\ldots, a_{2^N}$ of representatives of $\O_k/2\O_k$. Let $E_{n},E'_{n}$ be sequences of subsets of $\O_k$ such that $\nu_{E_n,n},\nu_{E'_n,n}$ converge weakly-* to $\nu,\nu'$ respectively and $n^{-1/N}E_{n}, n^{-1/N}E'_{n}$ are uniformly bounded. Define $$F_{2^Nn}=\bigcup_{i=1}^{2^{N-1}}(a_i+2E_n)\cup \bigcup_{i=2^{N-1}+1}^{2^{N}}(a_i+2E'_n)$$ and $F_{m}:=F_{2^N[m/2^N]}$. A simple computation shows that $\lim_{m\to\infty} \nu_{F_m,m}=\frac{1}{2}(\nu+\nu')$ so the latter belongs to $P$.

\textbf{Step 2.} Let $U\subset V$ be a bounded open set of finite Lebesgue measure such that $\p U$ is Jordan measurable and has Jordan measure $0$. Then the measure $\nu(A):=\Leb(A\cap U)$ belongs to $P$. Indeed it is enough to take $E_n=\O_k\cap n^{1/N}U.$  

\textbf{Step 3.} For every bounded measurable set $E\subset V$ of finite measure the measure $\nu_E(A):=\Leb(A\cap E)$ is in $P$. This follows from the fact that the Lebesgue measure is Radon so there exists a sequence of open sets $U_n$ containing $E$ such that $\nu_{U_n}$ converges weakly-* to $\nu_E$. Removing from $U_n$ a closed set of arbitrarily small Lebesgue measure we can assume that $\p U_n$ has Jordan measure $0$ so Step 2 applies.

\textbf{Step 4.} The convex hull of measures $\nu_E$ from the previous step is weakly-* dense in the set of measures of density at most $1$ and bounded support. Indeed, let $\nu$ such a measure with density $f\in L^1(V)$ such that $f(v)\leq 1$ almost everywhere. For every $t\in [0,1]$ let $E_t=\left\{v\in V| f(v)\geq t\right\}$. Those are bounded measurable sets of finite measure and we have 
$\nu=\int_0^1 \nu_{E_t}dt$. Hence, by Step 1 $\nu\in P$. The lemma is proven.
\end{proof}
\subsection{Angular distribution of prime ideals}
We prove a version of prime number theorem for number fields where principal ideals are weighted with respect to their "angular" position in $V^\times/ \O_k^\times$. This is very close to the prime number theorem for products of cylinders and sectors proved by  Mitsui \cite{Mitsui56}. The version we need is a little bit different and we don't need an explicit error term. The following result is rather folklore, we include a short proof for completeness.
\begin{lemma}\label{l-AngEqC}
Let $k$ be a number field and let $V=k\otimes_{\Q}\R$. Let $\varphi: V^\times \to \C$ be a continuous function such that $\varphi(t \lambda x)=\varphi(x)$ for every $x\in V^\times, \lambda\in\O_k^\times$ and  $t\in \R^\times$. For a principal ideal $I=a\O_k$ we put $\varphi(I):=\varphi(a)$. Then for $X>0$ 
$$\sum_{\substack{N(\frak p^l)\leq X\\ \frak p \textrm{ principal}}} \varphi(\frak p)\log N\frak p=\frac{X}{R_k h_k}\int_{\mathcal I/\O_k^\times}\varphi(t)dt + o(X),$$ where $R_k,h_k$ are the regulator and the class number of $k$ and $\mathcal I:=\left\{v\in V|\, \|v\|=1\right\}.$
\end{lemma}
\begin{proof}
Write $\mathcal A$ for the space of continuous functions $\varphi\colon V^\times \to \mathbb C$ invariant by $\R^\times$ and $\O_k^\times$. The unitary characters $\chi: V^\times\to \C^\times$ such that $\chi(\lambda)=1$ for every $\lambda\in\O_k^\times$ and $\chi(t)=1$ for every $t\in \R^\times$ span a dense subspace of $\mathcal A$. As a consequence it is enough to prove the statement for unitary characters $\chi$ as above. 

Our first step is to associate to $\chi$ a Hecke character. Write $\mathbb A^\times$ for the group of ideles of $k$ and $\mathbb A_\infty^\times$ and $\mathbb A_f^\times$ for the groups of infinite and finite ideles respectively. We distinguish the subgroup $\mathbb A^1$ of ideles of idelic norm $1$. Let $K=\prod_{\frak p}\O_{k_\frak p}^\times$ be the maximal compact subgroup of $\A_f^\times$. We identify $V^
\times$ with $\mathbb A_\infty^\times$. By abuse of notation let us write $\chi$ for the extension of $\chi$ to $\mathbb A_\infty^\times \times K\subseteq \mathbb A^\times$ by setting $\chi(vx)=\chi(v)$ for $v\in \mathbb A_\infty^\times,x\in K$. The character $\chi$ 
factors through $(\mathbb A_\infty^\times \times K)/\O_k^{\times}$ and the latter is a closed subgroup of $\mathbb A^\times/k^\times$ of index $h_k$. Let $\widehat\chi$ be 
any extension of $\chi$ to $\mathbb A^\times/k^\times$. There are precisely $h_k$ such extensions and they are all of the form $\psi\widehat\chi$ for $\psi: \mathbb A^{\times}  /(\mathbb A_\infty^\times \times K) k^\times=: {\rm Cl}_k\to \C^\times$ where ${\rm Cl}_k$ stands for the class group of $k$. Through the standard procedure, $\widehat \chi$ gives rise to an unramified Hecke character $\widehat \chi$ such that for every principal prime ideal $\frak p=a\O_k$ we have $\widehat\chi(\frak p)=\chi(a)$ (the character on the right is the original $\chi\colon V^\times \to\mathbb C^\times$). 
For any $\psi: {\rm Cl}_k\to \C^\times$, consider the Hecke L-function $$L(s,\psi\widehat\chi):=\prod_{\frak p}\left(1-\frac{\psi(\frak p)\widehat\chi(\frak p)}{(N\frak p)^s}\right)^{-1}.$$
 By \cite[Theorem 5.34]{IwaniecKowalski2004} there exist constants $c,b>0$ such that the function $L(s,\psi\widehat\chi)$ has at most one zero in the region $$\Re\, z>1-\frac{c}{b\log|\Delta_k|(|\Im\, z| +3)^b}.$$ The exceptional zero is always real, less than $1$ and can occur only when $\psi\widehat\chi$ is a real character. With this information at hand, the standard argument used to prove the prime number theorem  (see \cite[Theorem 5.13]{IwaniecKowalski2004}) shows that  for any $X>0$ we have
\begin{equation}\label{e-AngularE1}\sum_{N\frak p^l\leq X} \psi(\frak p)\widehat\chi(\frak p)\log N\frak p =rX+ o(X),
\end{equation} where $r=1$ if $L(s,\psi\widehat\chi)$ has a simple pole at $1$ and $r=0$ otherwise. In our case $r=1$ if $\psi\widehat\chi=1$ and $r=0$ otherwise.
We take the average of (\ref{e-AngularE1}) over all characters $\psi:{\rm Cl}_k\to \C^\times$ to get 
$$\sum_{\substack{N\frak p^l\leq X \\ \frak p^l \textrm{ principal}}} \widehat\chi(\frak p)\log N\frak p =\frac{1}{h_k}\sum_{N\frak p^l\leq X} \sum_{\psi\in \widehat{{\rm Cl}_k}}\psi(\frak p)\widehat\chi(\frak p)\log N\frak p=\begin{cases}\frac{X}{h_k} +o(X)& \textrm{ if } \chi=1\\ o(X) &\textrm{ otherwise.}\end{cases}.$$
Since $\int_{\mathcal I/\O_k^\times}1dt=R_k$ this agrees with the formula predicted by the lemma. We deduce that the lemma holds for $\chi$, which is enough to deduce the general case.
\end{proof}
\begin{corollary}\label{c-AngularEq} Let $U$ be a bounded open subset of $V^\times$. Then for any $t\in V^\times$ with $\|t\|$ sufficiently large there is at least one element $x\in tU\cap \O_k$ such that $x\O_k$ is a prime ideal power.
\end{corollary}
\begin{proof} First we prove the statement for $t\in \R^\times\subset V^\times$.
Let $U'$ be a small cylinder contained in $U$ such that $\overline U'\subseteq U$. Put $U''=\O_k^\times \R^\times U'$. Let $ \varphi _{0}: V \to \R _{ \geq 0}$ be any continuous function such that $ \varphi _{0} | _{U'} =1$ and $ \varphi _{0}$ vanishes outside $U$. For $v \in V^{ \times }$ put $ \varphi (v) := \sum _{x \in \O _{k} ^{\times }} \int _{0}^{ \infty } \varphi _{0} (txv) \frac{dt}{t}$. Function $ \varphi $ is continuous, positive on $U$, supported on $U''$ and it satisfies the assumptions of Lemma \ref{l-AngEqC}. As $V^\times/\R^\times\O_k^\times$ is compact the function $\varphi$ is necessarily bounded. There exist $0<a<b$ such that $U''\cap\left\{v\in V^\times| a\leq \|v\|\leq b\right\}\subseteq \O_k^\times U$. By Lemma \ref{l-AngEqC}, there is a positive constant $c$ such that
\begin{equation}\label{e-AngEqCol}\sum_{\substack{a|t|^N\leq N(\frak p^l)\leq b|t|^N\\ \frak p \textrm{ principal}}} \varphi(\frak p)\log N\frak p=c(b-a)|t|^N+ o(b|t|^N).\end{equation}
We deduce that for $|t|^N=\|t\|$ sufficiently large there exists an element $w\in t\O_k^\times U\cap \O_k$ such that $w\O_k$ is a prime power. We replace  $w$ by $w\lambda$ for some $\lambda\in\O_k^\times$ to get an element of $tU\cap \O_k$ generating a prime power. 

To get the general case choose an open set $W\subseteq V^\times$ and a finite set $y_1,\ldots,y_m$ of elements of $V^\times$ such that for every translate $tU, t\in V^\times$ there exists an $\lambda\in \O_k^\times$, $t_0\in \R^\times$ and $i\in\left\{1,\ldots,m\right\}$ such that $\lambda t_0 y_i W\subseteq t U$. This can be always arranged because $V^\times/ \R^\times \O_k^\times$ is compact. The case of the corollary that we have already proved applied to the open sets $y_iW$ implies that for $\|t_0\|$ sufficiently large the sets $t_0 y_i W$ all contain an element generating a prime power. But then so do the translates $\lambda t_0 y_i W$ for every $\lambda\in \O_k^\times$. Since one of them is contained in $tU$ and $\|t_0\|\to\infty$ as soon as $\|t\|\to\infty$  the corollary is proven.\end{proof}
\bibliography{ref} 
\bibliographystyle{plain} 
\end{document}